\documentclass[10pt,twoside]{amsart}
\usepackage{amsmath,amssymb}
\usepackage{amsfonts}
\usepackage{amsthm,amscd}
\usepackage{amstext}
\usepackage{color}
\usepackage{latexsym}
\usepackage{amscd,graphics}

\begin{document}
\newcommand{\comment}[1]{\marginpar{\footnotesize #1}} 

\newtheorem{proposition}{Proposition}[section]
\newtheorem{lemma}[proposition]{Lemma}
\newtheorem{sublemma}[proposition]{Sublemma}
\newtheorem{theorem}[proposition]{Theorem}

\newtheorem{maintheorem}{Main Theorem}
\newtheorem{corollary}[proposition]{Corollary}

\newtheorem{ex}[proposition]{Example}

\theoremstyle{remark}

\newtheorem{remark}[proposition]{Remark}

\theoremstyle{definition}
\newtheorem{definition}[proposition]{Definition}
\def\real{\mathbb{R}}
\def\integer{\mathbb{Z}}
\def\complex{\mathbb{C}}
\def\supp{\mathrm{supp}}
\def\var{\mathrm{var}}
\def\sgn{\mathrm{sgn}}
\def\sp{\mathrm{sp}}
\def\id{\mathrm{id}}
\def\Imm{\mathrm{Image}}
\def\cc{\Subset}
\def\cst{\mathrm{cst}}
\def\const{\mathrm{const}}
\def\dist{\mathrm{dist}}
\def\Lip{\mathrm{Lip}}
\def\BB{\mathcal{B}}
\def\CC{\mathcal{C}}
\def\DD{\mathcal{D}}
\def\EE{\mathcal{E}}
\def\FF{\mathcal{F}}
\def\GG{\mathcal{G}}
\def\HH{\mathcal{H}}
\def\II{\mathcal{I}}
\def\JJ{\mathcal{J}}
\def\KK{\mathcal{K}}
\def\LL{\mathcal{L}}
\def\LLL{\mathbb{L}}
\def\MM{\mathcal{M}}
\def\NN{\mathcal{N}}
\def \OO{\mathcal {O}}
\def \PP{\mathcal {P}}
\def \QQ{\mathcal {Q}}
\def \RR{\mathcal {R}}
\def\rr{r}
\def\SS{\mathcal{S}}
\def\TT{\mathcal{T}}
\def\XX{\mathcal{X}}
\def\YY{\mathcal{Y}}
\def\ZZ{\mathcal{Z}}
\def\FFF{\mathbb{F}}
\def\PPP{\mathbb{P}}
\def\sign{\operatorname{sign}}

\title[Whitney--H\"older continuity of the SRB measure]
{Whitney--H\"older continuity of the SRB measure for transversal families of smooth unimodal maps}
\author{Viviane Baladi, Michael Benedicks, and Daniel Schnellmann}
\address{Mathematical Sciences, Copenhagen University, 2100 Copenhagen, Denmark}
\email{baladi@math.ku.dk}

\address{KTH,
Department of Mathematics, S-100 44  Stockholm, Sweden}
\email{michaelb@kth.se}

\address{DMA, UMR 8553, \'Ecole Normale Sup\'erieure,  75005 Paris, France}
\email{daniel.schnellmann@ens.fr}
\date{\today}

\begin{abstract}
We consider $C^2$ families $t \mapsto f_t$ of $C^4$ nondegenerate unimodal maps. We study
the absolutely continuous invariant probability (SRB)
measure $\mu_t$  of $f_t$, as a function of $t$ on the set
of Collet--Eckmann (CE) parameters:

{\it Upper bounds:}  Assuming existence of a transversal
CE parameter, we find a positive measure set of
CE parameters $\Delta$, and, for each $t_0\in \Delta$, a
set $\Delta_0\subset \Delta$ of polynomially recurrent parameters containing
$t_0$ as a Lebesgue density point, and 
constants $C\ge 1$, $\Gamma >4$, so that, for every $1/2$-H\"older function $A$,
$$
 \Big|\int A\,  d\mu_t -\int A\,  d\mu_{t_0}\Big|
\le C  \|A\|_{C^{1/2}}|t-t_0|^{1/2}| \log |t-t_0||^\Gamma\, ,
\, \, \forall t \in \Delta_0\, .
$$ 
In addition, for all $t\in \Delta_0$,
the renormalisation period $P_t$ of $f_t$  satisfies
$P_t\le P_{t_0}$, and there are uniform bounds on the
rates of mixing of $f_t^{P_t}$ for all $t$ with $P_t=P_{t_0}$.
If $f_t(x)=tx(1-x)$, the set $\Delta$ contains almost all CE parameters.

{\it Lower bounds:} Assuming existence of a transversal mixing
Misiurewicz--Thurston parameter $t_0$, we find a  set of
CE parameters $\Delta'_{MT}$ accumulating at $t_0$, a constant $C\ge 1$,  and a
 $C^\infty$ function $A_0$, so that
$$
C |t-t_0|^{1/2}\ge\Big|\int A_0\,  d\mu_t -\int A_0\,  d\mu_{t_0}\Big|
\ge C^{-1}  |t-t_0|^{1/2}\, ,\, \, \forall t \in \Delta'_{MT}\, .
$$ 
\end{abstract}
\thanks{We are grateful to Lennart Carleson for encouraging us to work
on this question during a semester at Institut Mittag-Leffler 
(Djursholm, Sweden) in 2010. Warm thanks to Neil Dobbs and Mike Todd for explaining
their work  on discontinuity of the
SRB measure for unbounded goodness constants and to Gerhard Keller and
Carlangelo Liverani
for explaining Remark ~5 in \cite{kellerliverani}. A large part of this work was accomplished while
VB was a member of  UMR 8553 of CNRS, DMA, ENS, Paris, and during MB's one month visiting
professorship there in 2011. DS is grateful to the Mathematics Department in Copenhagen for
its hospitality in December 2012, when the paper was essentially completed.}
\maketitle

\section{Introduction and statement of results}

Let $f: X \to X$ preserve an ergodic invariant
probability measure $\mu$ which is absolutely continuous with respect to
Lebesgue. Then Birkhoff's theorem implies that there is a positive Lebesgue measure set of points $x$ 
for which the time averages of iterated Dirac masses
$\frac{1}{n}\sum_{k=0}^{n-1} \delta_{f^k(x)}$ converge (in the weak-$*$ topology)
to $\mu$. A measure  $\mu$ with this last property is also
called an SRB measure. All SRB measures studied in the present  paper
are absolutely continuous, but there exist SRB measures
which are not absolutely continuous,  in particular those
\cite{youngsrb} constructed
by Sinai, Ruelle, and Bowen for  smooth hyperbolic systems such as Anosov diffeomorphisms.

We are interested in differentiable
one-parameter families $t \mapsto f_t$, $t\in\real$ of
differentiable dynamical systems where $f_t$ admits a unique SRB measure for 
a positive measure set of parameters $t$. 
In the case where each $f_t$ is a smooth transitive Anosov diffeomorphism,  Ruelle 
\cite{rsbr, rsbr'} (see also \cite{Kat}) showed that $t\mapsto R_A(t):=\int A d\mu_t$ is differentiable
if $A$ is a smooth enough observable, and he obtained a formula
(the {\it linear response formula}) for $\partial_t R_A(t)$.
Since $\mu_t$ can be obtained as a fixed point for a Ruelle--Perron--Frobenius transfer operator with
a spectral gap, this formula can be proved via 
perturbation theory on a suitable Banach space. This result
led to the hope that linear response would
hold in other dynamical situations where the SRB measure is related to
a transfer operator with good spectral properties \cite{ruelle0, ruelle00, redherring}.
This hope was shattered when it was discovered (\cite{baladi, BS1},
see also \cite{mazzo}) that linear response does
not always hold in the simple situation of unimodal piecewise expanding interval
maps. (Contrarily to Anosov maps, piecewise expanding  maps are not structurally
stable.) More precisely, $t \mapsto R_A(t)$ is differentiable at $t=0$ if and only
(\cite[Thm 7.1]{BS1} see also the remarks after Theorem~\ref{t.main3} below)
if the family $f_t$ is {\it horizontal,}
that is, tangent to the topological class of $f_0$ at $t=0$. Horizontality is
an explicit codimension-one condition
on the vector field $\partial_t f_t$ \cite{BS2}. In the {\it transversal}
(non-horizontal)
case, Theorem~7.1 of \cite{BS1} (see also corrigendum) shows that 
the $|t \log |t||$ modulus of continuity which had been discovered long before
by Keller \cite{Ke} is
in fact optimal. (See also the discussion about parameters just after Theorem~\ref{t.main3}.)

Piecewise expanding maps can be viewed as a toy model for the more
difficult case of smooth unimodal maps: While all unimodal piecewise expanding
maps admit a unique absolutely continuous invariant probability measure, this 
does not hold in the smooth case, where vanishing of the
derivative at the critical point means that hyperbolicity is not guaranteed
(and occurs at best 
nonuniformly). 
The celebrated Collet--Eckmann condition (see \eqref{00} below)
implies  existence
 of a (unique) absolutely continuous invariant probability measure for 
smooth unimodal maps, with exponential decay of correlations in the mixing case. Indeed, one can
then construct the SRB measure as a fixed point of
a transfer operator, via a suitable tower construction
(there are several such constructions; see, e.g., \cite{young92}, \cite{keno}, \cite{BV}).
A generic family of smooth unimodal maps $f_t$ is transversal (or non-horizontal; see \eqref{eq.transversal} below
for a definition of transversality and \cite{tsujii}, \cite{ALM}, \cite{Levin}, \cite{BS}, \cite{GS}
for previous occurrences in the literature). In a transversal family,
the set of Collet--Eckmann parameters has positive measure, but does not contain
any intervals.
One natural question for families of smooth unimodal
maps is then to study the regularity
of $t\mapsto R_A(t)$, restricting $t$ to subsets of the  Collet--Eckmann parameters. 
We consider only the nondegenerate case, where
$f'_t(c)=0$ and $f''_t(c)<0$ at the critical point $c$. This includes the famous
logistic (or quadratic) family $f_t=t x(1-x)$.
Continuity of $t\to  R_A(t)$
was proved  on a
subset of ``good" Collet--Eckmann parameters $t$ by Tsujii
\cite{tsujiicont} and Rychlik--Sorets \cite{RySo}
in the 90's.  Parameters in this subset enjoy not only qualitative slow recurrence
ensuring the Collet--Eckmann property, but also quantitative control on the various
relevant constants.  (See 
Definition~\ref{good} for the notion of goodness used in the present paper.)

Quantified  goodness is indeed
necessary to ensure continuity, as we explain next:
A parameter $t$ is called
superstable if the critical point is periodic.
For the quadratic family, e.g.,
Thunberg proved \cite[Theorem C]{thun}
that there are superstable parameters $s_{n}$ of periods $p_n$, with
$s_{n}\to t$, where $t$ is a good Collet--Eckmann parameter,  
so that $\nu_{s_{n}}\to \nu$, where $\nu_{s_{n}}=\frac{1}{p_n}\sum_{k=0}^{p_n-1}
\delta_{f^k_{s_{n}}(c)}$,
and $\nu$ is the sum of atoms on a
repelling periodic orbit of $f_t$. Other sequences  $t_{n}\to t$ of superstable
parameters have the property that $\nu_{t_{n}}\to \mu_t$, the absolutely
continuous invariant measure of $f_t$. 
Dobbs and Todd  \cite{Dob} have pointed out to us that it is not very difficult  to
construct, starting from Thunberg's result, sequences of {\it renormalisable} Collet--Eckmann 
  maps (with nonuniform ``goodness,''
in the terminology introduced below) converging to a Collet--Eckmann map,
but such that the SRB measures do not converge. 
Dobbs and Todd \cite{Dob} have recently generalised
this result, finding {\it non-renormalisable} Collet--Eckmann 
 maps $f_{t'_{n}}$ (with nonuniform ``goodness'') converging to a Collet--Eckmann map $f_t$,
but such that the SRB measures do not converge.  
Such counter-examples can  be constructed while requiring that $f_t$ and all maps
$f_{t'_{n}}$ are  Misiurewicz--Thurston.
(Misiurewicz--Thurston  maps
are the smooth unimodal maps enjoying the most  expansion, see below \eqref{eq.alpha}
for a definition.)
These examples show that
continuity of $R_A(t)$ cannot hold on the set of {\it all}
Collet--Eckmann (or even Misiurewicz--Thurston) parameters: Some
uniformity in the constants is needed.

Existence of the SRB measure
holds under conditions much weaker than Collet--Eckmann
(see
\cite{RS} and references therein). Continuity of the SRB measure can be studied on suitable sets
of ``good'' parameters enjoying this weaker property.
(We would like also to draw attention to the exciting new approach of Shen \cite{WX}
to stochastic stability.)
Our aim here however is to study  {\it moduli} of continuity
of $t\mapsto R_A(t)$ for families of  smooth unimodal 
maps --- in order to go beyond mere continuity,
it seems wise (and perhaps necessary) to
restrict to subsets of good Collet--Eckmann parameters.

Until the present work, the only results going beyond continuity concerned 
fully horizontal families, that is, when all $f_t$  are topologically conjugated
to  $f_0$. Even in this ``trivial" setting, where linear response
can  indeed be obtained  (\cite{BS3},\cite{ruelle}, 
\cite{BS}), proofs were technically
involved, in particular in  \cite{BS}, where analyticity
was not assumed and the slow recurrence assumption was relatively weak.

We address here for the first time the modulus of continuity
of the SRB measures in  {\it transversal} families
of (nondegenerate) smooth unimodal maps. 
We conjectured  (\cite[(3) in \S 3.2]{BalT}, making more
precise \cite[Conj. B]{baladi}) 
that for  $C^1$ observables $A$ the function $R_A(t)$
is $\eta$-H\"older for all  $\eta <1/2$.
Our first main result, Theorem~\ref{t.main1}, gives a strengthening of this conjecture:
We show there is a set $\Delta$ of
Collet--Eckmann parameters, with $\Delta$ of positive measure, and, for each $t_0\in \Delta$, a
set $\Delta_0\subset \Delta$ of polynomially recurrent parameters containing
$t_0$ as a Lebesgue density point, and 
constants $C\ge 1$, $\Gamma >4$, so that, for every $1/2$-H\"older function $A$,
$$
 \Big|\int A\,  d\mu_t -\int A\,  d\mu_{t_0}\Big|
\le C  \|A\|_{C^{1/2}}|t-t_0|^{1/2}| \log |t-t_0||^\Gamma\, ,
\, \, \forall t \in \Delta_0\, .
$$ 
This immediately implies a more precise result in the analytic case
(Corollary~\ref{t.main2}). In particular,
for the logistic family $f_t(x)=tx(1-x)$, the set $\Delta$ contains almost all Collet--Eckmann parameters.

Our proof implies that the renormalisation period $P_t$ is $\le P_{t_0}$ for
$t\in \Delta_0$, as well as uniform bounds on the exponential mixing for the ergodic components of
$f^{P_t}$  for
$t\in \Delta_0$ so that $P_t=P_{t_0}$(Theorem ~\ref{corr}).

 We expected Conjecture B of \cite{baladi} to be ``essentially optimal."
Making this more precise, we asked in \cite{BS} whether one can
``construct a (non-horizontal) smooth family $f_t$ of quadratic unimodal maps, with 
$f_0$ a good map, so that $t \to \mu_t$, as a distribution of any order, is not differentiable (even in the sense of Whitney, at least for large subsets) at $t=0$[, or] so that it is not H\"older for any exponent $> 1/2$.''
Our second main result, Theorem~\ref{t.main3}, answers this
question positively:
Assuming that the family is transversal at a mixing
Misiurewicz--Thurston  map $f_{t_0}$,  we find a  set of
Collet--Eckmann parameters $\Delta'_{MT}$, accumulating at $t_0$, a constant $C\ge 1$,  and a
 $C^\infty$ function $A_0$, so that
$$
C |t-t_0|^{1/2}\ge\Big|\int A_0\,  d\mu_t -\int A_0\,  d\mu_{t_0}\Big|
\ge C^{-1}  |t-t_0|^{1/2}\, ,\, \, \forall t \in \Delta'_{MT}\, .
$$
We would like to point out that, in 
the piecewise expanding setting, the first counterexamples to
differentiability of the SRB (see \cite{baladi}, \cite{mazzo}) had  been obtained for
sequences of maps having pre-periodic critical points converging
to a map $f_{t_0}$ with a pre-periodic critical point. They were only later 
generalised to essentially all  $f_{t_0}$ \cite{BS1}, and
(except when the
postcritical orbit of $f_{t_0}$ is dense) any $t\to t_0$.

Our results  lead to several challenging questions for families
of smooth unimodal maps, in particular regarding the
size of the largest possible set $\Delta'_{MT}$, and what can be done 
if the Misiurewicz--Thurston assumption on $f_{t_0}$ is relaxed. (See the
comments after the statements of Theorems~\ref{t.main1} and ~\ref{t.main3}
and Corollary~\ref{t.main2} below.) We would like to note here a quantitative difference
with respect to the piecewise expanding case \cite{BS1} where
the modulus of continuity in the transversal case was 
 $|\log |t-t_0|| |t-t_0|$, so that violation of linear response arose from
 the logarithmic factor alone.

More open questions are listed in
\cite{BalT} and \cite{BS}.
In particular, the results in the present paper also give hope that analogous problems 
(see  \cite{BalT} and \cite{Ru11}) can be studied
for (the two-dimensional) H\'enon family,  which is transversal,
and where continuity
of the SRB measure in the sense of Whitney in the weak-$*$ topology
was proved by Alves et al. \cite{ACF}, \cite{ACF2}.

We would like also to suggest here a weakening
of the linear response problem: Consider a one-parameter family $f_t$ of
(say, smooth unimodal maps) through $f_{t_0}$ and, for each 
$\epsilon>0$, a random perturbation of $f_t$ with unique invariant
measure $\mu_{t}^\epsilon$, e.g., like in \cite{WX}. Then  for each positive
$\epsilon$, it should not be very difficult to
see that the map $t \to \mu_{t}^\epsilon$ is differentiable at $t_0$
(for essentially any topology in the image). Can we say something
(existence? dependence on  the perturbation?
relation with the susceptibility function or some of
its ``extensions'' \cite{BMS}?) about
the limit as $\epsilon \to 0$ of this derivative? (For a weak topology in the
image,
like Radon measures, or distributions of positive order.)

\smallskip
Before sketching the contents of the paper, we would like to highlight here some
of the difficulties we had to face, and what are the new ideas and techniques with respect to the
construction in \cite{BS}: We wish to compare the SRB measure 
of $f_0$ (assume $t_0=0$) to that of $f_t$ for suitable small $t$. Let us start with the similarities
with \cite{BS}:
Just like in \cite{BS}, we use transfer operators
$\widehat \LL_t$ acting on towers, with a projection $\Pi_t$ from
the tower to $L^1(I)$ so that $\Pi_t \widehat \LL_t=\LL_t \Pi_t$,
where $\LL_t$ is the usual transfer operator, and $\Pi_t \hat \phi_t=\phi_t$
with $\mu_t=\phi_t\, dx$ (here, $\hat\phi_t$ is the fixed point
of $\widehat\LL_t$, and $\phi_t$ is the invariant density of $f_t$). In
\cite{BS}, we adapted the tower  construction in \cite{BV}, allowing in particular the use of
Banach spaces of continuous functions. We start from this
adaptation. Another idea we import from
\cite{BS} is the use of truncated operators $\widehat \LL_{t,M}$
acting on truncated towers,
where the truncation level $M$ must be chosen carefully depending on $t$.
Roughly speaking, the idea is that   $f_t$ is comparable to $f_0$
for $M$ iterates (corresponding to the $M$ lowest levels
of the respective towers), this is the notion of an admissible pair $(M,t)$.
Denoting by $\hat \phi_{t,M}$ the maximal eigenvector
of  $\widehat \LL_{t,M}$, the starting point for both our upper and lower bounds
is (like in \cite{BS}) the decomposition (see \eqref{decc})
\begin{equation}\label{dec0}
\phi_t - \phi_0=\bigl [\Pi_t (\hat \phi_t - \hat \phi_{t, M}) +
 \Pi_0 (\hat \phi_{0,M}-\hat \phi_0)\bigr ] +
[\Pi_t ( \hat \phi_{t,M} -\hat \phi_{0,M}  )]
+[ (\Pi_t -\Pi_0) (\hat \phi_{0,M})] \, ,
\end{equation}
for admissible pairs.
The idea is then to get upper bounds on the first two terms by using
perturbation theory \`a la Keller--Liverani \cite{kellerliverani}, and to control
the last (dominant) term  by explicit computations on $\Pi_t -\Pi$
(which represents the ``spike displacement,'' i.e., the effect of the
replacement of $1/\sqrt{|x-f^k_0(c)|}$ by $1/\sqrt{|x-f^k_t(c)|}$ in the invariant
density).

We now move to the differences:
Using a tower with exponentially decaying levels
as in \cite{BV} or \cite{BS} would  limit us 
at best to an upper modulus of continuity $ |t|^{\eta}$ for $\eta <1/2$, and would
not yield any lower bound.
For this reason, we use instead tower levels with {\it polynomially decaying}
sizes, working with polynomially recurrent maps (``fat towers''). In order to construct the corresponding parameter set, we need
to {\it make use of very recent results of Gao and Shen} \cite{GS}.

It turns out  that applying directly the results of Keller--Liverani \cite{kellerliverani}
would only give  that the contributions
of the first and second terms of \eqref{dec0} are
bounded by $ |t|^{\eta}$ for $\eta <1/2$.
In order to estimate
the second term, we prove that $\widehat \LL_{t,M} -\widehat  \LL_{0,M}$ {\it acting
on the maximal eigenvector} is $O(|\log |t||^\Gamma |t|^{1/2})$
{\it in the strong norm} (see Lemma ~\ref{l.strongnorm}, which is used in 
Proposition~\ref{p.strongnorm};
in the Misiurewicz--Thurston case we get
get a better $O( |t|^{1/2})$ control).
It is usually not possible to obtain strong
norm bounds when bifurcations are present \cite{BY, kellerliverani}
(see \cite{gal} for an exception), and this remarkable feature here
is due to our choice of admissible pairs (combined with 
the fact that the towers for $f_t$ and $f$ are identical up to level $M$, just like
in \cite{BS}, see 
Lemma~\ref{bottomok}).  In order to estimate
the first term, we enhance the Keller--Liverani argument
(Proposition~\ref{p.trunc}), using again
that it suffices to estimate the
 perturbation for the operators {\it acting on the maximal eigenvector.} 
 
 The changes just described are already needed to obtain the exponent
 $1/2$ in Theorem~\ref{t.main1}. 
 In order to get lower bounds of Theorem~\ref{t.main3}, we use that {\it the tower associated to a
 Misiurewicz--Thurston map $f_0$ can be required to have levels
 with sizes bounded from below,} and that the truncation level
 can be chosen to be slightly larger ($2M$ instead of $M$).
 The final change, that we explain next, is also only needed to 
 obtain the
 lower bound in Theorem~\ref{t.main3}. Working
 with Banach norms based on $L^1$ as in \cite{BS} would give that the
 first two terms  in \eqref{dec0} are $\le C  |t|^{1/2}$, while the third is 
 $\ge C ^{-1} |t|^{1/2}$ for some large constant $C>1$. In other
 words, the estimates are too tight. However, {\it introducing Banach--Sobolev norms based on $L^p$ 
 for $p>1$} instead, we are able  to control the constants and make sure that the third
 term dominates the other two, as needed (see Section ~\ref{s.ndiff}).

The paper is organised as follows:
In the remainder of this section, we furnish precise definitions, as well as formal statements of
our main results.
In Section~ \ref{ss.uniform}, we construct the good parameter sets $\Delta_0\subset \Delta$
(Proposition~\ref{p.uniformc}), and we define the corresponding (polynomially
recurrent) good maps. In Section ~\ref{ss.tower},
we construct the tower, and we  collect the needed expansion and distortion bounds.
Section~\ref{uuniform} contains  Definition~\ref{Adm} of admissible pairs $(M,t)$.
In Section~ \ref{ss.transfer},
we introduce the strong and weak Banach norms ($\BB^{W^1_1}_t$, $\BB^{L^1}_t$,
$\BB^{L^p}_t$) on the tower, define the transfer operator $\widehat \LL_t$ associated to
$f_t$ and acting on these spaces, and list its main spectral properties. In Section 
~\ref{ss.trunc}, we introduce
the truncated transfer operators $\widehat \LL_{t,M}$ which play a key role in our analysis.
Section ~\ref{MTcase} contains the
construction of  the parameter set  $\Delta_{MT}$
and a brief description of the modifications which can be used
  to take advantage of the Misiurewicz--Thurston
setting. Then, we prove  Theorem~\ref{t.main1} in
Section ~\ref{main1} and Theorem ~\ref{t.main3} in
Section ~\ref{s.ndiff}.
The two appendices contain necessary but straightforward adaptations
of bounds in \cite{BS}.


\subsection{Setting}

\begin{definition}
\label{d.smooth}
The  {\it smooth one-parameter families of smooth nondegenerate unimodal maps} $f_t$ 
studied in the present paper are defined as
follows:
Let $I=[0,1]$, and fix $c$ in the interior of $I$.  We consider $C^2$ maps
$t \mapsto f_t$, from a nontrivial closed interval $\EE$ of $\real$ to $C^3$ endomorphisms $f_t$ of $I$.
We  {\it assume} that  each $f_t$ is a  $C^4$ unimodal
 map with
negative Schwarzian derivative and critical point $c$ (independent on $t$), and that
the $C^4$ norm of $f_t$ is bounded uniformly in $t$. We {\it suppose}
further that
$f_t''(c)<0$ (this is the nondegeneracy, or
quadratic-like property),
and that $f_t(0)=f_t(1)=0$. 
Put $c_{k,t}=f_t^k (c)$, for $k\ge0$, and set
$$
v_{t} = \partial_s f_s |_{s=t}\, .
$$
The function $v_t: I \to \real$ is $C^1$ (with a bound on the norm independent on $t$) by assumption. 
Finally, we {\it assume} that there  exist
uniformly $C^1$ functions $X_t:I \to \real$ so that
$$
v_t = X_t \circ f_t \, .
$$
\end{definition}

The archetypal example is the logistic family
\begin{equation}\label{logistic}
f_t(x)=tx(1-x)\, , \quad t\in  \EE \subset (0,4]\, ,
\end{equation}
where $c=1/2$, and $X_t (x) \equiv 1/t$.
The map $f_4$ (for which $c_{1,4}=f_4(c)=1$ and $f_4(c_{2,4})=c_{2,4}=0$) is called the Ulam--von Neumann map.

A map $f_t$ (or the corresponding parameter $t$) is called
{\it $(\lambda_c,H_0)$-Collet--Eckmann} for some $\lambda_c > 1$ and $H_0\geq 1$
(or simply Collet--Eckmann, if the meaning is clear) if
\begin{equation}\label{00}
|(f_t^k)' (c_{1,t}) |\ge  \lambda_c ^k\, , \quad \forall k \ge H_0 \, .
\end{equation}
Recall \cite{CoEc} that any Collet--Eckmann unimodal map $f_t$ admits a unique absolutely
continuous invariant probability measure 
$\mu_t=\phi_t\, dx$, also called the SRB measure. This measure is
ergodic and supported inside $[c_{2,t},c_{1,t}]$.
A map $f_t$ (or the corresponding parameter $t$) is called {\it mixing} if
$f_t$ is topologically mixing on $[c_{2,t}, c_{1,t}]$. The support of the SRB measure $\mu_t$
of a mixing map is equal to $[c_{2,t}, c_{1,t}]$.
A unimodal map $f_t$
is {\it renormalisable} if there exists an interval neighbourhood $\RR_{c,t}$ of $c$
so that the first return map to this interval is again a
unimodal map, and the smallest return time $P_t$ is at least two. 
The  largest such $P_t$ is  called the {\it renormalisation period.}
The map $f_t$ is mixing if and only if $f_t$ is not renormalisable  --- we say that
the renormalisation period $P_t$ of $f_t$ is equal to $1$ in this case.

The family $f_t$ is called {\it transversal} at a Collet--Eckmann parameter $t_1$, if $t_1$ lies in
the interior of $\EE$, and
\begin{equation}
\label{eq.transversal}
\JJ_{t_1}:=\sum_{j=0}^\infty\frac{\partial_t f_t(c_{j,t_1})|_{t=t_1}}{(f_{t_1}^j)'(c_{1,t_1})}\neq0\,.
\end{equation}
Slightly abusing language, we say that $t$  is a transversal Collet--Eckmann parameter 
if  $t$ is a Collet--Eckmann parameter in 
the interior of $\EE$ and \eqref{eq.transversal} holds.

\subsection{Whitney--H\"older regularity for  smooth families of nondegenerate smooth unimodal maps}

Our first result  settles  the upper bound  conjecture in \cite[Conj. B]{baladi} (see also \cite[\S 3.2]{BalT}).

\begin{theorem}[Whitney--H\"older regularity for transversal families]
\label{t.main1}
Let $f_t$ be a smooth one-parameter family of smooth nondegenerate unimodal maps.
If there exists a  transversal Collet--Eckmann parameter $t_1$,
then there exists a positive Lebesgue measure set $\Delta\subset \EE$ of Collet--Eckmann parameters such that for all 
$t_0\in\Delta$ and all $\Gamma>4$ there is a set $\Delta_0\subset\Delta$, which has $t_0$ as a Lebesgue density point,
and a constant $C$ such that for every $t\in\Delta_0$ and
each $1/2$-H\"older function $A$, we have
$$
\Big|\int_I A(x)d\mu_t-\int_I A(x)d\mu_{t_0}\Big|\le C|t-t_0|^{1/2}|\log|t-t_0||^\Gamma\|A\|_{C^{1/2}}\,,
$$
where
$$
\|A\|_{C^{1/2}}=\|A\|_{L^\infty}+\sup_{x\neq y}\frac{|A(x)-A(y)|}{|x-y|^{1/2}}\, .
$$
\end{theorem}

Restricting to $C^1$ functions $A$ (or functions of higher smoothness)
should not improve the upper bound (see \eqref{eq.step2}).

As a byproduct of our proof, we obtain the following result :

\begin{theorem}[Uniform bounds on renormalisation periods and rates of mixing]\label{corr}
In the setting of Theorem~\ref{t.main1}, the renormalisation period $P_t$ of $f_t$ is
not larger than $P_{t_0}$ for all $t\in \Delta_0$. In addition, for any
any $\zeta >0$ there exists $\Theta_1>1$ so that for all
$t\in \Delta_0$ for which $P_t=P_{t_0}$, each ergodic component $\mu_{j,t}$, $j=1, \ldots P_t$, of
$(f_t^{P_t},\mu_t)$, and all $C^\zeta$ functions $\psi$ and $\varphi$, there exists $C_{\varphi, \psi}$ so that
$$
\Big|\int (\varphi \circ f_t^{k P_t}) \psi \, d\mu_{j,t} -\int \varphi\, d\mu_{j,t}\, \int
\psi\, d\mu_{j,t}\Big|\le C_{\varphi, \psi}
\Theta_1^{-k}\, .
$$
\end{theorem}
(Theorem~\ref{corr} is an immediate corollary
of the last claim of Proposition~\ref{p.strongnorm}.)

We  next discuss the sets
$\Delta$ and $\Delta_0$.

A map $f_t$ is called  {\it polynomially recurrent of exponent $\alpha > 0$,} 
if there is $H_0\ge1$ so that 
\begin{equation}
\label{eq.alpha}
|c_{k,t}-c|> k^{-\alpha},\quad \text{for all } k\ge H_0\,.
\end{equation}
A map $f_t$ is called   polynomially recurrent of exponent $0$ 
if there is $C\ge1$ so that 
$|c_{k,t}-c|> 1/C$ for all $k\ge 1$.
A map $f_t$, or a parameter $t$, is called {\em Misiurewicz--Thurston} if the critical point of $f_t$ is pre-periodic, but not periodic (the postcritical
periodic orbit is then necessarily a strictly expanding orbit). Misiurewicz--Thurston maps are  Collet--Eckmann and
polynomially recurrent of exponent $0$. 
Misiurewicz--Thurston maps
 are not generic.

All parameters in the set $\Delta$ constructed in Theorem~\ref{t.main1} are polynomially recurrent for some
exponent $\alpha >1$. 
{\it Understanding the largest possible sets 
$\Delta$ and $\Delta_0$ for which Theorem~\ref{t.main1} 
holds, and whether the logarithmic factor can be suppressed
is a challenging question.} We conjecture that Theorem~\ref{t.main1} holds for $\Delta$ 
the set of all ``sufficiently slowly recurrent''
parameters,  where sufficiently slowly recurrent should include 
 polynomial recurrence of exponent $\alpha=0$ (the so-called Misiurewicz  case).
See Corollary ~\ref{t.main2} for  analytic families (where $\Delta$ contains almost
all Collet--Eckmann parameters), and
the upper bound in Theorem~\ref{t.main3} when $t_0$ is  Misiurewicz--Thurston
(without the Lebesgue density point property for the analogue of $\Delta_0$).

\begin{remark}
\label{r.all}
If  the transversality condition \eqref{eq.transversal} holds for almost all Collet--Eckmann parameters $t_1\in\EE$, 
then the set $\Delta$ in Theorem~\ref{t.main1} can be taken equal to the set of Collet--Eckmann 
parameters. This follows from Proposition~\ref{p.uniformc}. For example, a 
non-trivial analytic family of nondegenerate unimodal maps has this property, in particular
this holds for the logistic family $f_t(x)=tx(1-x)$. (See Section~\ref{ss.analytic}.)
\end{remark}

\begin{remark}[Mixing]\label{rk.mix} In \cite[Beginning of \S 5.2]{BS} it is claimed
incorrectly that $1$ is always the only eigenvalue  of the transfer operator
on the unit circle.
Since we did not assume mixing in \cite{BS}, there could be in fact  finitely
many other simple eigenvalues
of modulus one  in general (they are roots of unity --- see the  proof
of Proposition~\ref{mainprop} in Appendix~\ref{misc2} below and the reference \cite{karlin} to Karlin there). So, when constructing
the contour integrals in \cite[(112), Step 1 in \S 6]{BS}, we should avoid not only
a neighbourhood of the disc of radius $\theta_t$ there (see also \eqref{thetat}), but also  neighbourhoods of these
other eigenvalues of modulus $1$ (see the circle $\gamma$ in the proofs of
Propositions~\ref{truncspec} and~ \ref{p.strongnorm} below). 
Note also that exponential decay
of correlations is not needed
(up to replacing $\widehat \LL^n$ by $k^{-1} \sum_{n=0}^{k-1}\widehat \LL^{n}$
in the proof of the last claim of \cite[Proposition 4.11]{BS}, see 
\eqref{rrref}).

Note finally that we cannot apply the exactness argument from \cite[Corollary 2]{BV}
to show that  $1$ is a simple eigenvalue
for a nonnegative eigenvector of the transfer operator
(contrarily to what was stated in the proof of Proposition~4.11 of \cite{BS}),
because the transfer
operator $\widehat \LL_t$ is associated to a probabilistic and not a deterministic tower map.
However, we may apply classical results on positive operators \cite{karlin}
(details are given in Appendix~\ref{misc2} below).
\end{remark} 


\subsection{A stronger result in the analytic case}
\label{ss.analytic}
In the  case of the logistic  family $f_t(x)=tx(1-x)$, $t\in(0,4]$, 
Benedicks and Carleson \cite{BC} showed 
that the set of parameters $t$ for which $f_t$ is Collet--Eckmann has positive Lebesgue measure. 
A parameter $t$ is called {\it regular} if the critical point $c$ of $f_t$ belongs to the basin 
of a hyperbolic periodic attractor. 
The parameter $t$ is called {\it stochastic} if $f_t$ has an absolutely 
continuous invariant measure.
By Lyubich \cite{lyub}, Lebesgue almost every parameter is either regular or stochastic. 
Avila and Moreira \cite{AM1} proved that for almost every stochastic parameter $t$, 
the map $f_t$ is Collet--Eckmann. 
Further, in \cite{ALM} and \cite{AM2} the results in \cite{lyub} and \cite{AM1} 
are extended to non-trivial analytic
families of nondegenerate  unimodal maps. 
({\it Analytic} means that each $f_t$ is analytic and $t\mapsto f_t$ is
analytic, {\it non-trivial} means that the family is not 
contained in a topological class.)
Since  every Collet--Eckmann parameter $t_1$ of a non-trivial analytic family of 
nondegenerate maps $f_t$ 
is transversal (see \cite{Levin}),  Theorem~\ref{t.main1}, Theorem~\ref{corr}, and Remark~\ref{r.all},
give the following result.

\begin{corollary}[Application to analytic families of nondegenerate maps]
\label{t.main2}
Let $f_t$, be a non-trivial analytic family of nondegenerate (analytic) unimodal maps. 
For almost every Collet--Eckmann parameter $t_0 \in \EE$ 
and all $\Gamma>4$ there is a set $\Delta_0\subset  \EE$ of Collet--Eckmann parameters 
which has $t_0$ as a Lebesgue density point 
and a constant $C$ such that, for all $t\in\Delta_0$ and $A\in C^{1/2}([0,1])$,
$$
\Big|\int_I A(x)d\mu_t-\int_I A(x)d\mu_{t_0}\Big|\le C|t-t_0|^{1/2}|\log|t-t_0||^\Gamma\|A\|_{C^{1/2}}\, .
$$

In addition, the renormalisation period $P_t$ of $f_t$ is
not larger than $P_{t_0}$ for all $t\in \Delta_0$, and for any $\zeta >0$ there exists $\Theta_1>1$ so that for all
$t\in \Delta_0$ for which $P_t=P_{t_0}$, each ergodic component $\mu_{j,t}$, $j=1, \ldots, P_{t}$, of 
$(f^{P}_t, \mu_t)$ and all $C^\zeta$ functions $\psi$, $\varphi$ there exists $C_{\varphi, \psi}$ so that
$$
\Big|\int (\varphi \circ f_t^{k P_t}) \psi \, d\mu_{j,t }-\int \varphi\, d\mu_{j,t}\, \int
\psi\, d\mu_{j,t }\Big|\le C_{\varphi, \psi}
\Theta_1^{-k}\, .
$$
\end{corollary}

\smallskip

Again, understanding {\it the largest possible set
of parameters $t_0$ and $\Delta_0$ for which Corollary~\ref{t.main2} 
holds, and whether the logarithmic factor can be suppressed
is a challenging question.}
We  conjecture  that Corollary~\ref{t.main2} holds for all Collet--Eckmann parameters 
$t_0$ with ``sufficiently slow'' recurrence.

\subsection{H\"older upper and lower bounds for Misiurewicz--Thurston parameters}
Our second main result  addresses the lower bound in  Conjecture B in \cite{baladi} (see also \cite[\S 3.2]{BalT}).

\begin{theorem}[H\"older upper and lower bounds]
\label{t.main3}
Let $f_t$ be a smooth
one-parameter family of smooth unimodal maps.
Let $t_0$ be a mixing transversal Misiurewicz--Thurston parameter.
Then there exist an observable $A\in C^\infty$, a constant $C\ge 1$, 
and a sequence of Collet--Eckmann parameters $t_{(n)}$, $n\ge1$, 
with $t_{(n)}\to t_0$ as $n\to\infty$, such that
\begin{equation}\label{themain}
\frac{|t_{(n)}-t_0|^{1/2}}{C}\le\Big|\int_I A(x)d\mu_{t_{(n)}}-\int_I A(x)d\mu_{t_0}\Big|\le C|t_{(n)}-t_0|^{1/2}\,,\quad\forall\ n\ge1 \, .
\end{equation}
\end{theorem}

The mixing assumption is for simplicity (the proof shows that it suffices to suppose that
the deepest  renormalisation of --- the finitely renormalisable map -- $f_{t_0}$ is not conjugated to the
Ulam--von Neumann map).

The proof of Theorem~\ref{t.main3} produces a
 set $\Delta_{MT}$ of parameters $t_{(n)}$ which are also Misiurewicz--Thurston
(see Lemma~\ref{l.existencemt} and its proof), and either
 all $>t_0$ or all $<t_0$ (see Remark~\ref{sameside}). 
Using continuity of the absolutely continuous invariant
measures in the sense of Whitney (see the works of Tsujii
\cite{tsujiicont} and Rychlik--Sorets \cite{RySo}, or 
more recently Alves et al. \cite{ACF}, \cite{ACF2}), we can then easily construct
sequences of Collet--Eckmann parameters $\tilde t_{(n)}$ which are not Misiurewicz--Thurston, and
for which the lower bound in Theorem~\ref{t.main3} hold. However, we do not know if $t_0$ is a Lebesgue density point of the set of all such $\tilde t_{(n)}$. {\it Understanding the  set of sequences for which Theorem~\ref{t.main3}
holds is a challenging question. In fact, the toy model analogue of this question is also open
(see \cite[Theorem 7.1]{BS1} and its corrigendum, even in the case of
a pre-periodic critical point).}

Last, but not least,  the lower bounds in the piecewise expanding toy model of \cite[Theorem 7.1]{BS1} had been
first obtained only in the case when the critical point is pre-periodic
\cite[Theorem 6.1]{baladi}. {\it We conjecture that the conclusion 
\eqref{themain} of 
Theorem~\ref{t.main3} should also hold when $t_0$ enjoys a more generic slow recurrence
condition (such as polynomial recurrence), perhaps up to introducing a power of
$|\log|t_{(n)}-t_0||$
in both sides of \eqref{themain}.}

\section{Preliminaries -- Towers --  Transfer operators}
\label{s.prel}
In  Section \ref{ss.uniform}, we show that transversality at a Collet--Eckmann parameter ensures
 that polynomial recurrence
holds for uniform  constants  $\lambda_c>1$, $\alpha>1$, $H_0\ge1$, 
for a positive measure set of parameters $t$, with Collet--Eckmann parameters as Lebesgue density
points. In 
Section~\ref{ss.tower}
we adapt the tower map construction in \cite{BS} to our polynomially
recurrent setting.  Section ~\ref{uuniform} contains Lemma ~\ref{rootsing} giving estimates on iterated
unimodal maps for all $f_s$ close enough to a good parameter $f_t$
and all iterates 
not too big compared with $|t-s|$, as well as  the definition of admissible pairs $(M,t)$.

\subsection{Uniformity of constants}
\label{ss.uniform}
\begin{proposition}[Parameter set of good maps with uniform constants]
\label{p.uniformc}
Assume that $t_1\in\EE$ is a transversal Collet--Eckmann parameter.
Then, there exists a set $\Delta\subset \EE$ of transversal 
Collet--Eckmann parameters, 
for which $t_1$ is a Lebesgue density point \footnote{We do not claim that $t_1\in \Delta$.}
(in particular, $\Delta$ has positive Lebesgue
measure), so that
for all $t_0\in\Delta$ and all $\alpha>1$
there exist $H_0 \ge 1$ and 
a set $\Delta_0\subset\Delta$ containing $t_0$ 
as a Lebesgue density point such that, for all $t\in\Delta_0$, 
the map $f_t$ is $(\lambda_c,H_0)$-Collet--Eckmann 
and the polynomial recurrence condition \eqref{eq.alpha} of exponent $\alpha$
holds for all $k\ge H_0$.

Furthermore, for each $t_0\in\Delta$, there exist constants  $\rho>1$ (we may assume
$\rho < \sqrt \lambda_c$) and $C_0>0$  such that, for all $\delta>0$ 
sufficiently small, there is a constant $\epsilon(\delta)>0$ such that, 
for all $t\in(t_0-\epsilon(\delta),t_0+\epsilon(\delta))$, 
\begin{equation}
\label{eq.noreturn}
|(f_t^n)'(x)|\ge C_0\delta\rho^n\,,
\forall x \mbox{ so that }  |f_t^j(x)-c|\ge\delta\, , \forall \, 0\le j<n\, , 
\end{equation}
and
\begin{equation}
\label{eq.deltareturn}
|(f_t^n)'(x)|\ge C_0\rho^n\,,
\forall x \mbox{ so that }  |f_t^j(x)-c|\ge\delta\, , \forall \, 0\le j<n\, ,
\mbox  { and } |f_t^n(x)-c|<\delta \, .
\end{equation}
\end{proposition}

The following definition summarises the properties of parameters in 
the positive measure set $\Delta_0$ constructed in the previous proposition:

\begin{definition}[Good maps]\label{good}
A map $f_t$ (or the corresponding parameter $t$) is called good for the constants
$$
\lambda_c>1\, , H_0\ge 1\, , \alpha \ge 0\, ,
\rho>1\, , C_0>0\, ,
$$
(called its ``goodness constants'', or simply ``goodness'')
if $f_t$ is $(\lambda_c, H_0)$-Collet--Eckmann \eqref{00}, if it satisfies the polynomial recurrence 
condition \eqref{eq.alpha} for  $\alpha$ and $H_0$ (when $\alpha=0$ we require 
that $|c_{k,t}-c|\ge C_0$, for all $k\ge1$),  and if the expansion conditions
\eqref{eq.noreturn} and \eqref{eq.deltareturn} hold for  $\rho$, $C_0$, 
and any small $\delta>0$.
\end{definition}

\begin{proof}[Proof of Proposition~\ref{p.uniformc}]
We apply a recent work by Gao and Shen \cite{GS}. 
Since the map $f_{t_1}$ is Collet--Eckmann, all 
periodic orbits are repelling. Since $f_{t_1}$ is also transversal, we can apply 
\cite[Main Theorem]{GS}
which implies that there exists a set $\EE_1$ for which $t_1$ is a Lebesgue density 
point such that, for all parameters $t\in \EE_1$, 
the map $f_t$ satisfies the polynomial recurrence 
condition \eqref{eq.alpha} for any $\alpha >1$, $f_t$ is transversal at $t$, and $f_t$ is 
Collet--Eckmann.
For $j,H,\ell\ge1$, set
$$
\Omega_{j,H,\ell}=\left\{t\in \EE_1 \mid |c_{k,t}-c|\ge k^{-(1+j^{-1})},\ 
\mbox{ and }
|(f_t^k)'(c_{1,t})|\ge e^{k/\ell}, \ \forall k\ge H\right\}\, .
$$ 
For a set $\Omega$, let $\Omega^L\subset\Omega$ denote the set of Lebesgue density points 
within $\Omega$. We set 
$\Delta=\cap_{j\ge1}(\cup_{\ell\ge1}\cup_{H\ge1}\Omega_{j,H,\ell}^L)$. 
By Lebesgue's density theorem, $\Delta$ is equal to the set $\EE_1$ up to some zero 
Lebesgue measure set. Now, by construction, for all $t_0\in\Delta$ and all $\alpha>1$, 
choosing $j_0\ge$ such that $(1+j_0^{-1})\le\alpha$ we find $H_0\ge1$ and $\ell_0\ge1$ 
such that $t_0\in\Omega_{j_0,H_0,\ell_0}^L$. Setting 
$\Delta_0=\Omega_{j_0,H_0,\ell_0}^L\cap\Delta$ 
this concludes the first part of 
the proof of Proposition~\ref{p.uniformc}.

Regarding the second part we will apply a lemma by Tsujii \cite{tsujii}.
The conditions (ND), (CE)(i), (Hyp), (W), and (NV) in \cite{tsujii} are satisfied for $f_{t_0}$.  
The only point we have to check is the  ``backward Collet--Eckmann condition''
(CE)(ii), i.e., the existence of
constants $C>0$ and $\hat \rho>1$ such that 
\begin{equation*}
|(f_{t_0}^k)'(b)|\ge C\hat \rho^k\,,\quad\text{if}\ k\ge 1\ \text{and}\ f_{t_0}^k(b)=c\,.
\end{equation*}
Since $f_{t_0}$ is Collet--Eckmann, unimodal, and has negative Schwarzian derivative, 
\cite[Theorem A]{ns}
guarantees that the backward Collet--Eckmann condition holds.
By \cite[Lemma~5.1 (2)]{tsujii}, there exist $\rho>1$, $\delta_0>0$, and $C_0>0$ such that, 
for all $0<\delta\le\delta_0$, 
there is $\epsilon(\delta)>0$ such that \eqref{eq.noreturn} and \eqref{eq.deltareturn} hold, 
for all $f_t$ with $|t-t_0|\le\epsilon(\delta)$. (The fact that the constant in front of $\rho^n$ 
in \eqref{eq.noreturn} depends linearly on $\delta$ follows from the last line in the proof of 
\cite[Lemma~5.1]{tsujii}.)
\end{proof}

\subsection{The tower map for good $f_t$ --- Distortion estimates}
\label{ss.tower}
Assume that $f_t$ is good (recall Definition~\ref{good}) for
$\lambda_c$, $\alpha>1$,  $H_0$,
$\rho$, and $C_0$. 
(The case when $f_t$ is Misiurewicz--Thurston with $\alpha=0$
is treated in Section ~\ref{MTcase}.)
Let $\delta>0$ be small, to be determined later as a function essentially
of the goodness parameters and of the $C^2$ norm of $t\mapsto f_t$
(see, e.g., condition~\eqref{2.5}, condition just above \eqref{ints}, 
inequality~\eqref{eq.transl} in Lemma~\ref{rootsing}, and Proposition~\ref{ubalpha}, 
in which $\delta$ is so small such that \eqref{eq.deltaL} is satisfied;
in \eqref{2.5} and in \eqref{eq.deltaL}, $\delta$ depends on $L$, which 
is chosen in Lemma~\ref{bottomok} via Lemma ~\ref{rootsing} again).

We introduce a tower map $\hat f_t: \hat I_t \to \hat I_t$ 
similar to the one constructed in \cite[Section~3]{BS} 
(see also \cite{BV}). The tower in the
present paper is ``fatter'' since our polynomial recurrence assumption
allows us to choose  the size of the levels  polynomially small, instead of exponentially small. 
(To get the lower bound of Theorem~\ref{t.main3}, we shall later use levels of constant size, under a
 Misiurewicz--Thurston  assumption.)
Fix a constant
\begin{equation}\label{betas}
 \beta>\alpha+1 \, .
\end{equation}
The {\it tower} $\hat I_t$ associated to $f_t$ is the union
$\hat I_t= \cup_{k \ge 0} E_{k,t}$
of levels $E_{k,t} = B_{k,t} \times \{k\}$ satisfying the following properties:
The ground floor interval $B_{0,t}=B_0$ is the interval $I$. Fix a constant $L>1$
(the value of $L$ will be chosen  in Lemma~\ref{bottomok}).
For $k \ge 1$, the interval $B_{k,t}$ is 
centered at $c_{k,t}$ and such that
\begin{equation}\label{betaas}
\Big[c_{k,t}-\frac{k^{-\beta}}{L^3}, c_{k,t} + \frac{k^{-\beta}}{L^3}\Big] \subset B_{k,t}
\subset \Big[c_{k,t} - \frac{k^{-\beta}}{L}, c_{k,t}+ \frac{k^{-\beta}}{L}\Big] \, .
\end{equation}
(Observe that since $\beta>\alpha$ and $L>1$, 
we have $c \notin B_{k,t}$ for all $k \ge H_0$.)
Note that for given $\lambda_c>1$ and $H_0\ge1$, there exists a constant $C>1$ such 
that $\inf_{1\le k\le H_0}|c_{k,t}-c|\ge2C^{-1}$, for all $(\lambda_c,H_0)$-Collet--Eckmann 
maps $f_t$, $t\in\EE$. (This follows directly from $\sup_{t\in\EE}|f_t''(c)|<\infty$ combined 
with the expansion \eqref{00}.) Henceforth, (for given $\lambda_c>1$ and $H_0\ge1$) 
we assume that $\delta > 0$ is so small that for all $x$ with $|x-c|\le\delta$ 
and all $(\lambda_c,H_0)$-Collet--Eckmann maps $f_t$, $t\in\EE$, 
we have 
\begin{equation}\label{2.5}
|f_t^j(x)-c|\ge C^{-1}\,,\qquad
\text{for all } 1 \le j \le  H_0\, .
\end{equation}

For $(x,k) \in E_{k,t}$  we set
\begin{equation*}
\hat f_t (x,k) =\begin{cases}
(f_t(x),k+1) &\text{ if } k \ge 1 \text{ and } f_t(x) \in B_{k+1,t}\, ,\cr
(f_t(x),k+1) &\text{ if } k = 0 \text{ and } |x-c|\le\delta\, , \\
(f_t(x),0) &\text{ otherwise.}
\end{cases}
\end{equation*}
Denoting $\pi : \hat I_t \to I$ the projection to the first factor,
we have $f_t \circ \pi = \pi \circ \hat f_t$ on $\hat I_t$.
Define $H(\delta)$ to be the minimal
$k \ge 1$ such that there exist
some $x\in [c-\delta,c+\delta]$
such that $\hat f_t^{k+1}(x,0) \in E_0$. (In other words if a point starts to climb the tower, 
it climbs the tower at least until level $H(\delta)$.) 
For fixed goodness parameters $(\lambda_c,H_0,\alpha,\rho,C_0)$, fixed $\beta>1+\alpha$, 
and fixed $L$, observe that for each 
$H\ge1$, we can choose $\delta>0$ so small such that $H(\delta)\ge H$ for 
all good maps and their corresponding towers with these fixed parameters. 
(This follows immediately from the fact 
that $\sup_{t\in\EE}\|f_t\|_{C^1}<\infty$ and $|B_{k,t}|\ge2k^{-\beta}/L^3$.)
We assume throughout that $\delta>0$ is so small that $H(\delta)\ge\max(2,H_0)$.

We decompose $[c-\delta,c+\delta]\setminus\{0\}$ as a disjoint union of intervals
\begin{align}
\nonumber &[c-\delta, c+\delta]\setminus \{0\}=
\cup_{j \ge  H(\delta)} I_{j,t}\, , \quad I_{j,t}:= I_{j,t}^+ \cup I_{j,t}^-
\, , \\
\label{ints} 
&I_{j,t}^\pm:=\left\{ |x|<\delta, \pm x > 0,
\hat f_t^\ell(x,0)\in E_{\ell,t}, 0 \le \ell < j,
\hat f_t^j(x,0)\in E_0\right\}\, .
\end{align}
In other words, $I_{j,t}$ is the set of points which climb up the tower $j-1$ levels and 
fall back to the level $E_0$ at the $j$-th iteration. 
Note that $I_{j,t}$ can be empty for some $j$. 
(In particular $I_{j,t}=\emptyset$ for $1\le j<H(\delta)$.) For $k\ge0$, let 
$J_{k,t}$ 
denote the set of points in $[c-\delta,c+\delta]$ which climb up the tower at least $k$ levels, i.e.,
\begin{equation}\label{defJ}
J_{k,t}:=\{c\}\cup\bigcup_{j\ge k+1}I_{j,t}\,.
\end{equation}
Obviously, we have $J_{k,t}=[c-\delta,c+\delta]$, for all $0\le k\le H(\delta)$.
Later, when considering a fixed good map $f_0$, 
we will write $B_k$, $I_k$ and $J_k$ for $B_{k,0}$, $I_{k,0}$ and $J_{k,0}$, respectively.

The following lemma is the adaptation
of the distortion estimates from
\cite[Lemma 3.3]{BS} (which was an avatar of \cite{BV}, \cite[Lemma 5.3(1)]{V})
to our fat towers:

\begin{lemma}[Bounded distortion in the bound period]\label{cd}
Let $f_t$, $t\in\EE$, be good and assume that condition \eqref{2.5} is satisfied. 
Then, there exists $C > 1$ 
depending only on the goodness constants 
$(\lambda_c,H_0,\alpha)$, and on $\beta$ (in particular $C$ does not depend on 
$L$), such that for every $j \ge 1$, and every $k\leq j$, 
\begin{equation}\label{infprodgena}
 C^{-1} \leq  \frac{ |(f^k_t)'(x)|}{|(f^k_t)'(y)|}\le C \, ,\,\qquad 
\forall  x,y \in f_t(J_{j,t})\, .
\end{equation}
\end{lemma}

\begin{proof} 
Recall the intervals $B_{\ell,t}$ from \eqref{betaas}.
For $1 \le \ell \le k  \le j$, pick $x_\ell$ and $y_\ell$  in $f_t^\ell(J_{j,t})
\subset B_{\ell,t}$. 
Observe that there exists a constant $C$ such that $|f_t'(y)|\ge|y-c|/C$, for all $y\in I$ and $t\in\EE$.
We have 
\begin{align}\label{infprod}
\prod_{\ell=1}^{k}
\frac{|f_t'(x_\ell)|}{|f_t'(y_\ell)|}
\nonumber&
\le  \prod_{\ell=1}^{k}
\biggl (1+\frac{\sup  |f_t''|}{|f_t'(y_\ell)|}  |x_\ell-y_\ell|\biggr )
\le  \prod_{\ell=1}^{k}
\biggl (1+C\sup  |f_t''| \frac{|x_\ell-y_\ell|}{|y_\ell-c|}\biggr )
\\
&
\le  \prod_{\ell=1}^{\infty}
 (1+\tilde C \sup  |f_t''|  \ell^{-\beta+\alpha} )< \infty \, ,
\end{align}
uniformly in $j$.
We used that $|x_\ell-y_\ell|\le L^{-1} \ell^{-\beta}$ and, if
$\ell>H_0$, that
$|y_\ell-c|\geq  \ell^{- \alpha}-L^{-1}\ell^{-\beta}\ge\ell^{-\alpha}/2$. 
If $1\le\ell \le H_0$, we used condition~\eqref{2.5}. 

The series $\ell^{-\beta+\alpha}$ is summable since
$\beta > \alpha+1$. Choosing  $y_\ell=f^{\ell-1}(y)$ and $x_\ell=f^{\ell-1}(x)$, 
we get the upper bound in (\ref{infprod}). Taking $y_\ell=f^{\ell-1}(x)$ and $x_\ell=f^{\ell-1}(y)$,
we obtain the lower bound in (\ref{infprod}). Observe that while the last product in \eqref{infprod}
blows up when $\beta$ tends to $1+\alpha$, it does not depend on the constant $L$.
\end{proof}

The following key estimate is our polynomial version of
\cite[Proposition 3.7]{BS} (the proof is to be found in Appendix~\ref{misc}):

\begin{proposition}[Key estimate for polynomially recurrent maps]\label{ubalpha}
Let $f_t$ be good and assume that condition \eqref{2.5} is satisfied. 
If $\delta>0$ is sufficiently small (see \eqref{eq.deltaL}), then there exists $C> 0$
depending only on the goodness constants 
$(\lambda_c,H_0,\alpha,\rho,C_0)$, on $\beta$, and on
$\delta$, such that for every $j \geq 0$ we have
 \begin{equation}\label{serieb}
  \sum_{k =j+1}^\infty \frac{1}{|(f_t^{k-j})'(f_t^j(c_1))|} \leq C j^ \alpha \, .
\end{equation}
\end{proposition}

Our proof gives a constant $C$ which blows up when 
$\delta \to 0$, and it requires smaller $\delta$ if
 $L$ is large, but considering a fixed map $f_{t_0}$ as in Theorems~\ref{t.main1} and \ref{t.main3}, 
 both parameters may be chosen once and for all, depending
on the goodness parameters of $f_{t_0}$ and the $C^2$ norm of $t \mapsto f_t$.

The following notation will be convenient:
For $k\ge 1$, let
\begin{equation*}
f^{-k}_{t,+}:=(f_t^k|_{U_{t,+}})^{-1}\, ,
\, \, f^{-k}_{t,-}:=(f_t^k|_{U_{t,-}})^{-1}\, ,
\end{equation*}
where $U_{t,+}$ is the monotonicity
interval of $f_t^k$ containing $c$ located to the right of $c$, and
$U_{t,-}$ is the monotonicity
interval of $f_t^k$ containing $c$, located to the left of $c$.

The following polynomial version of the upper and lower bounds in 
\cite[Lemma 3.4 and  Lemma 4.1]{BS}), about 
points which climb for exactly $j-1$ steps, recall \eqref{ints}, will be needed:

\begin{lemma}[The $j$-bound intervals $I_{j,t}^\pm$]\label{sizeIj}
Let $f_t$ be good and assume that condition \eqref{2.5} is satisfied.
Then there exists a constant $C$ 
depending only on the goodness constants 
$(\lambda_c,H_0,\alpha)$ and on $\beta$ 
(in particular $C$ does not depend on $L$), such that 
we have
\begin{equation}\label{upbdx}
|x-c |\le C|B_{j-1,t}|^{1/2}| (f^{j-2}_t)'(c_{1,t})|^{-1/2}  \,,\quad\forall x\in J_{j-1,t},\ j\ge1\,,
\end{equation}
and for all $x\in  I_{j,t}$, $j\ge H(\delta)$, we have
\begin{equation}\label{eight'} |f_t'(x)|\ge \frac{1}{C} 
\frac{\sqrt{|f_t^j(x)-c_{j,t}|}}{ |(f^{j-1}_t)'(c_{1,t})|^{1/2}} \ge \frac{1}{C L^{3/2}} j^{-\beta/2} 
\frac{1}{|(f^{j-1}_t)'(c_{1,t})|^{1/2}}
 \, ,
\end{equation} 
\begin{equation}\label{expii} 
 |(f^j_t)'(x)|\geq \frac{1}{C} \sqrt{|f_t^j(x)-c_{j,t}|} |(f^{j-1}_t)'(c_{1,t})|^{1/2}
\geq \frac{1}{ CL^{3/2}} j^{-\beta/2} |(f^{j-1}_t)'(c_{1,t})|^{1/2}
\, .
\end{equation} 
In addition,  there exists a constant $C$ 
depending on $L$ and on the constant in Proposition~\ref{ubalpha} 
so that
for all $j \ge H(\delta)$ and $x \in f^j(I_{j,t})$ we have,
for $\zeta=+$ or $-$,
\begin{equation}
\label{rootsing2}
\biggl|\partial_x \frac{1}{|(f^j_t)'(f^{-j}_{t,\zeta} (x))|}
\biggr | \le C  \frac{j^{\max(1+2\alpha+\beta/2,3\beta/2)}}
{|(f_t^{j-1})'(c_1)|^{1/2}}\, , 
\end{equation}
and, finally,
\begin{equation}
\label{rootsing3}
\biggl |
\partial^2_x \frac{1}{|(f^j_t)'(f^{-j}_{t,\zeta} (x)))|}
\biggr | \le   C  \frac{j^{\max(4\alpha+1+\beta/2,5\beta/2)}}
{|(f_t^{j-1})'(c_{1,t})|^{1/2}}\, .
\end{equation}
\end{lemma}


An immediate corollary of \eqref{upbdx} is exponential decay of the
length of $I_{j,t}$ and $J_{j,t}$. More precisely, for any 
fixed goodness constants and any $L$,
there exists a constant $\tilde C$ so that for all $j\ge H_0$ 
\begin{equation}\label{decaybound}
| I_{j,t} |  \le | J_{j-1,t} | \le \tilde Cj^{-\beta/2}{|(f_t^{j-2})'(c_1)|^{-1/2}} 
\le \tilde C^2 j^{-\beta/2} \lambda_c^{-j/2}\, ,
\end{equation}
where the second last inequality holds for all $j\ge2$.

\begin{proof}
To simplify the writing, we remove the  $t$
from the notation and write, e.g., $f$, $I_j$, and $J_j$ instead of $f_t$, $I_{j,t}$, and $J_{j,t}$.

Let $x \in J_{j-1}$, $j\ge1$.
First, our definitions and the mean value theorem
imply that there exists $y\in J_{j-1}$ so that
$$
|(f^{j-2})'(f(y))| |f(x)-c_1|\le|B_{j-1}|/2\, .
$$ 
Therefore, Lemma ~ \ref{cd} and the fact that
$|f(x)-c_1|\ge C^{-1} |x-c|^2$ 
(recall that $f'(c)=0$ and $f''(c)\ne 0$) yield \eqref{upbdx}.

Next, the reverse consequence of the mean value theorem
\begin{equation}\label{reverse}
|(f^{j-1})'(f(y))| |f(x)-c_1| \ge C^{-1} |f^j(x)-c_k|
\ge  C^{-1}L^{-3} j^{-\beta}\, ,
\end{equation}
together with $|f(x)-c_1|\le C |x-c|^2$  and Lemma ~ \ref{cd} gives
\begin{equation}\label{upbdx'}
 |x-c |\ge  C^{-1} \frac{|f^j(x)-c_k|^{1/2}}{| (f^{j-1})'(f(y))|^{1/2}}\ge   \frac{1} {C^2L^{3/2}}
j^{-\beta/2} | (f^{j-1})'(c_1)|^{-1/2}\, ,
\end{equation}
where in the last inequality we used that $f^j(x)\notin B_j$.
The bound  (\ref{eight'}) then
follows from \eqref{upbdx'}.

To show \eqref{expii}, we decompose 
$|(f^j)'(x)|
= |(f^{j-1})'(f(x))| |f' (x)|
$, and we apply  \eqref{eight'},
noting that Lemma ~ \ref{cd}, implies
\begin{equation}\label{firstt}
|(f^{j-1})'(f(x))| \ge C^{-1} |(f^{j-1})'(c_1)|\, .
\end{equation}

Note that  reversing the inequalities in the arguments above  also gives
\begin{equation}\label{expiiu} 
 |(f^j_t)'(x)|\leq C \sqrt{|f_t^j(x)-c_{j,t}|} |(f^{j-1}_t)'(c_{1,t})|^{1/2}
\, .
\end{equation} 

Assume now that $\zeta=+$ (the other case is similar).
To prove \eqref{rootsing2}, recall first the proof of Lemma~\ref{cd} which implies that  there is
$C \ge 1$ (depending on the goodness and, 
in a weaker way, on $\beta$ and $L$)
so that 
\begin{equation}\label{firstnote}
|f'(f^j(y))|\ge C^{-1} j^{-\alpha} \, , \, 
\forall y\in J_k\, , \, \forall 1 \le j \le k\, .
\end{equation}
Next, Lemma~\ref{cd} and
Proposition ~\ref{ubalpha} give $C>0$ so that
\begin{align}
\label{basic}
\sup_{y\in J_k}\frac{1}{|(f^{k-j})'(f^j(y))|}&\le C
\sum_{\ell = 1}^{\infty}
\frac{1}{  |(f^{\ell})'(f^{j-1}(c_1))|}\le   C j^\alpha \,  , 
\, \forall 1 \le j \le k-1 \, .
\end{align}
Applying \eqref{firstnote} and \eqref{basic} for $j \ge 1$
and \eqref{eight'} and \eqref{expii} for $j=0$,  we find $C>0$ so that
\begin{align}
 \label{gluu}
\sup_{y\in I_k}
\partial_y \frac{1}{|(f^k)'(y)|}&\le
  \sup_{y\in I_k}
\sum_{j=0}^{k-1}
\frac{|f''(f^j(y))|}{|(f^{k-j})'(f^j(y))|}\frac {1}{|f'(f^j(y))|}
\le C  k^{\max(1+2\alpha, \beta)}   \, .
\end{align}
Then, \eqref{expii}  (or its proof) gives
\begin{equation}
\label{basic2}
\sup_{y \in I_k}\frac{1}{|(f^m)'(y)|}\le 
C {L^{3/2}}
k^{\beta/2} \frac{ 1}
{ |(f^{m-1})'(c_1)|^{1/2}}\, ,
\quad \forall  1\le m \le k \, .
\end{equation}
Finally, if $x \in f^k(I_k)$,
\begin{equation}\label{finally}
\partial_x \frac{1}{|(f^k)'(f^{-k}_+(x))|}=\frac{1}{|(f^k)'(f^{-k}_+(x))|}\cdot
\partial_y \frac{1}{|(f^k)'(y)|} \, .
\end{equation}
The first factor in \eqref{finally} is bounded
by \eqref{basic2} for $m=k$, the second by \eqref{gluu}, so that
 we have proved \eqref{rootsing2}.

To prove \eqref{rootsing3}, we start from the decomposition
\eqref{finally}, and  we deduce from
\eqref{gluu}  that for any
$x\in f^k(I_k)$, setting $y=f^{-k}_+(x)$,
\begin{align}
\nonumber \partial^2_x \frac{1}{|(f^k)'(f^{-k}_+(x))|}
&\le C k^{\max(2 \alpha,\beta)} \partial_x \frac{1}{|(f^k)'(f^{-k}_+(x))|} 
 \\
\label{toprove} &\quad +
\frac{1}{|(f^k)'(f^{-k}_+(x))|}\cdot
 \sum_{j=0}^{k-1}
  \partial_x\biggl (
\frac{|f''(f^j(y))|}{|(f^{k-j})'(f^j(y))| |f'(f^j(y))|}\biggr )\, .
\end{align}
By \eqref{rootsing2},
the first term in the right hand side is bounded by a constant times
$k^{\max(1+4\alpha+\beta/2,5\beta/2)}|(f^{k-1})'(c_1)|^{-1/2}$.
For the  second term, we have,
for $0\le j \le k-1$,
$$
 \partial_x
\frac{|f''(f^j(y))|}{|(f^{k-j})'(f^j(y))||f'(f^j(y))|}
\le  
\frac{1}{|(f^k)'(f^{-k}_+(x))|}
\partial_y
\frac{|f''(f^j(y))|}{|(f^{k-j})'(f^j(y))||f'(f^j(y))|}\, .
$$
Since $f$ is $C^3$, the Leibniz formula  gives for $0\le j \le k-1$, 
\begin{align}
\nonumber
&
\frac{1}{|(f^k)'(f^{-k}_+(x))|}\partial_y
\frac{|f''(f^j(y))|}{|(f^{k-j})'(f^j(y))||f'(f^j(y))|}\\
\nonumber &\qquad\le
\frac{1}{|(f^k)'(f^{-k}_+(x))|}
\frac{1}{|(f^{k-j})'(f^j(y))||f'(f^j(y))|}
 \cdot 
\biggl [
|f'''(f^j(y))||(f^j)'(y)|\\
\label{next}  &\qquad\qquad\qquad+
|f''(f^j(y))|
\bigl [ \frac{|f''(f^j(y))||(f^j)'(y)|}{|f'(f^j(y))|}+
\sum_{\ell=j}^{k-1}
\frac{|f''(f^\ell(y))| |(f^\ell)'(y)|}{|f'(f^\ell(y))|}
\bigr ]
\biggr ]\\
\nonumber &\qquad\le
\frac{C }{|(f^{k-j})'(f^j(y))|} \cdot 
\biggl [\frac{1}{|(f^{k-j})'(f^j(y))|}
\Bigl (\frac{1}{|f'(f^j(y))|} +\frac{1}{|f'(f^j(y))|^2}\Bigr )
\\ \nonumber &\qquad\qquad\qquad\qquad\qquad\qquad\qquad
\qquad\qquad\qquad\quad
+\frac{1}{|(f^{k})'(y)|}
\sum_{\ell=j}^{k-1}
\frac{|(f^\ell)'(y)|}{|f'(f^\ell(y))|}
\biggr ]\, .
\end{align}
If $j\ge 1$, we may apply \eqref{firstnote} 
and \eqref{basic}, so that
\eqref{next} implies
\begin{align*}
&
\frac{1}{|(f^k)'(f^{-k}_+(x))|}
\sum_{j=1}^{k-1}
\partial_y
\frac{|f''(f^j(y))|}{|(f^{k-j})'(f^j(y))||f'(f^j(y))|} 
\le
C k^{\alpha+1} \cdot 
\bigl [k^{2\alpha}+k^{3\alpha} +
k^{\alpha} 
\bigr ]\, .
\end{align*}
If $j=0$, then \eqref{next} together with \eqref{eight'} and
\eqref{basic2} for $m=k$ imply
(distinguish between $\ell=0$ and $\ell \ge 1$)
\begin{align*}
&
\frac{1}{|(f^k)'(f^{-k}_+(x))|}\partial_y
\frac{|f''(y)|}{|(f^{k})'(y)||f'(y)|}
\\ \nonumber &\qquad\qquad
\le C \frac{k^{\beta/2}}{|(f^k)'(c_1)|^{1/2} }(2 k^{\beta}+ 
k^{3\beta/2}|(f^k)'(c_1)|^{1/2} 
+  k^{\alpha +1})\, .
\end{align*}
The two above inequalities, together with \eqref{toprove}
and \eqref{basic2} for $m=k$,
give
\eqref{rootsing3}.
\end{proof}

\subsection{Maps in a neighbourhood of a good map -- Admissible pairs $(M,t)$}
\label{uuniform}

We  next state  some basic facts about the maps $f_s$ in a neighbourhood
of a good map $f_{t_0}$. To simplify the writing, we assume $t_0=0$ and remove the 
$t_0$ from the notation.
We emphasize that the maps $f_s$  in the following lemma
are not necessarily all Collet--Eckmann. (Indeed, for both
our main theorems, we shall apply the mean value theorem,  $B(t)-B(0)=t B'(s_t)$, 
or the fundamental theorem
\footnote{To prove (115) in \cite{BS} one should apply the fundamental
theorem of calculus and Fubini instead of the mean value theorem, see \eqref{eq.fubini}
for a similar computation.} of calculus
 $B(t)-B(0)=\int_0^t B'(s)\, ds $, in parameter space. Even
if $0$ and $t$ are good parameters, the parameters
$s_t$ and $s \in [0,t]$ are not all good.)
Recall the intervals $I_{k,t}$ and $J_{k,t}$ defined in \eqref{ints} and \eqref{defJ}.

\begin{lemma}[Uniformity of goodness and distortion constants for suitable maps
$f_s$ and iterates $M$]\label{rootsing}
Let $f=f_0$
be good for  parameters $(\lambda_c, H_0, \alpha>1, \rho, C_0)$ and assume that 
\eqref{2.5} is satisfied.
Then there exist constants $C\ge1$ and $\epsilon>0$
(depending only on $\lambda_c$, $H_0$, $\alpha$, $\rho$,  $C_0$,
$\beta$, and in particular not
on $L$)  so that, for any pair
$(s,M)$, $M\ge1$ and $s\in(-\epsilon,\epsilon)$ satisfying 
\begin{equation}\label{okstart}
|(f^{k-1})'(c_{1})||s|\le k^{-\beta}
\, ,\quad \forall 1\le k\le M\, ,
\end{equation}
the following holds:
We have
\begin{equation}\label{eq.tdistortion1}
C^{-1} \le \left|\frac{(f_s^{k-1})'(x)}{(f^{k-1})'(x)}\right|\le C\, ,
\quad\forall x\in f(J_{k-1})\,,\ \forall 1<k\le M\,.
\end{equation}
Furthermore, we have
\begin{equation}
\label{rootsing1}
|(f_s^k)'(x)|\ge C^{-1}L^{-3/2}k^{-\beta/2}|(f^{k-1})'(c_{1})|^{1/2}\, ,
\quad \forall x\in I_k\,,\forall 1\le k\le M\,,
\end{equation}
and
\begin{equation}
\label{eq.transu}
|\partial_sf_s^k(x)|\le C|(f^{k-1})'(c_{1})|\,,
\quad\forall x\in J_{k-1}\,,\forall1\le k\le M\,, 
\end{equation}
and, if $\JJ_0\neq0$ (recall \eqref{eq.transversal}) then for $\delta>0$ sufficiently small we have
\begin{equation}
\label{eq.transl}
|\partial_sf_s^k(x)|\ge C^{-1} |\JJ_0||(f^{k-1})'(c_{1})|\,,\quad \forall x\in J_{k-1}\,,
\forall H(\delta)\le k\le M\,, 
\end{equation}
where $\sign(\partial_sf_s^k(x)\cdot (f^{k-1})'(c_{1}))=\sign(\JJ_0)$.
Finally, for $1\le k<l\le M-1$, we have
\begin{equation}
\label{eq.postcrit1}
\frac{|(f_s^{k})'(c_{1,s})|}{|(f_s^{l})'(c_{1,s})|}\le Ck^{\alpha}\, ,
\end{equation}
where the constant in \eqref{eq.postcrit1} depends in addition also on $\delta$.
\end{lemma}

\begin{proof}
As a preliminary step, note that \eqref{okstart} implies that
we can assume, up to taking
$\epsilon$ small enough (depending only on $\lambda_c$ and $H_0$) that
\begin{equation}\label{okstartM}
|s|< M^{-\beta}\, .
\end{equation}

Next, for $x\in J_{M-1}$ and $s$ satisfying \eqref{okstart}, 
let 
$$D_k=\sup_{t\in[0,s]}|f^k(x)-f_t^k(x)|\, .
$$ 
Fix $\alpha+1<\beta_0<\beta$. We claim that 
\begin{equation}
\label{eq.claim}
D_k\le k^{-\beta_0}\,,\quad\forall 1\le k\le M-1\,.
\end{equation} 
We show this by induction over $k$. For $\epsilon>0$ sufficiently small, 
 \eqref{eq.claim} obviously holds for all small $k$'s and arbitrary $M$. 
Assume that $1<k\le M-1$ and that the claim holds for $k-1$.
Observe that 
\begin{equation}
\label{eq.partialt}
\partial_tf_t^k(y)=\sum_{j=1}^k(f_t^{k-j})'(f_t^j(y))(\partial_tf_t)(f_t^{j-1}(y))\, ,
\quad\forall y\in I\ \text{and}\ t \in\EE\,,
\end{equation}
and there exists $\tilde C>1$ so that
\begin{equation*}
|\partial_tf_t(y)|,\ |f_t''(y)|,\ |\partial_tf_t'(y)|\le\tilde C\,, \quad
\forall y\in I\ \text{and}\ t\in \EE \,.
\end{equation*}
Hence, by applying twice the mean value theorem 
\begin{align*}
|f_t'(f_t^k(x))-f'(f^k(x))|&\le
|f_t'(f_t^k(x))-f_t'(f^k(x))|+|f_t'(f^k(x))-f'(f^k(x))|\\
&\le\tilde CD_k+\tilde C|t|\,,\quad\forall t\in[0,s]\,,
\end{align*}
and we get $|f_t'(f_t^k(x))|\le|f'(f^k(x))|+CD_k+\tilde C|t|$.
Combined with \eqref{eq.partialt}, it follows
\begin{equation}
\label{eq.221}
\left|\frac{\partial_tf_t^k(x)}{(f^{k-1})'(f(x))}\right|
\le\sum_{j=1}^k\frac{\tilde C}{|(f^{j-1})'(f(x))|}
\prod_{i=j}^{k-1}\left(1+\frac{\tilde C|t|+\tilde CD_i}{|f'(f^i(x))|}\right)\,.
\end{equation}
Since $f^{i+1}(J_{M-1})\subset B_{i+1}$, for $i\le M-2$, it follows that 
$|f^i(x)-c|\ge|f^i(c_{1})-c|-i^{-\beta}/L$. 
By \eqref{eq.alpha} and \eqref{2.5} (maybe increase $\tilde C$ in order 
to control the small $i$'s), we see that 
$|f^i(x)-c|\ge\tilde C^{-1} i^{-\alpha}$, for all $i\le M-2$, and since the critical point is nondegenerate,
we have $|f'(f^i(x))|\ge\tilde C^{-1}|f^i(x)-c|\ge\tilde C^{-2}i^{-\alpha}$. 
Together with the induction assumption on $D_i$, $i\le k-1$, and the assumption on $s$, 
it follows that for all $t\in[0,s]$
\begin{align*}
\frac{\tilde C|t|+\tilde CD_i}{|f'(f^i(x))|}
\le\tilde C^3i^\alpha(|t|+D_i)
\le\tilde C^3i^\alpha(M^{-\beta}+i^{-\beta_0})
\le2\tilde C^3i^{-(\beta_0-\alpha)}\,.
\end{align*}
Since $i^{-(\beta_0-\alpha)}$ is summable, 
the product in \eqref{eq.221} is uniformly bounded by a constant $C'$ and, 
by the mean value theorem and the distortion estimate
Lemma~ \ref{cd} for $t=0$, we conclude
$$
D_k\le C'\tilde CC(\lambda_c-1)^{-1}|(f^{k-1})'(c_{1})||s|
\le C'\tilde CC(\lambda_c-1)^{-1}k^{-\beta}\,,
$$
where in the last inequality we used the assumption
\eqref{okstart}  on $s$.
This shows \eqref{eq.claim}. 

If $f$ is Misiurewicz--Thurston (with $\alpha=\beta=0$), 
then as explained in Section~\ref{MTcase} below the right hand side of \eqref{okstart} is 
replaced by a sufficiently small constant $\eta>0$. Then, we
derive by a similar calculation as the one 
showing \eqref{eq.claim} that there exists a constant $C'$ (depending only 
on the goodness parameters of $f$) such that for all 
sufficiently small $\eta>0$,
\begin{equation}
\label{eq.MTeta}
D_k\le C'\eta|(f^{M-k})'(c_k)|^{-1}\,,\quad\forall\ 1\le k\le M-1\,.
\end{equation}
The $\eta$ in the above bound is then used to show a positive lower bound for \eqref{eq.etalower} 
below. The remaining estimates for the polynomial case and the 
Misiurewicz--Thurston case are the same. For further comments when $f$ is 
Misiurewicz--Thurston see Section~\ref{MTcase}.

We may now proceed to the estimates.
For the distortion estimate \eqref{eq.tdistortion1}, we find, similarly as when deriving \eqref{eq.221},
\begin{equation}
\label{eq.etalower}
\left|\frac{(f_s^{k-1})'(x)}{(f^{k-1})'(x)}\right|
\le\prod_{i=0}^{k-2}\left(1+\frac{\tilde C|s|+\tilde CD_i}{|f'(f^i(x))|}\right)\,, 
\quad\forall x\in f(J_{k-1})\,.
\end{equation}
Using \eqref{eq.claim} and \eqref{okstartM}, we can proceed as in the proof
of Lemma~\ref{cd} to show  that the above product is bounded. 
The lower bound is obtained in a similar way, 
where, without loss of generality, it is enough to consider the case of large $M$. 
(To deal with small $M$, we might decrease $\epsilon$, like when proving
\eqref{okstartM}.)

The estimate \eqref{rootsing1} follows from  
$|(f_s^k)'(x)|\ge C^{-1}|x-c||(f_s^k)'(f_s(x))|$ combined 
with \eqref{eq.tdistortion1}, Lemma~\ref{cd}, and \eqref{upbdx'}.

By a similar calculation as in deriving \eqref{eq.221} 
(whose right hand side is uniformly bounded, as we have shown above), 
there exists a constant $C'\ge1$ so that
\begin{equation}
\label{eq.partialss}
|\partial_sf_s^k(x)|\le C'|(f^{k-1})'(f(x))|\,,\quad \forall x\in J_{k-1}\,.
\end{equation}
By the distortion estimate Lemma~\ref{cd}, this shows \eqref{eq.transu}.

Regarding \eqref{eq.transl}, recall \eqref{eq.tdistortion1} 
and \eqref{infprodgena}, and the fact that $f$ is 
$(\lambda_c,H_0)$-Collet--Eckmann. For $H_0\le k_0\le k\le M-1$ and
$x\in J_{k-1}$, we get
\begin{equation}
\label{eq.trans1l}
\Big|\sum_{j=k_0}^{k}\frac{(\partial_s f_s)(f_s^j(x))}{(f_s^j)'(f_s(x))}\Big|
\le\sum_{j=k_0}^{k}\frac{\tilde C C^2}{|(f^j)'(c_1)|}
\le\frac{\tilde C C^2\lambda_c}{\lambda_c-1}\lambda_c^{-k_0}\,.
\end{equation}
Fix $k_0\ge H_0$ such that the right hand side of \eqref{eq.trans1l} is smaller 
than $\JJ_0/4$. Once $k_0$ is fixed, we can take $\epsilon$ and $\delta$
small enough so that 
$$
\Big|\sum_{j=0}^{k_0-1}\Big(\frac{(\partial_t f_t)(c_j)|_{t=0}}{(f^j)'(c_1)}
-\frac{(\partial_s f_s)(f_s^j(x))}{(f_s^j)'(f_s(x))}\Big)\Big|\le\frac{\JJ_0}4\,.
$$
Recalling \eqref{eq.partialt} and the definition of $\JJ_0$ in \eqref{eq.transversal}, we conclude that 
$$
\Big|\frac{\partial_sf_s^k(x)}{(f_s^{k-1})'(f_s(x))}\Big|
=\Big|\sum_{j=0}^{k-1}\frac{(\partial_s f_s)(f_s^j(x))}{(f_s^j)'(f_s(x))}\Big|\ge\frac{\JJ_0}4\,.
$$
Applying once more \eqref{eq.tdistortion1} 
and \eqref{infprodgena}, this implies \eqref{eq.transl}, provided 
$\delta$ is so small such that $H(\delta)\ge k_0$ (observe that the constants in the estimates 
we used do not depend on $\delta$). The statement about the signs follows immediately.

Finally, by Proposition~\ref{ubalpha}
\footnote{The proof of this fact uses properties~\eqref{eq.noreturn} 
and \eqref{eq.deltareturn} in Proposition~\ref{p.uniformc}.}
there exists a constant $C'$ such that for all $t$ which is 
good for the same parameters $\lambda_c$, $H_0$, $\alpha>1$, $\rho$, $C_0$, we have
$$
|(f_t^k)'(c_{1,t})|/|(f_t^{l})'(c_{1,t})|\le C'k^{\alpha}\,,\quad\forall l>k\ge1\,,
$$
and claim \eqref{eq.postcrit1}  follows immediately 
by \eqref{eq.tdistortion1}.
\end{proof}

 Let $f_t$ be a smooth one-parameter family of smooth nondegenerate unimodal maps.
As usual, we put $f=f_0$.
Adapting  \cite[(107)--(109)]{BS} to the polynomial towers of the present work,
and in view of Lemma ~\ref{bottomok},
we introduce a key definition:

\begin{definition}[Admissible pairs]\label{Adm}
Let $C_a > C$, where $C\ge 1$ 
is given by  Proposition ~\ref{ubalpha}.
Let $\alpha \ge 0$, $\beta \ge 0$, $\epsilon >0$.
A pair $(M,t)$ with  $M \in \integer_+$ 
is called a $(C_a, \alpha, \beta,\epsilon)$-{\it admissible pair}
(or just an admissible pair, if the meaning is clear)  if $0<|t|<\epsilon$ and
\begin{equation}\label{bd2}
|(f^M)'(c_1)||t|\le C_a^{-1}M^{-(\alpha+\beta)}\, , 
\end{equation}
and $M$ is maximal for this property.
\end{definition}

Observe that if $(M,t)$ is an admissible pair then, by the maximality of $M$, we find 
a constant $C$ (only dependent on $\sup|f'|$) such that 
\begin{equation}
\label{eq.admupper}
|(f^M)'(c_1)|^{-1}\le CC_aM^{\alpha+\beta}|t|\,.
\end{equation}

To motivate the definition, 
let  $f=f_0$, be good 
for parameters  $\lambda_c$,  $H_0$, $\alpha >1$,
$\rho$, $C_0$.
Let $\delta>0$  be so small such that all  results in Section~\ref{s.prel} hold. 
Choose $\beta >\alpha +1$.
Let $\epsilon>0$ be given by  Lemma~\ref{rootsing}).
(In our application below, $\epsilon$ may be further reduced
when invoking  Lemma~\ref{bottomok}.)
Then, if  $t$ is good for the
same parameters and $(M,t)$ is a
$(C_a, \alpha,\beta,\epsilon)$-admissible pair, we claim
that the  estimates in Lemma~\ref{rootsing} hold for $M$ and all $|s|\le t$ ,
with constants depending only on $C_a$. (Indeed,  \eqref{eq.postcrit1} holds for $s=0$ by Proposition~\ref{ubalpha}
since $C_a$ is larger than the constant $C$ from that proposition,
so that \eqref{okstart} is satisfied by the admissibility condition.)
In addition, using \eqref{okstart} again,
 we may ensure, by Lemma~\ref{bottomok} below, that
the tower of $f_t$ coincides with that of $f$ up to level $M$.

\section{Banach spaces and transfer operators  on  the tower}
\label{uu}

In this section we define the Banach spaces, transfer operators, and truncated
transfer operators used to prove our theorems, and 
we strengthen the results from \cite{BS} on these objects:
In Sections ~\ref{ss.transfer} and \ref{ss.trunc} we consider $f_t$ good for
parameters $(\lambda_c, H_0, \alpha, \rho, C_0)$, with
$\alpha>1$, using the notation and tower construction from Section~\ref{ss.tower}. 
In Section~\ref{MTcase},
we summarise the changes needed to adapt the constructions of 
Sections ~\ref{s.prel}, ~\ref{uu}, and ~\ref{ss.transfer}, \ref{ss.trunc}
to the Misiurewicz--Thurston case (where we take $\alpha=0$).

\subsection{Banach spaces and transfer operators $\widehat \LL_t$}
\label{ss.transfer}
Just like in \cite{BS}, we shall work with Sobolev spaces. For integer $r\ge0$, recall that the 
generalized Sobolev norm of $\psi:I\to\complex$ is 
$$
\|\psi\|_{W_1^r}=\|\partial_x^r\psi(x)\|_{L^1(I)}\,.
$$
Note that $\|\psi\|_{L^\infty}\le C\|\psi\|_{W_1^1}$ (cf. inequality \eqref{sobeb} below).

Fix $\lambda$ so that
\begin{equation}\label{48}
1 <\lambda  <  \min(\lambda_c^{1/2},\sqrt \rho) \, .
\end{equation}
(The square root in $\lambda<\sqrt \rho$ is used in \eqref{eq.hm} below.)
Let $\Lambda_t\ge \lambda_c$ be so that, for some constant $C=C_t\ge 1$,
\footnote{The supremum of Collet--Eckmann constants $\lambda_{c,t}$ is not
always a Collet--Eckmann constant, this is why we introduce $\Lambda_t$.
See also Lemma~\ref{l.existencemt}.}
\begin{equation}\label{defL}
| (f_t^k)'(c_{1,t})| \ge \frac{\Lambda_t^k}{C_t}\, ,\,\,  \forall k \ge 1 \, .
\end{equation}

We first introduce the Banach space of functions on the
tower on which the transfer operator (to be defined next) will act:

\begin{definition}[Spaces $\BB_t=\BB_t^{W_1^1}$, $\BB^{L^1}_t$, $\BB^{L^p}_t$]\label{Bspace}
Let $\BB_t=\BB_t^{W_1^1}$ be the space of sequences
$\hat \psi=(\psi_k :I \to \complex,\,  k \in \integer_+)$,  so that
each $\psi_k$ is $W_1^1$ and, in addition,
\begin{equation}
\label{defban}
\supp( \psi_0) \subset (0, 1) \, , \quad\text{and}\quad
\supp (\psi_k )\subset
J_{k,t}
\, , \quad \forall k \ge1\, , 
\end{equation}
endowed with the norm 
$$
\| \hat \psi\|_{\BB_t}= 
\sum_{k\ge0}\|\psi_k\|_{W_1^1}\, .
$$
Let $\BB^{L^1}_t$ be the space of
sequences $\hat \psi$ of 
functions $\psi_k \in L^1(I)$ satisfying \eqref{defban}, 
with
\begin{equation}\label{trunc}
\|\hat \psi\|_{\BB^{L^1}_t}=\sum_{k\ge 0} \lambda^k\|\psi_k \|_{L^1(I)}\, .
\end{equation}
For $p > 1$ and $\rr=\rr(t,p)$ so~\footnote{Note that \eqref{48} implies that $r < 1/p<1$.
If $\lambda \to 1$ then $\rr \to -\infty$, but it is instead convenient to 
take $\lambda\sim \min(\lambda_c^{1/4},\sqrt{\rho})$, in view of \eqref{theta0}.} that
\begin{equation}
\label{eq.mu}
\lambda^{1-\rr}=\Lambda_{t}^{\frac{1}{2}(1-\frac{1}{p})} \,,
\end{equation}
let $\BB_t^{L^p}$ be the space\footnote{Defining $\BB_t^{L^p}$
by interpolation instead would not be appropriate, in view
of  \eqref{eqMT}, \eqref{eq.hm'}.} of sequences $\hat\psi$ of functions $\psi_k\in L^p(I)$ 
satisfying \eqref{defban}, with 
$$
\|\hat\psi\|_{\BB_t^{L^{p}}}:=\sum_{k \ge 0}\lambda^{k\rr}\|\psi_k\|_{L^{p}(I)}\,.
$$
\end{definition}

 \smallskip

\begin{remark}[Strong and weak norms]
Generally, $\BB_t^{W_1^1}$ will be the ``strong'' norm and  $\BB^{L^1}_t$
the ``weak'' norm, in the usual Lasota--Yorke meaning, see e.g. \eqref{eq.LaYo}. 
(It is easy to check that $\BB_t^{W_1^1}$ is continuously
embedded in  $\BB^{L^1}_t$ using \eqref{decaybound} and \eqref{defban}, \eqref{48}.) 

The auxiliary weak norms $\BB_t^{L^{p}}$ for $p> 1$
will only be used in
the Misiurewicz--Thurston case (where $\beta=0$), to get the lower bound in
Theorem~\ref{t.main3}.  We have, recalling \eqref{decaybound}, 
\begin{align*}
\|\hat\psi\|_{\BB_t^{L^{p}}}&=\sum_{k \ge 0}\lambda^{k\rr}\left(\int_I|\psi_k|^{p}\,dx\right)^{1/p}
\le\sum_{k \ge 0}\lambda^{k\rr}|\supp(\psi_k)|^{1/p}\|\psi_k\|_{L^\infty}\\
&\le C\sum_{k \ge 0}\lambda^{k\rr}k^{-\beta/2p} \lambda_c^{-k/2p}\|\psi_k\|_{W_1^1}\,.
\end{align*}
Since  
$\rr <1/p$, we get  $\|.\|_{\BB_t^{L^{p}}}\le C\|.\|_{\BB_t}$
 for any $p> 1$ by using $\lambda <\sqrt \lambda_c$ from \eqref{48}.

In addition the embedding 
$\BB_t^{L^{p}}\subset\BB^{L^1}$ is bounded for any $p > 1$:
\begin{align}\label{L1Lp}
\|\hat\psi\|_{\BB^{L^1}}&=\sum_{k\ge0}\lambda^k\|\psi_k\|_{L^1(I)}
\le\sum_{k\ge0}\lambda^k|\supp(\psi_k)|^{(p-1)/p}\|\psi_k\|_{L^{p}(I)}\\
\nonumber &\le C\sum_{k\ge0}\lambda^k k^{-\beta(p-1)/2p} |(f^{k-1})'(c_1)|^{-(p-1)/2p}\|\psi_k\|_{L^{p}(I)}
\le C C_t \|\hat\psi\|_{\BB_t^{L^p}}\,,
\end{align}
by the H\"older inequality, and the definitions \eqref{eq.mu} of $\rr$ and \eqref{defL} of $\Lambda_t$. 
\end{remark}

The {\it projection} $\Pi_t(\hat \psi)$ for a function $\hat \psi \in \BB_t$ is defined by
\begin{align}\label{defproj}
\Pi_t(\hat \psi)(x)&=\sum_{k \ge 0, \varsigma\in \{+,-\}}
 \frac{\lambda^k}{|(f_t^{k})'(f^{-k}_{t,\varsigma}(x))|} \psi_k(f^{-k}_{t,\varsigma}(x)) 
 \chi_{k,t}(x)
   \, ,
\end{align}
where $\chi_{k,t}=1_{[0, c_{k,t}]}$ if
$f_t^k$ has a local maximum at $c$, while $\chi_{k,t}=1_{[ c_{k,t},1]}$ if
$f_t^k$ has a local minimum at $c$ (we set $\chi_{0,t}\equiv 1$, and 
when the meaning is clear we will omit the factor $\chi_{k,t}$ in the formula; also in 
the definition of the transfer operator $\widehat\LL$ in \eqref{a} below we will not write the 
factor $\chi_{k,t}$).
Note that, for $\hat\psi\in\BB_t$, the function $f_t^k:[c,1]\cap\supp(\psi_k)\to I$ is injective, and 
by a change of variables we have
\begin{equation*}
\int_0^1 \frac{|\psi_k(f_{t,+}^{-k}(x))|\chi_{k,t}(x)}{|(f_t^k)'(f_{t,+}^{-k}(x))|}\, dx
=\int_c^1 |\psi_k(x)| \, dx \, ,
\end{equation*}
and a similar formula holds when considering the branch $f_{t,-}^{-k}$ (instead of integrating over 
$[c,1]$ we integrate over $[0,c]$ on the right hand side). 
Thus, we have $\|\Pi_t (\hat \psi)\|_{L^1(I)}\le\|\hat \psi\|_
{\BB_t^{L^1}}$. The case of $\BB_t^{L^p}$ for $p>1$ is a little less trivial:

\begin{lemma}\label{projp}
For any $p>1$ and any $1\le \tilde p < p\frac{2}{p+1}$ there exists $C(p,\tilde p)\ge 1$ so that
$
\|\Pi_t (\hat \psi)\|_{L^{\tilde p}}
\le C(p,\tilde p) C_t \|\hat \psi\|_{\BB_t^{L^p}}
$.
\end{lemma}

\begin{proof} 
By the Minkowski inequality and a change of variable
\begin{align*}
\|\Pi_t(\hat\psi)\|_{L^{\tilde p}}
&\le\sum_{k \ge 0, \varsigma\in \{+,-\}}\lambda^k
\left(\int_I\frac{|\psi_k(f^{-k}_{t,\varsigma}(x))|^{\tilde p}}{|(f_t^{k})'(f^{-k}_{t,\varsigma}(x))|^{\tilde p}}\,dx\right)^{1/\tilde p}\\
&\le\sum_{k \ge 0}\lambda^k
\left(\int_I\frac{|\psi_k(y)|^{\tilde p}}{|(f_t^{k})'(y)|^{\tilde p-1}}\,dy\right)^{1/\tilde p}\,.
\end{align*}
For $q'>1$ and $p'>1$ so that $q'^{-1}+p'^{-1}=1$,  applying the H\"older inequality, changing variables
again, and using the first inequality in \eqref{expii} (which holds for
any $y \in J_{k,t}$)
give
\begin{align*}
&\left(\int_I\frac{|\psi_k(y)|^{\tilde p}}{|(f_t^{k})'(y)|^{\tilde p-1}}\,dy\right)^{1/\tilde p}
\le
\left(\int_{\supp(\psi_k)}\frac1{|(f_t^{k})'(y)|^{(\tilde p-1)q'}}\,dy\right)^{1/\tilde pq'}
\|\psi_k\|_{L^{\tilde pp'}(I)}\\
&\,\,\le \sum_{\zeta=\pm}
\left(\int \frac1{|(f_t^{k})'(f^{-k}_{t,\varsigma}(x))|^{(\tilde p-1)q'+1}}\,dx\right)^{\frac{1}{\tilde pq'}}
\|\psi_k\|_{L^{\tilde pp'}(I)}\\
&\,\, \le 
\frac{C}{(|(f^{k-1})'(c_{1,t})|^{1/2})^{(1-1/\tilde p)+1/(\tilde p q')}} 
\left(\int_I \frac{\chi_{k,t}(x)}{\sqrt{|x-c_{k,t}|}^{(\tilde p-1)q'+1}}\,dx\right)^{\frac{1}{\tilde pq'}}
\|\psi_k\|_{L^{\tilde pp'}(I)} .
\end{align*}
Now, if 
\begin{equation}\label{integg}
(\tilde p-1)q'+1<2
\end{equation} 
then  $|x-c_{k,t}|^{-\frac{(\tilde p-1)q'+1}{2}}$ is integrable, and we find
$$
\left(\int_I\frac{|\psi_k(y)|^{\tilde p}}{|(f_t^{k})'(y)|^{\tilde p-1}}\,dy\right)^{1/\tilde p}
\le C^2 
\frac{1}{(|(f^{k-1})'(c_{1,t})|^{1/2})^{(1-1/\tilde p)+1/(\tilde p q')}} 
\|\psi_k\|_{L^{\tilde pp'}(I)}
$$
Set $p=\tilde p p'$. Then, $\tilde p=p (1-1/q')$ so that
\eqref{integg} amounts to 
$$q'< \frac{p+1}{p-1}\, ,$$ 
and the condition on $\tilde p$ 
is $\tilde p < p\frac{2}{p+1}$, as announced.
Using
that our choices give
$$
1-\frac{1}{p}=1-\frac{1}{\tilde p p'}
= 1-\frac{1}{\tilde p} +\frac{1}{\tilde p q'}
\, ,
$$
we find (by definition of $\rr$ and $\Lambda_t$)
$$
\frac{\lambda^{k(1-\rr)}}
{|(f^{k-1})'(c_{1,t})|^{[(1-1/\tilde p)+1/(\tilde p q')]/2}}
\le C\frac{\Lambda_{t}^{k\frac{p-1}{2p}}}
{|(f^{k-1})'(c_{1,t})|^{[(1-1/\tilde p)+1/(\tilde p q')]/2}}\le C C_t\,,
$$
which concludes the proof of the lemma. 
\end{proof}

In order to define the transfer operator $\widehat \LL_t$, we 
introduce smooth {\it cutoff functions} $\xi_{k,t}$ defined as follows. 
The smoothness of the cutoff function is due to the fact that the functions in $\BB_t$ 
are smooth (we want this smoothness to be preserved when applying the 
transfer operator defined below).
For each $k \ge 0$, let $\xi_{k,t} : I \to [0,1]$ be a $C^\infty$ function, with 
\begin{align*}
\supp(\xi_{0,t})= [c- \delta,c+\delta]\, , \qquad 
\xi_{0,t}|_{[c-\frac{\delta}{2},c+\frac{\delta}{2}]} \equiv 1\, ,
\end{align*}
while  for $k\ge1$ we set $\xi_{k,t}\equiv1$ if $I_{k+2,t}=\emptyset$, 
and,  otherwise we assume
\begin{itemize}
\item $\xi_{k,t}$ is unimodal,
\item $\supp(\xi_{k,t})=J_{k+1,t}$,
\item $\xi_{k,t}|_{\cup_{\varsigma\in\{+,-\}}
f_{t,\varsigma}^{-(k+1)}[c_{k+1,t}-(k+1)^{-\beta}/(2L^3),c_{k+1,t}+(k+1)^{-\beta}/(2L^3)]}\equiv1$,
\item $\sup|\partial_x^j\xi_{k,t}(x)|\le C|J_{k+1,t}|^{-j}$, for $j=1,2,3$.
\end{itemize}
(The last property we assume holds also for $k=0$.)
Note that  $\xi_{k,t}(y)>0$ if and only if $\hat f_t (f_t^{k}(y),k)\in B_{k+1,t}\times(k+1)$. 
Further, observe that if $\xi_{k,t}\not\equiv1$, then 
$f_t^{k+2}(J_{k+2,t})$ is adjacent to the boundary of $B_{k+2,t}$ from which 
follows that $|f_t^{k+2}(J_{k+2,t})|\ge (k+2)^{-\beta}L^{-3}/2$.
Hence, we derive similarly as in the estimate \eqref{decaybound} above 
that for some constant $C\ge1$
\begin{equation}
\label{eq.decbelow}
|J_{k+1,t}|\ge|J_{k+2,t}|\ge C^{-1}k^{-\beta/2}|(f_t^{k+1})'(c_1)|^{-1/2}\,,
\qquad\text{if }I_{k+2,t}\neq\emptyset\,.
\end{equation}
This will give the estimate \eqref{kappa'} below.

\begin{definition}[Transfer operator]\label{opp}
The transfer operator $\widehat \LL_t$ is defined for  $\hat \psi \in \BB_t$ by
\begin{equation}\label{a}
(\widehat \LL_t  \hat \psi )_k (x)=
\begin{cases}
\frac{\xi_{k-1,t}(x)}{\lambda} \cdot\psi_{k-1} (x)& k \ge 1 \,, \\
\sum_{j \ge 0, \varsigma\in \{+,-\}}
\frac{\lambda^j (1- \xi_{j,t}(f^{-(j+1)}_{t,\varsigma}(x)))} {|(f_t^{j+1})'(f^{-(j+1)}_{t,\varsigma}(x))|}\cdot
\psi_j(f^{-(j+1)}_{t,\varsigma}(x))
& k=0 \, .
\end{cases}
\end{equation} 
\end{definition}
 
Note that some $j$-terms in the sum for  $(\widehat \LL_t \hat \psi)_0 (x)$ vanish, in particular, 
for all $1\le j <H_0$, because of our choice of small $\delta$.
If  $0<\xi_{j,t}(y)<1$, 
then $y$ will contribute to  both
$(\widehat \LL_t \hat \psi)_{j+1} (y)$
and 
$(\widehat \LL_t \hat \psi)_0 (f^{j+1}(y))$. In other words,
the transfer operator just defined 
is  associated to a multivalued (probabilistic-type)
tower dynamics.
For this multivalued dynamics, some points may fall from the tower a little
earlier than they would for $\hat f_t$. However, the conditions
on the  functions $\xi_{k,t}$ guarantee that they do not
fall {\it too} early. More precisely, if we define ``fuzzy''
analogues of the intervals $I_{k,t}$ and $J_{k,t}$
from \eqref{ints} as follows
\begin{equation}\label{tildeint}
\widetilde I_{k,t}:=
\{ x \in I \mid \xi_{k,t}(x)<1 \, , \, \xi_{j,t}(x)>0\, , \forall
0\le j < k \}\, ,
\quad
\widetilde J_{k,t}=\{c\}\cup _{j\ge k+1} \widetilde I_{k,t}\, , 
\end{equation}
then we can replace $I_{k,t}$ and $J_{k,t}$ by $\widetilde I_{k,t}$ in the
previous estimates, in particular in Lemma~\ref{rootsing}.
Indeed, just observe that if a point ``falls'' according to our fuzzy
dynamics, it would have fallen for some choice of intervals $\widetilde B_{k,t}$ so that
$$
[c_{k,t}-k^{-\beta}/(2L^3),c_{k,t}+k^{-\beta}/(2L^3)]
\subset \widetilde B_{k,t}\subset B_{k,t}
\, .
$$ 
Since we can apply Lemma~\ref{rootsing} to the fuzzy intervals, we 
can combine \eqref{eq.decbelow} with \eqref{expii}, \eqref{rootsing2}, and 
\eqref{rootsing3}, and it follows immediately from the conditions on $\xi_{k,t}$ that there 
is a constant $\tilde C\ge1$ such that
\begin{align}\label{kappa'}
&\|\xi_{k,t}\circ f^{-(k+1)}_{t,\pm}\|_{C^1}\le \tilde C k^{\beta}\, ,\ \ 
\|\xi_{k,t}\circ f^{-(k+1)}_{t,\pm}\|_{C^2}\le \tilde C k^{\max(1+2\alpha+\beta,2\beta)},\\
\nonumber
&\|\xi_{k,t}\circ f^{-(k+1)}_{t,\pm}\|_{C^3}\le \tilde C k^{\max(1+4\alpha+\beta,3\beta)}
\, , \quad\text{for all}\ \ k\ge1\,.
\end{align}
(This is the polynomial analogue of condition
 \cite[(75)]{BS}; the case $j=3$ is used together with \eqref{rootsing2},
\eqref{rootsing3} in Appendix~\ref{misc2}.)  

\begin{remark}[Overlap control]
\label{overlap}
In contrast to the intervals $I_k$, the
intervals $\widetilde I_k$ do not have pairwise disjoint interiors. 
Nevertheless, it follows from the first paragraph in the proof of Lemma~\ref{bottomok}
(see in particular \eqref{overlap'})
 that if $L$ is large enough (and thus $\delta$ small
 enough), we may choose the cutoff functions
$\xi_k$ so that for each $k$, the cardinality
of those $\widetilde I_j$, $j\neq k$, whose interiors intersect the interior of
$\widetilde I_k$ is bounded  by $2$. In other
words each set $\{x\in \supp \, \xi_k(x) \mid \xi_k(x)\ne 1 \}$ is contained in $I_{k+2}$, and hence these
sets are disjoint. 
(This overlap control is used to get the Lasota--Yorke estimate at the heart
of Proposition~\ref{mainprop}. The fact that the overlap is at most two
is  used to get
a good control in \eqref{eq.eigencom2}, which is
essential for Theorem ~\ref{t.main3}.)
\end{remark}

Now, if we introduce the ordinary (Perron--Frobenius) transfer operator
$$
\LL_t : L^1(I)\to L^1(I)\, ,\qquad
\LL_t \varphi(x)=\sum_{f_t(y)=x} \frac{\varphi(y)}{|f_t'(y)|}\, ,
$$
then we have
\begin{equation}\label{commute}
\LL_t \Pi_t (\hat \psi)=\Pi_t(\widehat \LL_t\hat \psi)\, ,
\, \, \forall \hat \psi \in \BB_t^{L^1}\, .
\end{equation}
(See, e.g., \cite{BS} below equation (78).)
In particular, if $\widehat \LL_t\hat \phi=\hat \phi$
then $\LL_t\Pi_t(\hat \phi)=\Pi_t(\hat \phi)$.

Set $w(x,k)=\lambda^k$, for $x\in I$ and $k\ge0$, and define $\nu$ to be the nonnegative measure on 
$\cup_{k\ge0}I\times\{k\}$ whose 
density with respect to Lebesgue is $w(x,k)$.

\begin{proposition}[Spectral properties of $\widehat \LL_t$]\label{mainprop}
Let $f_t$ be good for  parameters $\lambda_c$, $H_0$, $\alpha>1$, $\rho$, $C_0$.  Choose
$\delta>0$ small,  $\beta>\alpha+1$ and $\lambda>1$  as in \eqref{48}. 
Then the operator $\widehat \LL_t$ is bounded on $\BB_t$, and for any
\begin{equation}\label{theta0}
1 < \Theta_0 <\min( \frac{ \lambda_c^{1/2}}{\lambda}\, , \lambda)\, ,
\end{equation}
the  essential spectral radius of $\widehat \LL_t$ on $\BB_t$  is bounded by $\Theta_0^{-1}$.
The spectral radius of $\widehat \LL_t$ on $\BB_t$
is equal to $1$, where $1$ is a simple eigenvalue  for
a nonnegative eigenvector $\hat \phi_t$. If $f_t$ is mixing, then $1$ is the only
eigenvalue of modulus $1$, otherwise the other eigenvalues of
modulus $1$ are simple and located at
roots of unity $e^{2\imath j\pi/P_t}$, $j=0,\ldots , P_t-1$, for $P_t\ge 2$ the renormalisation period of $f_t$.
The  fixed point of the dual of  $\widehat \LL_t$ is $\nu$.
If  $\nu(\hat \phi_t)=1$, then $\phi_t:=\Pi_t (\hat \phi_t)$
is the density of the unique absolutely continuous $f_t$-invariant probability measure. 
Finally,\footnote{We
use here that $f_t$ is  $C^4$ and not just $C^3$.} $\hat \phi_{t,0} \in W^2_1$, uniformly in the goodness
 (once $\delta$, $\beta$, $L$, and $\lambda$ are fixed).
\end{proposition}

\begin{proof}
The proof is an adaptation
of Propositions~4.10 and 4.11 in \cite{BS} to our fat tower, using the polynomial
recurrence condition. We give it in Appendix~\ref{misc2},
mentioning here only that
the key (Lasota--Yorke) estimate is that there 
exists a constant $C>0$, depending only on the goodness
of $f_t$, $\delta$, and $L$, such that 
\begin{equation}
\label{eq.LaYo}
\|\widehat\LL^n_t(\hat\psi)\|_{\BB_t}\le C\Theta_0^{-n}\|\hat\psi\|_{\BB_t}+C\|\hat\psi\|_{\BB_t^{L^1}}\,,\quad\forall n\ge1\,,
\end{equation}
for all $\hat\psi\in\BB_t$. 
\end{proof}


\subsection{Truncated transfer operators $\widehat \LL_{t,M}$ on the tower}
\label{ss.trunc}

We introduce 
for each $M \ge 0$ the {\it truncation operator}
$\TT_M$ defined  by 
\begin{align}\label{truncc}
\TT_M(\hat \psi)_k=
\begin{cases} \psi_k & k \le M\\
0 & k > M  \, .
\end{cases}
\end{align}
By  definition $\TT_M$
is a bounded operator on  $\BB_t$, with $\|\TT_M\|_{\BB_t}\le 1$ for any
$M$. The {\it truncated transfer operator} $\widehat \LL_{t,M}:\BB_t\to\BB_t$ is the bounded
operator defined by
$$
\widehat \LL_{t,M} = \TT_M \widehat \LL_t \TT_M \, .
$$

The following proposition lists the basic spectral properties of the truncated
transfer operator. For the maximal eigenvector $\hat\phi_t$ of $\widehat\LL_t$ 
given by Proposition~\ref{mainprop} 
we assume always that it is normalised by $\nu(\phi_t)=1$. 
 
\begin{proposition}[Spectral properties of the truncated operator $\widehat \LL_{t,M}$]
\label{truncspec}
For any $t$ which is
good for parameters $\lambda_c$, $H_0$, $\alpha>1$, $\rho$, $C_0$, the essential spectral radius of $\widehat \LL_{t,M}$ acting on 
$\BB_t$ 
is not larger than $\Theta_0^{-1} <1$, where $\Theta_0$ satisfies condition~\eqref{theta0}.

There exists $M_0\ge 1$ (depending only on the goodness
 of $f_t$) so that,
for all $M \ge M_0$,  the operator $\widehat \LL_{t,M}$ has a real
nonnegative maximal eigenfunction $\hat \phi_{t,M}$,
for a simple eigenvalue $\Theta_0^{-1}<\kappa_{t,M}\le1$, and
the dual operator of  $\widehat \LL_{t,M}$ has a 
nonnegative maximal eigenfunction $\nu_{t,M}$. 
If we normalise $\nu_{t,M}$ by 
$\nu_{t,M}(\hat \phi_t)=1$, and $\hat\phi_{t,M}$ by
$\nu(\hat \phi_{t,M})$ (recall that $\nu=\nu_t$), then we have the bounds
$\sup_M\|\nu_{t,M} \|_{(\BB_t^{L^1})^*}\le C_1$,
 $\sup_M\|\hat \phi_{t,M}\|_{\BB_t} \le C_1$, and 
$\sup_M\|\hat \phi_{t,M,0}\|_{W^2_1}\le C_1$, for a constant $C_1$ depending
only on the goodness  of $f_t$ and the $C^4$ norm of $f_t$.

Furthermore, fixing $\upsilon <1$, and setting
\begin{equation}\label{tauM}
\tau_{t,M}=  M^{(\alpha-\beta)/2} \lambda^M |(f_t^M)'(c_{1,t})|^{-1 /2}<1\, ,
\end{equation}
there exists $C_t\ge 1$ 
so that for all $M\ge M_0$
\begin{align}\label{klbounds}
\|\hat \phi_t-\hat \phi_{t,M}\|_{\BB^{L^1}} &\le C_t 
\tau_{t,M}^\upsilon \, , \, \, 
\| \nu- \nu_{t,M}\|_{\BB_t^*} \le C_t  \tau_{t,M}^\upsilon \, , 
\, \, |\kappa_{t,M} -1 | \le C_t \tau_{t,M}^\upsilon\, .
\end{align}
In particular, 
\begin{equation}
\label{eq.kappa}
1\le \kappa_{t,M}^{-M} \le C_t \,,\qquad\forall M\ge M_0\, .
\end{equation}
\end{proposition}

Bootstrapping from the estimates above, Proposition \ref{p.trunc}
will give uniformity of $C_t$ as a function of $t$
and the more precise control on  $\|\hat \phi_t-\hat \phi_{t,M}\|_{\BB^{L^1}}$
and estimates on $\|\hat \phi_t-\hat \phi_{t,M}\|_{\BB^{L^p}}$ 
($p>1$) that are needed
for Theorems~\ref{t.main1} and \ref{t.main3}.

\begin{proof}[Proof of Proposition~\ref{truncspec}]
We adapt the proof of \cite[Lemma ~4.12]{BS} to our
polynomial tower setting, i.e., we apply the perturbation
results of Keller and Liverani \cite{kellerliverani}.
As usual, we assume that $t=0$ and set $f=f_0$. Uniformity in $t$ of the constant $C_1$
will follow from uniformity of the goodness.

The claim about the essential spectral radius can be obtained by
going over the proof of Proposition ~ \ref{mainprop}, checking that it applies
to $\widehat \LL_{M}$ and that the constants
are uniform in $M$. More precisely, there exists $C \ge  1$  so that for all
$n$ and all $M$
\begin{equation}\label{LYM}
\max(\|\widehat \LL^n (\hat \psi)\|_{\BB^{W^1_1}}, 
\|\widehat \LL_{M}^n (\hat \psi)\|_{\BB^{W^1_1}})
\le  C \Theta_0^{-n} \| \hat \psi\|_{\BB^{W^1_1}} + 
C  \|\hat \psi\|_{\BB^{L^1}} \, ,
\end{equation}
and  (note that $\nu(|\widehat \LL^n_M(\hat \psi)|)\le\nu(\widehat \LL^n_M( |\hat \psi|))\le
\nu(\widehat \LL^n( |\hat \psi|))=\nu(|\hat\psi|)=\|\hat\psi\|_{\BB^{L^1}}$; see also \eqref{l1bd} below)
\begin{equation}\label{addref0}
\| \widehat \LL^n  \|_{\BB^{L^1}}\le 1\, ,\quad
\|\widehat \LL_{M}^n  \|_{\BB^{L^1}}\le 1 \, , \quad \forall M \, ,
\forall n\, . 
\end{equation}

To prove the other claims, we shall use that
there exists $C$ so that for all large enough $M$ 
\begin{equation}\label{addref1}
\|(\widehat   \LL - \widehat \LL_{M}) (\hat \psi)\|_{\BB^{L^1}}\le C 
\tau_M \|\hat \psi \|_{\BB^{W^1_1}}\, . 
\end{equation}
This inequality is an easy consequence of
\begin{align*}
&\|\ (\id - \TT_ M) \hat \psi\|_{\BB^{L^1}}
\le C \tau_M \|\hat \psi\|_{\BB^{W^1_1}} \, ,
\end{align*} 
which follows from the estimate $\|\psi_k\|_{L^1}\le|\supp(\psi_k)|\sup|\psi_k|$
combined with \eqref{decaybound}, Proposition~\ref{ubalpha}, and \eqref{sobeb}
(and recalling the $\lambda^k$ weight in the $\BB^{L^1}$ norm).
Indeed, recalling \eqref{decaybound}, we  have, 
\begin{align}
\nonumber
\|(\id-\TT_M)(\hat \psi)\|_{\BB^{L^1}}&=
\sum_{k\ge M+1}\lambda^k\int_I|\psi_{k}(x)|\,dx
\le \sum_{k\ge M+1}\lambda^k|\supp(\psi_{k})|\|\psi_{k}\|_{L^\infty}\\
\label{eq.hm} &\le C^2 \|\hat \psi\|_{\BB}\lambda^M \frac{M^{-\beta/2}}{|(f^M)'(c_{1})|^{1/2}}\sum_{k\ge M+1}
\frac{\lambda^{k-M}}{|(f^{k-M-1})'(c_{M+1})|^{1/2}}\,.
\end{align}
Since $\lambda <\sqrt \rho$  we derive, as in the proof
of Proposition \ref{ubalpha}, that the last sum on the right hand side is bounded
by a constant times $M^{\alpha/2}$.

Then, setting $\PPP(\hat \psi)= \hat \phi \nu(\hat \psi)$
and $\PPP_{M}(\hat \psi)= \hat \phi_M \nu_M(\hat \psi)$ for the respective spectral
projectors of $\widehat \LL$ corresponding to their maximal eigenvalue,
 \cite[Theorem 1, Corollary 1]{kellerliverani} give for any $\upsilon <1$
 a constant $C_t$ so that
 $\|(\PPP_{M}-\PPP)(\hat \psi)\|_{\BB^{L^1}}\le C_t \tau_M^\upsilon \|\hat \psi \|_{\BB}$, which gives,
 taking $\hat \psi=\hat \phi$, that $\|  \hat \phi_M-\hat \phi\|\le  C_t \tau_M^\upsilon$.
 
 We cannot claim yet that $C_t$ is uniform in $t$, because we
 have not proved yet that there exists a neighbourhood of $1$ which intersects the spectrum
$\sigma(\widehat\LL_t)$  of  $\widehat \LL_t:\BB_t \to \BB_t$ only at $z=1$, for all  good $t$
close enough to a good $t_0$. 
 Indeed, a priori, the renormalisation period $P_t$ of $f_t$ could be unbounded, and
the constants 
\begin{equation}\label{thetat}
\theta_t=\sup\{z \in \sigma (\widehat \LL_t)\mid
|z|\ne 1\}<1
\end{equation}
could accumulate at $1$ for (good) $t\to t_0$. Uniformity of $P_t$ and $\theta_t$ when
$t\in \Delta_0$ is the last claim of
 Proposition~ \ref{p.strongnorm}  below.
 \footnote{This holds a fortiori in the easier setting
 of \cite{BS} where $P_t\equiv P$, proving the claim on $\theta_t$  in \cite[\S5.2]{BS}.} 
 Note that $\theta_t$ gives an upper bound on the rate of decay of correlations of $f^{P_t}$
 (for $C^1$ functions, e.g.).

We may also apply the results of \cite{kellerliverani}
to the dual operators (exchanging the roles of the weak and strong norms):
Indeed, for any  $\mu \in (\BB^{L^1})^*$, we have
\begin{equation}\label{dual1}
\sup_{\|\hat \psi \|_{\BB^{W^1_1}}\le 1}
|\mu(\widehat \LL(\hat \psi)-\widehat \LL_M(\hat\psi))|\le C \tau_M \sup_{\|\hat \psi \|_{\BB^{L^1}}\le 1}|\mu(\hat \psi)|\, ,
\end{equation}
while the Lasota--Yorke estimates for $\widehat \LL^*$ and $\widehat \LL_M^*$ are
an immediate consequence of those for  $\widehat \LL$ and $\widehat \LL_M$.
 Setting $\PPP^*( \mu)=  \nu \mu(\hat \phi)$
and $\PPP^*_{M}( \mu)=  \nu_M  \mu(\hat \phi_M)$,   \cite[Theorem 1, Corollary 1]{kellerliverani}  give
\begin{equation}\label{dual2}
\sup_{\|\hat \psi\|_{\BB}\le 1}|[(\PPP^*_{M}-\PPP)(\mu)] (\hat \psi)|
\le C_t  \tau_M^\upsilon  \sup_{\|\hat \psi\|_{\BB^{L^1}}\le 1}| \mu (\hat \psi)|
\, ,
\end{equation}
and, recalling our normalisation,  we may apply the above bound to  $ \mu=\nu $.
Altogether, this gives  the bounds \eqref{klbounds} for  
$\upsilon \in (0, 1)$.
 The bound $\kappa_M\le 1$ follows from the fact that $\kappa_M$ is
an eigenvalue for $\hat \phi_M \in \BB^{L^1}$ and since by \eqref{addref0}
 the spectral radius of  $\widehat \LL_M$
on $\BB^{L^1}$ is bounded by $1$.

It follows from what has been done up to now (and using the fact that $\widehat \LL_M$ is
a nonnegative operator to analyse  its maximal eigenvector, which satisfies
$\hat \phi_M =\lim_{n \to \infty} \frac{1}{n} \sum_{k=0}^{n-1}
\kappa_M^{-k} \widehat \LL_M^k (\hat \phi)$, see also \cite[pp. 933-935, Thm~27]{karlin}) that 
$$\sup_M\|\hat \nu_{t,M} \|_{(\BB_t^{L^1})^*}\le C_1\, ,
\quad \sup _M \|\hat \phi_M\|_{\BB^{W^1_1}}\le C_1\, , $$ 
uniformly in $t$.
To show $\sup_M\|\hat \phi_{t,M,0}\|_{W^2_1} \le C_1$,
we proceed  like when proving  the analogous statement
of Proposition ~ \ref{mainprop} in Appendix~\ref{misc2},
and we get uniform bounds in $M$ and $t$. 
\end{proof}

Recall the notion \eqref{bd2} of admissible pairs $(M,t)$.
In the following lemma, we use the freedom  in the choice of 
the intervals $B_{k,t}$ and cutoff functions $\xi_{k,t}$ in order to, loosely speaking, 
identify the towers up to some level $M=M(t)$ for $t$ close to $t_0=0$ .
The result is a counterpart to Proposition~5.9 in \cite{BS}. The difference 
with the horizontal case there
is that, in our present transversal case, the distance $|c_{k,t}-c_{k}|$ grows 
like $|t||(f^k)'(c_1)|$, i.e., exponentially fast.  

\begin{lemma}[Identical $M$-truncated towers for $f$ and $f_t$]\label{bottomok}
Let $f=f_0$ be good for  parameters $\lambda_c$, $H_0$, $\alpha>1$, 
$\rho$, $C_0$.
If the constant $L>1$ in the definitions of the tower and
of the cutoff functions $\xi_k$ is  sufficiently large, then 
we can choose a tower $\hat f : \hat I \to \hat I$, cutoff functions $\xi_k$,  and
a transfer operator $\widehat\LL$ 
for $f$, and $\epsilon>0$  such that for any $M \ge 1$,
and for any $t\in(-\epsilon,\epsilon)$ which is good for the same
parameters and so that
\begin{equation}\label{tcond}
|(f^{k-1})'(c_1)||t|\le k^{-\beta}\,\quad\forall1\le k\le M\,,
\end{equation}
one can construct the tower $\hat f_t:\hat I_t\to\hat I_t$ and the transfer
operator $\widehat \LL_t$ such that \begin{equation}
\label{eq.bottom}
J_{k+1,t}=J_{k+1}\,\quad\text{and}\quad\xi_{k,t}= \xi_{k}\, , \, \, \forall0\le k\le M-1 \, .
\end{equation}
\end{lemma}

\begin{proof}
Let $\epsilon>0$ be so small as in Lemma~\ref{rootsing}, and take $L>\max(2C^3,4)$ where $C$ is 
given by Lemma~\ref{cd} and 
Lemma~\ref{rootsing}. 
(Observe that the constant $C$ in Lemma~\ref{cd} and 
Lemma~\ref{rootsing}, respectively, does not depend on $L$ and on $\delta$.)
Regarding Remark~\ref{overlap} above, we will be a bit careful in choosing the tower for $f$.
We fix a tower $\hat f:\hat I\to\hat I$ where the levels $E_k$, $k\ge1$, are defined  inductively by 
setting 
$$
B_{k}=\Big[c_{k}-\frac{k^{-\beta}}{2L^2},c_{k}+\frac{k^{-\beta}}{2L^2}\Big]\,,\quad 
\text{if }\frac{k^{-\beta}}{L^2}\le|f^k(J_{k-1})|\le2\frac{k^{-\beta}}{L^2}\,,
$$ 
and
$$
B_{k}=\Big[c_{k}-\frac{k^{-\beta}}{L^2},c_{k}+\frac{k^{-\beta}}{L^2}\Big]\,,\quad\text{otherwise.}
$$ 
This implies
$|f^{k}(I_{k})|/|f^{k}(J_{k-1})|\ge1/2$ whenever $I_k\neq\emptyset$, which in turn implies that 
\begin{equation}
\label{eq.xicut}
\frac{|I_{k}|}{|J_{k-1}|}\ge C^{-2}\sqrt{\frac{|f(I_{k})|}{|f(J_{k-1})|}}\ge C^{-3}/\sqrt{2}\,,
\end{equation}
where in the second inequality we used Lemma~\ref{cd}. 
Since the length of $I_k$ is comparable to the length of $J_{k-1}$, we can now construct 
the cutoff function $\xi_{k-2}$ such that 
\begin{equation}\label{overlap'}
\{x\, |\, 0<\xi_{k-2}(x)<1\}\subset I_k \, ,\qquad \forall k\ge 2\, ,
\end{equation}
and $|\partial_x^r\xi_{k-2}|\le\tilde C|J_{k-1}|^{-r}$, $r=1,2,3$, for some constant $\tilde C$, 
and all the other requirements on cutoff functions are satisfied.

Recall that no point falls from the tower up to level $H(\delta)$. Hence,
the assertion of Lemma~\ref{bottomok} is trivially satisfied for all $0\le k\le H(\delta)-2$
(for arbitrary choices of towers and transfer operators).
Let $H(\delta)-1\le k\le M-1$, and assume that \eqref{eq.bottom} is satisfied for all $0\le j\le k-1$.
Assuming $C$ so large that $C^{-1}|x-c|\le|f_t'(x)|\le C|x-c|$, 
we derive from \eqref{eq.tdistortion1} and Lemma~\ref{cd} that 
\begin{equation}
\label{eq.LLL}
C^{-3}|f^{k+1}(J_{k})|\le |f_t^{k+1}(J_{k})|
\le C^3|f^{k+1}(J_{k})|\,,
\end{equation}
for all $t$ satisfying \eqref{tcond}. 
If $I_{k+1}=\emptyset$ then  
$f^{k+1}(J_{k})\subset B_{k+1}$. Hence, by the upper bound in 
\eqref{eq.LLL}, the choice of $L$, and the definition of $B_{k+1}$,
we have 
$|f_t^{k+1}(J_{k})|\le L^{-1}(k+1)^{-\beta}$. It follows that 
we can choose $B_{k+1,t}$ satisfying \eqref{betaas} and so that  $I_{k+1,t}=\emptyset$.
By definition, $\xi_{k-1,t}= \xi_{k-1}\equiv 1$, which implies \eqref{eq.bottom}. 

If $I_{k+1}\neq\emptyset$ then  the interval
$f^{k+1}(J_{k+1})$ is adjacent to the boundary of $B_{k+1}$. 
Set $B_{k+1,t}=[c_{k+1,t}-b,c_{k+1,t}+b]$, where $b=|f_t^{k+1}(J_{k+1})|$.
By the choice of $L$ and $B_{k+1}$, and using both inequalities in \eqref{eq.LLL}, 
we get
$L^{-3}(k+1)^{-\beta}\le b \le L^{-1}(k+1)^{-\beta}$ from which 
follows that $B_{k+1,t}$ satisfies the condition~\eqref{betaas}. 
By construction $J_{k+1,t}=J_{k+1}$ and $I_{k+1,t}=I_{k+1}$. 
Hence, we can set $\xi_{k-1,t}=\xi_{k-1}$ which concludes the proof of Lemma~\ref{bottomok}.
\end{proof}

\subsection{The  Misiurewicz--Thurston case}
\label{MTcase}

In this section we discuss the modifications in the parameter selection,
tower construction, and transfer operator properties, which will
allow us to get stronger results in the Misiurewicz--Thurston case.

If $f_t$ is Misiurewicz--Thurston,  we shall prove next that
we may take $\alpha=\beta=0$ in the definitions
in Sections ~\ref{s.prel} and ~\ref{uu}. In particular, the sizes of the levels of the tower are uniformly
bounded from below.
(The fact that the size of the levels is bounded
from below will be essential to get the lower bound of Theorem~\ref{t.main3}
 in Section~\ref{s.ndiff}, see e.g. \eqref{largelevel}.)
The first remark is that we can take $\beta=\alpha=0$
in the distortion Lemma ~\ref{cd}, if we
assume that $L$ is large enough so that for all $k\ge 1$, an $L^{-1}$ neighbourhood of $c_k$ 
does not intersect a fixed neighbourhood of $c$. Next, Proposition ~\ref{ubalpha} holds, 
setting $\alpha=0$ (its proof is trivial in the Misiurewicz--Thurston case).
The exponential decay property \eqref{decaybound} also holds, setting $\beta=0$,
and all bounds in Lemma~\ref{sizeIj}  are true, setting $\alpha=\beta=0$, and removing
the remaining factor $j$ in the right hand sides of \eqref{rootsing2} and \eqref{rootsing3}.
All claims in Lemma~\ref{rootsing} are true for $\alpha=\beta=0$, 
up to replacing $k^{-\beta}$ by $\eta$ for some small $\eta>0$ in \eqref{okstart} 
(see also the paragraph containing \eqref{eq.MTeta} in the proof of Lemma~\ref{rootsing}).
If $t$ is Misiurewicz--Thurston, we use the definition \eqref{bd2} of admissible pairs $(M,t)$,
setting $\alpha=\beta=0$. Then,
in the Misiurewicz--Thurston case, if $C_a$ is large enough then 
$|(f^M)'(c_1)||t|\le C_a^{-1}$ implies \eqref{okstart}, and we shall
assume this throughout.

We  now construct the set $\Delta_{MT}$ of sequences $t(n)\to t_0$
which will give Theorem~\ref{t.main3}, by  exhibiting Misiurewicz--Thurston
parameters with uniform average postcritical expansion. 
As usual we assume $t_0=0$ and we remove the $0$ from the notation.
(We shall discuss the
Banach space construction, transfer operator properties, and Lemma~\ref{bottomok} in the Misiurewicz--Thurston case
after the proof of the following lemma.)

\begin{lemma}[Admissible Misiurewicz--Thurston pairs with uniform  postcritical multipliers]
\label{l.existencemt}
Let $f=f_0$ be a mixing Misiurewicz--Thurston map, and assume the family $f_t$ is transversal at
$f_0$.
Let $\ell_0$ denote the  postcritical period of $f$.
If  $C_a$ in \eqref{bd2} defining admissible pairs for
$\alpha=\beta=0$ is large enough then, defining $\Lambda_t$ in \eqref{defL} for
Misiurewicz--Thurston
 maps
by $$\Lambda_t=\lim_{k\to \infty} |(f^k_t)'(c_{1,t})|^{1/k}\, , $$
there exists $C>1$ 
such that 
for each  large enough integer $m$, there exists a Misiurewicz--Thurston map $f_t$
and an integer $M$ with $|M-m|<\ell_0$
such that $(M,t)$ is an admissible pair and 
\begin{equation}
\label{eq.lpt}
C^{-1}\Lambda_t^k\le|(f_t^{k-1})'(c_{1,t})|\le C\Lambda_t^k\,,\qquad \forall\ k\ge1\, ,
\end{equation}
furthermore, setting $\Lambda=\Lambda_0$, we have
\begin{equation}
\label{eq.lptm}
C^{-1}\le\left(\frac{\Lambda_t}{\Lambda}\right)^M\le C\, ,
\end{equation}
and, for $0\le k\le M$, the points
$c_{k,t}$ and $c_k$ are either both local maxima or both local minima for $f_t^k$ and $f^k$, 
respectively, and, 
for $H(\delta)\le k\le M$ and $x\in J_k$, 
if they are local maxima then 
$f_t^k(x)<f^k(x)$, and if they are  local minima then $f^k(x)<f_t^k(x)$.

Finally,   the set $\Delta_{MT}$  of  parameters satisfying the
above properties enjoy uniform goodness constants. 
(This set is infinite countable and accumulates at $t=0$.)
\end{lemma}

\begin{proof}[Proof of Lemma~\ref{l.existencemt}]
For $m$ large, $c_m$ is in the postcritical periodic orbit of $f$. Let $M\ge m$ be minimal such that 
$|(f^{M+i})'(c_1)|>|(f^{M})'(c_1)|$, for all $i\ge1$. Obviously $M-m<\ell_0$.

Since the family $f_t$ is transversal at $f_0$, the map
$f_0$ is not conjugated to the Ulam--von Neumann map, i.e., $c_2=f^2_0(c)$ is not 
equal to the left fixed point $f(0)=0$. (If it were, by  \eqref{eq.transl}, 
for $k$ large, a neighbourhood of $t=0$ in $\EE$ would be mapped by $t\mapsto f_t^k(c)$ 
to a neighbourhood of $0$ in $\real$.  But since $0$ is the left endpoint of $I$, this is not possible.)
Therefore  $c_M$ is contained in the open interval $(c_2,c_1)$. 
Furthermore, there is at least one point $y$ in the interior of $[c_2,c_1]\setminus\{c_i\ |\ i\ge0\}$ 
which is eventually mapped to $c_M$ but such that $f^i(y)\neq c$, for all $i\ge0$. 
\footnote{If $f_0$ is not mixing, we could apply the argument to its deepest renormalisation
$f_0^R$ on a mixing interval $\RR_{c}$, if we assumed that $f_0^R$
is not conjugated to an Ulam--von Neumann map.}

Let then $t_1=t_1(M)>0$ be minimal 
and $t_2=t_2(M)>t_1$ be maximal such that $(M,t)$ is an admissible pair  for all $t\in(t_1,t_2]$. 
Observe for further use that, by definition of admissible pairs, if $m$ is large enough,
\begin{align*}
\frac{t_2-t_1}{t_2}
&=\frac{C_a^{-1}(|(f^{M})'(c_1)|^{-1}-\max_{i\ge 1}|(f^{M+i})'(c_1)|^{-1})}{C_a^{-1}|(f^{M})'(c_1)|^{-1}}\\
&=1-\max_{1\le i \le \ell_0}|(f^i)'(c_M)|^{-1}\,,
\end{align*}
so that,   by the definition of $M$, and for large enough $m_0$,
\begin{align}
\label{eq.t1t2}\inf_{ m\ge m_0} \inf_{M\ge m}\frac{|t_2(M)-t_1(M)|}{|t_2(M)|}>0\, .
\end{align}
By Lemma~\ref{rootsing}, 
for $k\ge H(\delta)$, the sign of $(f^{k-1})'(c_1)\cdot\partial_tf_t^k(c)|_{t=0}$ 
is independent on $k$.
We are therefore in one of the following two situations:
Either for all $k\ge H(\delta)$, if $c_k$ is a local maxima for $f^k$ 
then $\partial_tf_t^k(c)|_{t=0}<0$, while if $c_k$ is a local minima for $f^k$ 
then $\partial_tf_t^k(c)|_{t=0}>0$;  in this  case we set $\Delta(M)=[-t_2,-t_1)$.
Or for all $k\ge H(\delta)$, if $c_k$ is a local maxima for $f^k$ 
then $\partial_tf_t^k(c)|_{t=0}>0$, while if $c_k$ is a local minima for $f^k$ 
and $\partial_tf_t^k(c)|_{t=0}<0$;
in this case, we set $\Delta(M)=(t_1,t_2]$.  
By \eqref{infprodgena} and by the sign assertion just below \eqref{eq.transl} (applied to $x\in J_k$), 
our choice of $\Delta(M)$ implies the assertion 
below property~\eqref{eq.lptm}.

By \eqref{eq.transu}, $|c_j-c_{j,t}|$ is bounded from above by a constant times 
$|(f^{j-1})'(c_1)||t|$, for all $j\le M$ and all $|t|\le t_2$. 
Since $|(f^M)'(c_1)||t|\le C_a^{-1}$ and since the derivative $|(f^{M-j})'(c_j)|$ 
of the Misiurewicz--Thurston map $f$
grows exponentially  in $M-j$,
there exists an integer $\ell_1$, 
which does not depend on $M$, 
such that $c_{M-\ell_1,t}$ is contained in $(c_2,c_1)$, for all $t\in[-t_2,t_2]$. 
On the other hand, by the uniformity of $\ell_1$, the fact that $|(f^{M+1})'(c_1)||t|\ge C_a^{-1}$, 
and by the transversality property~\eqref{eq.transl}, 
there is a constant $\tilde C>1$ 
such that 
$$
\inf_M |\{c_{M-\ell_1,t}\ |\ t\in\Delta(M)\}|\ge \tilde C^{-1}
\inf_M \frac{t_2(M)-t_1(M)}{ t_2(M)}\ge \frac{1}{\tilde C^{2}}\,,
$$
where in the last inequality we used \eqref{eq.t1t2}.
It easily follows that there is a finite collection $\mathcal J$ of open intervals contained in $(c_2,c_1)$ such that for 
each large enough $m$ the interval $\{c_{M(m)-\ell_1,t}\ |\ t\in\Delta(M(m))\}$ 
contains at least one interval of $\mathcal J$.

Since we assumed that
$f_0$ is mixing, the support of its absolutely continuous invariant measure
is $[c_2,c_1]$.
Recall the point $y\in(c_2,c_1)$ constructed in the beginning
of the proof, and let $k_M\ge1$ be minimal such that $f^{k_M}(y)=c_M$.
Fix $m$ large and
let $J\in\mathcal J$ be covered by $\{c_{M(m)-\ell_1,t}\ |\ t\in\Delta(M(m))\}$. By ergodicity of $f$ on the support of the absolutely continuous
invariant measure, 
there exists $k=k(J)$ so that $f^{k}(J)$ contains  $y$.
Therefore, we find a point $x\in J$ and an iterate $j_M\ge1$ 
(with $j_M\le k(J)+k_M$) such that $f^{j_M}(x)=c_M$ 
and so that the points $f^i(x)$, $0\le i\le j_M-1$, avoid a neighbourhood $V$ of $c$. Hence, 
by the implicit function theorem, if 
$|t_2|$ is sufficiently small, for all $t\in[-t_2,t_2]$, 
we find points $x_{M,t}\in J$ and $d_{M,t}\in (c_2,c_1)$ 
(depending differentiably on $t$) ,
such that $d_{M,0}=c_M$,
and $d_{M,t}$ is a (repelling) periodic point for $f_t$ with period $\ell_0$,
while $f_t^{j_M}(x_{M,t})=d_{M,t}$,
and the points $f_t^i(x_{M,t})$, $0\le i\le j_M-1$, avoid a neighbourhood $V'\subset V$ of $c$. 
Since $\{c_{M-\ell_1,t} \mid t \in \Delta(M)\}$ contains  $J$, and  $J$ contains 
the closure of $\{x_{M,t} \mid t \in \Delta(M)\}$,
it follows by the intermediate value theorem
that there exists $t\in\Delta(M)$ such that $c_{M-\ell_1,t}=x_{M,t}$
(with $f_t^{j_M}(x_{M,t})$ the repelling periodic point $d_{M,t}$). 

Since the number of intervals in $\JJ$ is finite and since $c_M$ can attain maximal $\ell_0$ 
different values, it follows that 
$\sup_{J\in \JJ} k(J)<\infty$ and $\sup_Mk_M<\infty$. 
Hence, there is an integer $j_0$ and a neighbourhood 
$U$ of $c$, such that for every large $m$, defining $M(m)$
as above,  there exist $t=t(M)\in\Delta(M)$ 
and $j\le j_0$ such that $c_{M-\ell_1+j,t}=d_{M,t}$ (the repelling periodic
point constructed above by considering a suitable $J$) and $c_{i,t}\notin U$, for all $i\ge1$. 
By construction, $(M,t(M))$ is an admissible pair. 
By the admissible pair condition, 
 $|f'(f^i(c_m))-f_t'(f_t^i(d_{M,t}))|$, $1\le i\le\ell_0$, is bounded from above 
by a constant times $\Lambda^{-M}$. This immediately implies  \eqref{eq.lptm}. 
It remains to show ~\eqref{eq.lpt}, which easily follows
from the distortion  bounds \eqref{eq.tdistortion1}, property~\eqref{eq.lptm}, and the fact 
that $c_{M,t}$ is iterated at most $j_0$ steps to the postcritical periodic orbit 
while these iterations lie outside the neighbourhood $U$ of $c$.

Let $\Delta_{MT}$ be the sequence of Misiurewicz--Thurston parameters 
$$\{ t(M(m))\mid m \ge m_0\}$$
just defined.
Uniformity of goodness constants for $\Delta_{MT}$ is straightforward: 
Take $\lambda_c<\Lambda$ and assume that  $M$  
is sufficiently large. Then, by \eqref{eq.lpt} and \eqref{eq.lptm}, 
we find an integer $H_0$ such that, for all $t\in\Delta_{MT}$, the map $f_t$ is 
$(\lambda_c,H_0)$-Collet--Eckmann. Properties~\eqref{eq.noreturn} and \eqref{eq.deltareturn} 
can be shown  as in the proof of Proposition~\ref{p.uniformc}
(where we might have to intersect $\Delta_{MT}$ with a $\epsilon(\delta)$-neighbourhood of 
$t_0=0$).

This concludes the proof of Lemma~\ref{l.existencemt}.
\end{proof}

\begin{remark}\label{sameside}
The assertion in Lemma~\ref{l.existencemt} just below \eqref{eq.lptm} holds only on one side of $t_0$. 
Hence, the Misiurewicz--Thurston parameters constructed
in Lemma~\ref{l.existencemt} lie either all to the left or all to the right of $t_0$.
This will simplify the proof of the lower 
bound in Section~\ref{ss.lower} considerably by avoiding potential cancellations. 
However, it is quite likely that, by doing careful estimates and by possibly adapting 
the observable $A_D$ below, one would obtain a similar lower bound in 
Section~\ref{ss.lower}, even if the assertion in Lemma~\ref{l.existencemt} below \eqref{eq.lptm} is 
not satisfied (allowing the construction of Misiurewicz--Thurston parameters on both sides).
\end{remark}

\smallskip

\noindent {\bf Banach spaces, transfer operators, the spectral Propositions ~\ref{mainprop} and \ref{truncspec}.}

\smallskip
In the Misiurewicz--Thurston case, we may now construct the Banach spaces, 
the transfer operators $\widehat \LL_t$ and the truncated
transfer operators $\widehat \LL_{t,M}$ exactly like in Sections~\ref{ss.transfer} and ~\ref{ss.trunc},
setting $\alpha=\beta=0$. (The only difference is that \eqref{kappa'} is replaced by
$\|\xi_k\circ f^{-(k+1)}_\pm\|_{C^3}\le C$ for all $k\ge 1$.)
In the definition of the space $\BB^{L^p}$, we take $\Lambda_t$ in 
\eqref{defL} equal to $\Lambda_t$ in \eqref{eq.lpt}.
We claim that for each $p>1$ there exists $C$ so that
\begin{equation}\label{theclaim}
\sup_n \|\widehat\LL^n_t \|_{\BB_t^{L^p}}\le C \, .
\end{equation}
Indeed, 
for $k\ge 1$,   the definition of $\widehat\LL_t$ gives
$$
\lambda^{k\rr} \|[\widehat\LL_t\hat\psi]_k\|_{L^p(I)}\le 
\lambda^{(k-1)\rr} \lambda^{\rr-1}\|\hat\psi_{k-1}\|_{L^p(I)} \, ,
$$
and we only have to consider  $[\widehat\LL_t\hat\psi]_0$.
Using Minkowski's inequality, a change of variable, 
the bound \eqref{expii}, and the definition of $\Lambda_t$,
we find a constant $C$ such that
\begin{align*}
\|[\widehat\LL_t\hat\psi]_0\|_{L^p(I)}
&\le\sum_{j=0 \mbox{ \small or } j\ge H(\delta)}\lambda^j\Big(\int_I\frac{|(1-\xi_j)\psi_j|^p}
{|(f_t^{j+1})'(y)|^{p-1}}dy\Big)^{1/p}\\
&\le C\sum_{j=0 \mbox{ \small or } j\ge H(\delta)}\lambda^j\Lambda_t^{-\frac{p-1}{2p}j}\|\psi_j\|_{L^p(I)}\,.
\end{align*}
By the definition \eqref{eq.mu} of $\rr$, the right hand side is bounded by 
$C\|\hat\psi\|_{\BB_t^{L^p}}$.
We have proved $\|\widehat\LL_t \|_{\BB_t^{L^p}}\le C$. 
The proof  of
$\sup_n \|\widehat\LL^n_t \|_{\BB_t^{L^p}}\le C$ using the
above remarks is  straightforward under the Misiurewicz--Thurston assumption, 
exploiting the overlap control of fuzzy intervals in Remark~\ref{overlap}
(simplifying greatly Appendix~B of \cite{BS}).
The Lasota--Yorke inequality   
\begin{equation}\label{eq.cond2}
\max(\|\widehat \LL_t^n (\hat \psi)\|_\BB, 
\|\widehat \LL_{t,M}^n (\hat \psi)\|_\BB)
\le  C \Theta_0^{-n} \| \hat \psi\|_\BB + 
C   \|\hat \psi\|_{\BB^{L^p} }\, ,
\end{equation}
 for $p>1$  follows from the
Lasota--Yorke inequality  for $\BB$ and $\BB^{L^1}$ and the
embedding of  $\BB^{L^p}$ in $\BB^{L^1}$ given by \eqref{L1Lp}.
Note  that Rellich--Kondrachov gives that the embedding  $\BB^{W^1_1}\subset \BB^{L^p}$
is compact since $p<\infty$ and using \eqref{decaybound}.
Finally,
\begin{align}
\label{eqMT}\|(\id-\TT_{M})\hat \psi\|_{\BB^{L^p}}&=
\sum_{k\ge M+1}\lambda^{k\rr}\left(\int_I|\psi_{k}(x)|^{p}\,dx\right)^{1/p}\\
\nonumber &\le C\sum_{k\ge M+1}\frac{\lambda^k \sup|\psi_k|}{\lambda^{k(1-\rr)}|(f^{k-1}_t)'(c_{1,t})|^{1/2p}}\\
\nonumber &\le C^2 \lambda^M \sum_{k\ge M+1}\frac1{\Lambda_{t}^{(1-1/p)k/2}\Lambda_{t}^{k/2p}}\\
\nonumber   &\le C^3 \lambda^M |(f^M_t)'(c_1)|^{-1/2}\, .
\end{align}
(In the  second inequality, we used the definition \eqref{eq.mu} of 
 $\rr$.  In the last inequality we used 
 the {\it upper bound}
$|(f^{k-1}_t)'(c_{1,t})|\le C \Lambda_t^k$ for all $k$.)

Therefore, the spectral properties stated in Propositions~\ref{mainprop} and ~\ref{truncspec}
(setting $\alpha=\beta=0$)  hold, with the same proofs. In fact,
 by \eqref{theclaim} and \eqref{eq.cond2}, \eqref{eqMT},
we may use the norm of $\BB^{L^p}$ for all $p\ge 1$
as a weak norm.

Finally, if $(M,t)$ is an admissible pair furnished by Lemma~\ref{l.existencemt},
then by \eqref{eq.lpt} and \eqref{eq.lptm}, we have 
$$
C^{-3}\Lambda^k\le|(f_t^{k-1})'(c_{1,t})|\le C^3\Lambda^k\,,\qquad \forall\ 1\le k\le 2M\,.
$$
Using this estimate one can easily adapt the proof of Lemma~\ref{bottomok} 
to ensure that the tower of 
$f_t$ coincides with the tower of $f$ 
up to level $2M$ (instead of $M$). Thus, in the 
Misiurewicz--Thurston case we have 
\begin{equation}\label{bottomokMT}
J_{k+1,t}=J_{k+1}\,\quad\text{and}\quad\xi_{k,t}= \xi_{k}\, , \qquad \forall\ 0\le k\le 2M-1 \, .
\end{equation}

\section{Whitney--H\"older upper bounds --- Proof of Theorem ~ \ref{t.main1}}
\label{main1}

\subsection{The main decomposition -- Upper bounds}

For    $f_t$ either good with $\alpha>1$ or
Misiurewicz--Thurston,
we proved in Sections~\ref{ss.transfer} and \ref{MTcase}, respectively,
that the invariant density of $f_t$  can  be written as
$\phi_t =\Pi_t( \hat \phi_t)$, where $\hat \phi_t$ is the nonnegative
and normalised fixed point of $\widehat \LL_t$ on $\BB_{t}$. 
Writing $f=f_{t_0}$ (as usual we assume $t_0=0$ and we 
remove the $0$ from the notation),
as in \cite[\S 6]{BS}, our starting point is the decomposition
\begin{align}\label{decc}
\phi_t - \phi=&\bigl [\Pi_t (\hat \phi_t - \hat \phi_{t, M}) +
 \Pi (\hat \phi_{M}-\hat \phi)\bigr ] +
[\Pi_t ( \hat \phi_{t,M} -\hat \phi_{M}  )]
+[ (\Pi_t -\Pi) (\hat \phi_{M})] \, .
\end{align}
We next state three propositions giving  upper bounds on the three terms
in the  right hand side of  the above
decomposition. The 
proof of Theorem~\ref{t.main1} will easily follow. The upper bounds for
the first two square brackets in \eqref{decc} have
a stronger form in the Misiurewicz--Thurston case, and they will be used 
in combination with Proposition ~\ref{l.lower} below (which gives a lower
bound for the third square bracket in the decomposition \eqref{decc}) to show
Theorem~\ref{t.main3} in Section~\ref{s.ndiff}.

We first discuss the effect of parameter change on the truncated
eigenvalues, i.e., the contribution of the second square bracket in the right hand side of \eqref{decc}.
The following proposition shows
that our choice of admissible pairs $(M,t)$ was indeed optimal:

\begin{proposition}[Strong norm control of  $t \mapsto \hat \phi_{t,M}$ ]\label{p.strongnorm}
If $f=f_0$ is good, with $\alpha>1$,  constructing the tower and transfer
operators as in Sections~\ref{ss.tower}
and ~\ref{ss.transfer}--\ref{ss.trunc} (in particular $\beta >\alpha+1$, and we use
Lemma~\ref{bottomok}),  there exists  $C$ such that for each admissible pair $(M,t)$,
with $t$ good for the same parameters and $|t|$ sufficiently small, 
we have
\begin{equation}
\label{eq.strongnormp1}
\|\hat\phi_{t,M}-\hat\phi_{M}\|_{\BB} \le CM^{\max(2+2\alpha,1+\beta)}|t|^{1/2}\,.
\end{equation}

If $f=f_0$ is Misiurewicz--Thurston, mixing, and transversal, recalling   the tower and transfer
operator construction in
Section~\ref{MTcase}  ($\alpha=\beta=0$, recalling also \eqref{bottomokMT}) and 
the set $\Delta_{MT}$ of Misiurewicz--Thurston parameters 
accumulating at
$0$ given by Lemma~\ref{l.existencemt},
there exists  a
constant $C$ such that for each admissible pair $(M,t)$
with  $t \in  \Delta_{MT}$ and $|t|$ sufficiently small, we have 
\footnote{The proof of \eqref{eq.pmttl} shows that truncating at $M+C\log M$ would be sufficient; 
but  the proof for $2M$ in \eqref{eq.strongnormp2} is not harder.}
for each $p\ge 1$
\begin{equation}
\label{eq.strongnormp2}
\|\hat\phi_{t,2M}-\hat\phi_{2M}\|_{\BB^{L^p}}\le \|\hat\phi_{t,2M}-\hat\phi_{2M}\|_{\BB} \le C|t|^{1/2}\,.
\end{equation}

Finally, in both cases above, the renormalisation period $P_t$ of $f_t$ is not larger
than $P_0$ 
 for all
small enough $t$, good with the same parameters, the constant
$C_t$ from Proposition~\ref{truncspec} is bounded uniformly
in such $t$, and there exists $\Theta_1>1$ so that
for any such $t$ satisfying in addition $P_t=P_0$,  recalling \eqref{thetat},
 we have
$\theta_t<\Theta_1^{-1}$.
\end{proposition}

The proof of Proposition ~\ref{p.strongnorm} is given in Section \ref{s.mtm}.

We next control the effect of truncation, i.e., the contribution of the
terms in the first square bracket of \eqref{decc}.
For any $\upsilon <1$, Proposition ~\ref{truncspec} and the definition of admissible pairs
$(M,t)$ give a constant $C$ so that
$$
\|\hat\phi_{t,M}-\hat\phi_t\|_{\BB^{L^1}}\le C|t|^{\upsilon/2} \, .
$$
The following proposition gives the improvement of the above
bound needed for both our main theorems:

\begin{proposition}[Weak norm control of $M \mapsto \hat \phi_{t,M}$]
\label{p.trunc}
If $f=f_0$ is good, with $\alpha>1$, 
constructing the tower and transfer
operators as in Sections~\ref{ss.tower}
and ~\ref{ss.transfer}--\ref{ss.trunc},
there exists $C$ such that for each admissible pair $(M,t)$,
with $t$ good for the same parameters and $|t|$ sufficiently small, 
we have
\begin{equation}
\label{eq.pmtt}
\max\{ \|\hat\phi_{t,M}-\hat\phi_t\|_{\BB^{L^1}},
\|\nu_{t,M}-\nu_t\|_{\BB^*} \} \le C|t|^{1/2}M^{2+\alpha}
 \, . 
\end{equation}

If $f=f_0$ is Misiurewicz--Thurston, mixing, and transversal, recalling the construction in
Section~\ref{MTcase} (in particular, $\alpha=\beta=0$) and 
the set $\Delta_{MT}$ of Misiurewicz--Thurston parameters 
accumulating at
$0$ given by Lemma~\ref{l.existencemt},
there exists a constant $C$ such that for each admissible pair $(M,t)$
with $t\in \Delta_{MT}$, we have  for any $p>1$,
\begin{equation}
\label{eq.pmttl}
\max\{ \|\hat\phi_{t,2M}-\hat\phi_t\|_{\BB^{L^p}}, \|\nu_{t,2M}-\nu_t\|_{\BB^*}|\}\le C|t|^{1/2}\,,
\end{equation}
\end{proposition}

Proposition~\ref{p.trunc} is proved in Section~\ref{s.mtt}. 
\noindent The last ingredient for the
proof of Theorem ~ \ref{t.main1} is the following elementary but crucial
lemma, which takes
care of the last contribution in \eqref{decc}, i.e., the displacement
of the ``spikes'' (the square root singularities $1/\sqrt{|x-c_{k,t}|}$ in the invariant
densities):

\begin{proposition}[Upper  bounds on spike displacement]
\label{l.upper}
If $f=f_0$ is good, with $\alpha>1$, taking
$\beta >\alpha +1$ and constructing the tower and transfer
operators as in Sections~\ref{ss.tower}
and ~\ref{ss.transfer}--\ref{ss.trunc},
there exists a constant $C$ such that for each admissible pair $(M,t)$,
with $t$ good for the same parameters and $|t|$ sufficiently small, 
and for all $A\in C^{1/2}$, 
$$
|\int_IA(x)(\Pi_t-\Pi)(\hat \phi_M)(x)dx|
\le C|t|^{1/2}\|A\|_{C^{1/2}}\, .
$$
\end{proposition}

\begin{remark}
The proof of Proposition~\ref{l.upper} applied to the Misiurewicz--Thurston setting 
$\alpha=\beta=0$
would give an additional factor $|\log|t||$ in the upper bound, since the size of the $B_k$'s 
does not converge to $0$ when $k\to\infty$. If the observable $A$ is $C^1$, 
this $\log$-factor vanishes.
(See the upper 
bound in Proposition~\ref{l.lower} and its proof.) \end{remark}

\begin{proof}[Proof of Theorem ~ \ref{t.main1}]
Let $\Delta\subset\EE$ be the set given by Proposition~\ref{p.uniformc}. 
For given $\Gamma>4$, we can choose $\alpha>1$ and $\beta>1+\alpha$, 
so that $4<\max(2+2\alpha,1+\beta)\le\Gamma$. For this choice of $\alpha$, 
by Proposition~\ref{p.uniformc}, we find for each $t_0\in\Delta$ a set 
$\Delta_0\subset\Delta$  of good 
parameters having the same goodness constants as $t_0$, and $\Delta_0$ contains $t_0$
as a Lebesgue density point. Recall the constant $\epsilon(\delta)$ in Proposition~\ref{p.uniformc}, 
and the choice of $\epsilon>0$ in Lemmas~\ref{bottomok} and \ref{rootsing}.
Now we can choose $\delta>0$ and $\epsilon(\delta)\ge\epsilon>0$ small enough (and $L>1$ 
in Section~\ref{ss.tower} 
large enough) 
so that all corresponding assertions in Sections~\ref{s.prel}--\ref{uu} and in 
Proposition~\ref{p.strongnorm}--\ref{l.upper} 
hold (with uniform constants) for all $t\in\Delta_0\cap(t_0-\epsilon,t_0+\epsilon)$. 
(In Lemma~\ref{rootsing}, we allow of course $t\in(t_0-\epsilon,t_0+\epsilon)$.) 
Redefining the set $\Delta_0$ as
$\Delta_0\cap(t_0-\epsilon,t_0+\epsilon)$, 
we obtain the set $\Delta_0$ in Theorem~\ref{t.main1}.

In the following assume $t\in\Delta_0$. 
As remarked after the definition \eqref{defproj} of the projection $\Pi_t$, 
we have $\|\Pi_t(\hat \phi_t - \hat \phi_{t, M})\|_{L^1(I)}
\le\|\hat \phi_t - \hat \phi_{t, M}\|_{\BB^{L^1}}$. Thus, we can apply 
Propositions ~\ref{p.trunc} and \ref{p.strongnorm} to
bound the $L^1(I)$ norm of the two first square brackets on the right hand side of \eqref{decc}. 
Regarding the last square bracket we apply Proposition~\ref{l.upper}.
Altogether, we derive
$$
\Big|\int_I A(x)\phi_t(x)dx-\int_I A(x)\phi_{t_0}(x)dx\Big|\le C|t-t_0|^{1/2}
M^{\max(2+\alpha,1+\beta)}\|A\|_{C^{1/2}}\,.
$$
Since, by \eqref{bd2}, $M$ is bounded from above by a constant times $|\log|t-t_0||$ and
since $\max(2+2\alpha,1+\beta)\le\Gamma$, this concludes 
the proof of Theorem~\ref{t.main1}.
\end{proof}

It remains to prove Proposition~\ref{l.upper}.

\begin{proof}[Proof of Proposition~\ref{l.upper}]
\label{elemm}
Let $1\le k\le M$, and focus on the branch $f^{-k}_+$ (the other one is handled in a similar way).
Assume that $c_k$ and $c_{k,t}$ are local maxima for $f^k$ and $f^k_t$, respectively
(the other possibility is treated similarly and left
to the reader).

For $A\in C^{1/2}$, we need to consider
\begin{align}
\label{eq.step2}
\Big|\int_{0}^{c_{k,t}}&\lambda^k A(x)\frac{ \phi_{M,k}(f^{-k}_{t,+}(x))}{|(f^k_t)'(f^{-k}_{t,+}(x))|}\,dx
-\int_{0}^{c_k}\lambda^k A(x)\frac{ \phi_{M,k}(f^{-k}_+(x))}{|(f^k)'(f^{-k}_+(x))|}\,dx\Big|\\
\nonumber
&=\lambda^k\Big|\int_c^1(A(f_t^k(x))-A(f^k(x)))\phi_{M,k}(x)\,dx\Big|\\
\nonumber
&\le\lambda^k\|A\|_{C^{1/2}}\int_c^1|f_t^k(x)-f^k(x)|^{1/2}\phi_{M,k}(x)\,dx
\nonumber
\end{align}
By Proposition~\ref{truncspec} it follows that 
$\sup|\phi_{M,k}|\le\lambda^{-k}\kappa_M^k\sup|\phi_{M,0}|\le C\lambda^{-k}$. 
By the second inequality in \eqref{decaybound}, we have
$|\supp(\phi_{M,k})|\le C|(f^{k-1})'(c_1)|^{-1/2}k^{-\beta/2}$, and for $|t|$ sufficiently small
 we have, by \eqref{eq.tdistortion1},
\begin{align*}
\sup_{x\in\supp(\phi_{M,k})}|f_t^k(x)-f^k(x)|\le C|(f^{k-1})'(c_1)||t|\, .
\end{align*}
It follows that 
$$
\Big|\int_0^1A(x)(\Pi_t-\Pi)(\hat \phi_M)(x)dx\Big|
\le C^3\|A\|_{C^{1/2}}|t|^{1/2}\sum_{k=0}^Mk^{-\beta/2}\,.
$$
Since $\beta/2>1$ this concludes the proof of Proposition~\ref{l.upper}.
\end{proof}


\subsection{The effect of parameter change on $\hat \phi_{t,M}$: Proof of Proposition~\ref{p.strongnorm}}
\label{s.mtm}

The next lemma is the key to  Proposition~\ref{p.strongnorm}:

\begin{lemma}[Strong norm estimates for the maximal eigenvector]
\label{l.strongnorm}
In the setting of
Proposition~\ref{p.strongnorm}, there exists a constant $C$ such that the following holds. If $f_t$, $f$, and $M$ are as in 
\eqref{eq.strongnormp1} then
\begin{equation}
\label{eq.strongnorm1}
\|(\widehat\LL_{t,M}-\widehat\LL_{M})\hat\phi_{M}\|_{\BB} \le CM^{\max(2+2\alpha,1+\beta)}|t|^{1/2}\,.
\end{equation}

If $f_t$, $f$, and $M$ are as in 
\eqref{eq.strongnormp2} then
\begin{equation}
\label{eq.strongnorm2}
\|(\widehat\LL_{t,2M}-\widehat\LL_{2M})\hat\phi_{2M}\|_{\BB} \le C|t|^{1/2}\,.
\end{equation}
\end{lemma}

Note that we cannot apply directly the results of Galatolo and Nisoli \cite{gal}
to deduce Proposition~\ref{p.strongnorm} from the above lemma,
because the eigenvectors are not fixed points. 

\begin{proof}[Proof of Lemma~\ref{l.strongnorm}]
This argument is similar to \cite[App C]{BS}, with the very important difference that
we must now deal with a much larger dominant  term which arises from the
transversality assumption, recall ~\eqref{eq.transl}.  (See also Lemma~\ref{rootsing}
and the comment above it.) 
We provide a detailed proof:

The main part of the proofs of \eqref{eq.strongnorm1} and \eqref{eq.strongnorm2} can be done simultaneously. 
When there is a difference, we shall refer as usual to the setting in \eqref{eq.strongnorm1} as the {\em polynomial case} and 
to the setting in \eqref{eq.strongnorm2} as the {\em Misiurewicz--Thurston case}. 
In the estimates below, the constants $\alpha$ and $\beta$ will appear, and
in the Misiurewicz--Thurston case,
except if otherwise mentioned,  these estimates are to be read by setting $\alpha=\beta=0$. 
Henceforth, let $\bar M:=M$ in the polynomial case and  $\bar M:=2M$ in the 
Misiurewicz--Thurston case. Observe first that 
$$
[(\widehat \LL_{t,\bar M}-\widehat \LL_{\bar M})\hat\phi_{\bar M}]_j\equiv0\,,\quad\text{for all }j\ge1\,.
$$
For $j>\bar M$ this follows immediately by the truncation. 
For $1 \le j\le\bar M$, this follows from the fact that we constructed
the tower for $f_t$ to coincide with the tower for $f$ up to level $\bar M$, 
i.e., $\xi_{j,t}\equiv\xi_j$, for $1\le j\le\bar M$ (see Lemma~\ref{bottomok}) and
recall that in the Misiurewicz--Thurston case we used very special properties of $f_t$ and $f$ in 
order to  obtain identical towers up to the higher level $2M$ (see also \eqref{bottomokMT}).
Henceforth we consider only  level $0$.

For $1\le k\le\bar M+1$, set $\varphi_k=(1-\xi_{k-1})\phi_{\bar M,k-1}$.
In order to prove Lemma~\ref{l.strongnorm}, we have to show that 
the term
\begin{equation}
\label{eq.level0}
\Big\|\partial_x\sum_{\begin{subarray}{c}k=1\\ \varsigma\in\{+,-\}\end{subarray}}^{\bar M+1}\lambda^{k-1}\Big[
\frac{\varphi_k(f^{-k}_{t,\varsigma}(x))}{|(f^k_t)'(f^{-k}_{t,\varsigma}(x))|}
-\frac{\varphi_k(f^{-k}_{\varsigma}(x))}{|(f^k)'(f^{-k}_{\varsigma}(x))|}
\Big]\Big\|_{L^1}\,.
\end{equation}
is bounded above (up to a constant) by $M^{\max(2+2\alpha,1+\beta)}|t|^{1/2}$ in the 
polynomial case and by $|t|^{1/2}$ in the Misiurewicz--Thurston case.
We consider first the indices $k \le M+1$, i.e., 
the terms which correspond to a fall from a level below $M$. 
For the polynomial case, this includes 
 all terms we have to study. Regarding the Misiurewicz--Thurston case, 
the terms corresponding to a fall from levels between $M+1$ and 
$2M$ are easier to deal with, and they are treated at 
the end of this proof. 

Since, for $k\le M+1$ the signs of $(f^k_t)'$ and $(f^k)'$ are identical in the domains we are 
interested in, we can skip writing the absolute values in the following estimates.

Henceforth, let $k\le M$, and consider only 
$\varphi_k$ such that $\varphi_k\not\equiv0$.
In order to estimate \eqref{eq.level0}, recall that $\phi_{\bar M,0}\in W_1^2$ and note 
that, by Fubini and the fundamental theorem of calculus, we have 
\begin{equation}
\label{eq.fubini}
\int_I \partial_x\Big[\frac{\varphi_k(f^{-k}_{t,\varsigma}(x))}{(f^k_t)'(f^{-k}_{t,\varsigma}(x))}
-\frac{\varphi_k(f^{-k}_{\varsigma}(x))}{(f^k)'(f^{-k}_{\varsigma}(x))}\Big] dx
=\int_0^t\int_I \partial_x\partial_s\frac{\varphi_k(f^{-k}_{s,\varsigma}(x))}
{(f^k_s)'(f^{-k}_{s,\varsigma}(x))}dxds\,.
\end{equation}

Observe that if $(x,t)\mapsto \Phi_t(x)\in I$ is a $C^1$ map on $I\times \EE$
so that $x\mapsto \Phi_t(x)$ is invertible, then we have
\begin{equation*}
\partial_t \Phi^{-1}_t(x)|_{t=s}
= -\frac{(\partial_t \Phi_t|_{t=s}) \circ \Phi_{s}^{-1}(x)}
{(\partial_x \Phi_{s}) \circ \Phi_{s}^{-1}(x)}\, .
\end{equation*}
We consider only the branch $f_+^{-k}$. The branch $f^{-k}_-$ is handled similarly. 
For $s \in [0,t]$, we derive
\begin{multline}\label{eq.stars}
\partial_s\frac{\varphi_k(f^{-k}_{s,+}(x))}{(f^k_s)'(f^{-k}_{s,+}(x))}\\
=
\frac{(\partial_sf_s^k)(f^{-k}_{s,+}(x))}{(f^k_s)'(f^{-k}_{s,+}(x))}\partial_x(\varphi_k(f^{-k}_{s,+}(x)))
+\partial_x\frac{(\partial_sf_s^k)(f_{s,+}^{-k}(x))}{(f^k_s)'(f^{-k}_{s,+}(x))}\varphi_k(f^{-k}_{s,+}(x))
\, .
\end{multline}
Setting $y=f^{-k}_{s,+}(x)$, and taking the $x$-derivative we get
\begin{align}
\label{eq.txderivative}
\partial_x\partial_s&\frac{\varphi_k(f^{-k}_{s,+}(x))}{(f^k_s)'(f^{-k}_{s,+}(x))}\\
\nonumber &=
\frac{\partial_sf_s^k(y)}{(f_s^k)'(y)}\partial_x^2(\varphi_k(f^{-k}_{s,+}(x)))
+\frac2{(f_s^k)'(y)}\partial_y\frac{\partial_sf_s^k(y)}{(f_s^k)'(y)}\partial_x(\varphi_k(f^{-k}_{s,+}(x)))\\
\nonumber &\quad\quad+\frac{\varphi_k(y)}{[(f_s^k)'(y)]^2}\partial_y^2\frac{\partial_sf_s^k(y)}{(f_s^k)'(y)}
\end{align}
By the assumptions on the  cutoff
functions $\xi$, including Remark~\ref{overlap},
and since $\hat\phi_{\bar M}$ is the eigenfunction of $\widehat\LL_{\bar M}$ for
the eigenvalue $\kappa_{\bar M}$, we have 
\begin{equation}
\label{eq.eigencom1}
\varphi_k(y)=\lambda^{-k}\kappa_{\bar M}^k(1-\xi_{k-1}(y))\xi_{k'-1}(y)\phi_{\bar M,0}(y)\, ,
\end{equation}
where $1\le k'<k$ is maximal such that $\xi_{k'-1}\not\equiv1$. 
Recall that, by Proposition~\ref{truncspec}, there is 
$C\ge1$ such that for all $\bar M$, we have 
$\max\{\|\phi_{\bar M,0}\|_{L^\infty},\|\phi_{\bar M,0}'\|_{L^\infty}\}
\le\|\phi_{\bar M,0}\|_{W_1^2}\le C$. 
Recall also the property \eqref{kappa'} of the cutoff functions.
(Observe that in order to derive \eqref{kappa'}, we used the estimates 
\eqref{expii}, \eqref{rootsing2}, and \eqref{rootsing3} in Lemma~\ref{sizeIj} 
which we a priori cannot apply for the map $f_s$ since $s$ might not be good. However, 
using the estimates provided by Lemma~\ref{rootsing}, we deduce that these estimates 
still hold for $f_s$, and, thus, the property \eqref{kappa'} holds also 
for $\xi_{k-1}\circ f_{s,\pm}^{-k}$.)
We obtain 
\begin{align}
\label{eq.eigencom2}
\|\partial_x^r(\varphi_k(f^{-k}_{s,+}(x)))\|_{L^1}&\le 
C^2 \lambda^{-k}k^{(r-1)\beta}\,,\quad\text{for } r=0,1\,,\ \text{and}\\
\nonumber
\|\partial_x^2(\varphi_k(f^{-k}_{s,+}(x)))\|_{L^1}&\le 
C^2 \lambda^{-k}k^{\max(1+2\alpha,\beta)}\,,
\end{align}
where the appearance of the factor $k^{-\beta}$ is explained as follows:
The size of the support of $\varphi_k(f^{-k}_{s,+}(x))$ is bounded above by 
a constant times $k^{-\beta}$. 
Hence, all terms in the derivative in \eqref{eq.eigencom2} not containing 
$\phi_{\bar M,0}''$ can be estimated by taking the supremum times the size of the 
support. It is easy to see that the term containing $\phi_{\bar M,0}''$ 
 is bounded by a constant times $\lambda^{-k}$ 
(there is no $k^{\beta}$ factor here). 
In the Misiurewicz--Thurston case, there is only a constant times $\lambda^{-k}$ on 
the right hand sides in \eqref{eq.eigencom2}.

To estimate the $L^1$ norm of \eqref{eq.txderivative}, we must next
consider the factor $\partial_sf_s^k(y)/(f_s^k)'(y)$ and its $y$-derivatives. 
By  \eqref{eq.transu} and \eqref{rootsing1} in Lemma~\ref{rootsing}, we get 
\begin{equation}\label{eq.tx}
\left|\frac{\partial_sf_s^k(y)}{(f^k_s)'(y)}\right|\le C^2 k^{\beta/2}|(f^{k-1})'(c_1)|^{1/2}\,,
\quad\forall\ y\in\supp(\varphi_k)\,.
\end{equation}
Observe that \eqref{eq.tx} and \eqref{eq.eigencom2} give 
a polynomial factor $k^{\max(1+2\alpha+\beta/2,3\beta/2)}$ in the upper bound of the 
$L^1$ norm of the first term on the 
right hand side of \eqref{eq.txderivative} (this explains the corresponding factor in 
the bound of \eqref{eq.gamma4} below).
Regarding the  $y$-derivative, observe first that, by \eqref{eq.partialt}, we have
$$
\frac{\partial_sf_s^k(y)}{(f_s^k)'(y)}=\sum_{j=1}^k\frac{(\partial_sf_s)(f_s^{j-1}(y))}{(f_s^j)'(y)}
=\sum_{j=1}^k\frac{X_s(f_s^j(y))}{(f_s^j)'(y)}\,.
$$
Therefore, we obtain
\begin{equation}
\label{eq.secondline}
\partial_y \frac{\partial_sf_s^k(y)}{(f_s^k)'(y)} =
\sum_{j=1}^{k} X_s'(f_s^j(y))
-\sum_{j=1}^{k}X_s(f_s^j(y))
\sum_{\ell=0}^{j-1}\frac{f_s''(f_s^\ell(y))}{(f_s^{j-\ell})'(f_s^\ell(y))f_s'(f_s^\ell(y))}\, .
\end{equation}
We shall use again the estimates in Lemma~ \ref{rootsing}. For $1\le m\le k$ and $y\in\supp(\varphi_k)$, 
we get, by \eqref{eq.tdistortion1} and \eqref{rootsing1},
\begin{equation}
\label{eq.mks}
\frac1{|(f_s^m)'(y)|}=\left|\frac{(f_s^{k-1})'(f_s(y))}{(f_s^k)'(y)(f_s^{m-1})'(f_s(y))}\right|
\le C^3\frac{k^{\beta/2}|(f^{k-1})'(c_1)|^{1/2}}{|(f^{m-1})'(c_1)|}\,,
\end{equation}
and, by \eqref{infprodgena}, \eqref{eq.tdistortion1},  and \eqref{eq.postcrit1}, 
$$
\frac1{|(f_s^{k-m})'(f_s^m(y))|}=\left|\frac{(f_s^{m-1})'(f_s(y))}{(f_s^{k-1})'(f_s(y))}\right|\le C^5 m^\alpha\,.
$$
The dangerous factors in estimating \eqref{eq.secondline} (even more dangerous when estimating 
the terms \eqref{eq.term1} and \eqref{eq.term2} below) 
are powers of $f_s'(y)$ in the denominator
and factors $(f_s^m)'(y)$ for large $m$ in the numerator. 
The dominant terms on the right hand side of \eqref{eq.secondline} appear when $\ell=0$ 
(summing over $j$ when $\ell=0$ gives a geometrical series).
We derive that there is a constant $\tilde C$ so that 
\begin{equation}
\label{eq.txx}
\left|\partial_y \frac{\partial_sf_s^k(y)}{(f_s^k)'(y)}\right|
\le\tilde Ck^{\beta}|(f^{k-1})'(c_1)|\,.
\end{equation}

Regarding the
$y$-derivative of order two (appearing in the last term in \eqref{eq.txderivative}), observe that 
\begin{align}
\nonumber
&\partial_y^2\frac{\partial_sf_s^k(y)}{(f_s^k)'(y)}=\\
\label{eq.term1}
&\sum_{j=1}^{k} X_s''(f_s^j(y))(f_s^j)'(y)
-\sum_{j=1}^{k} X_s'(f_s^j(y))
\sum_{\ell=0}^{j-1}\frac{f_s''(f_s^\ell(y)) (f_s^\ell)'(y)}{f_s'(f_s^\ell(y))}\\
\label{eq.term2}
&-\sum_{j=1}^{k} X_s(f_s^j(y))
\sum_{\ell=0}^{j-1}\Bigg[
\frac{f_s'''(f_s^\ell(y)) (f_s^\ell)'(y)}
{(f_s^{j-\ell})'(f_s^\ell(y)) f_s'(f_s^\ell(y))}\\
\nonumber
&\quad-\frac{f_s''(f_s^\ell(y))^2 (f_s^\ell)'(y)}
{(f_s^{j-\ell})'(f_s^\ell(y))[f_s'(f_s^\ell(y))]^2}
-\frac{f_s''(f_s^\ell(y))}{f_s'(f_s^\ell(y))}
\sum_{i=\ell}^{j-1}
\frac{f_s''(f_s^i(y)) }{(f_s^{j-i})'(f_s^{i}(y)) f_s'(f_s^{i}(y))}
\Bigg]
\, .
\end{align}
The dominant terms in \eqref{eq.term1} 
appear when $\ell=0$. Summing over $j$ gives the upper bound
$k^{1+\beta/2}|(f^{k-1})'(c_1)|^{1/2}$ (up to a constant). The last expression \eqref{eq.term2} 
contains the largest terms.
The dominant terms appear in the last line when $i=\ell=0$. 
Summing over $j$, which gives a geometrical series, we derive that \eqref{eq.term2} is bounded from 
above by a constant times $k^{3\beta/2}|(f^{k-1})'(c_1)|^{3/2}$. It follows that there is a constant $\tilde C$ so that
$$
\left|\partial_y^2\frac{\partial_sf_s^k(y)}{(f_s^k)'(y)}\right|
\le\tilde Ck^{3\beta/2}|(f^{k-1})'(c_1)|^{3/2}\,.
$$
We have bounded all terms regarding the $L^1$ norm of
\eqref{eq.txderivative} (observe that $|(f_s^k)'(y)|^{-1}\le C\frac{k^{\beta/2}}{|(f^{k-1})'(c_1)|^{1/2}}$, 
by setting $m=k$ in \eqref{eq.mks}).
Recalling \eqref{eq.fubini}, we conclude that 
\begin{equation}
\label{eq.gamma4}
\left\| \lambda^{k-1}\partial_x\left(\frac{\varphi_k(f^{-k}_{t,+}(x))}{(f^k_t)'(f^{-k}_{t,+}(x))}
-\frac{\varphi_k(f^{-k}_{+}(x))}{(f^k)'(f^{-k}_{+}(x))}\right)\right\|_{L^1}
\end{equation}
is bounded above by a constant times 
$k^{\max(1+2\alpha+\beta/2,3\beta/2)}|(f^{k-1})'(c_1)|^{1/2}|t|$, where in the
Misiurewicz--Thurston case the polynomial factor vanishes.
Thus, in the polynomial case, by \eqref{eq.postcrit1}, we get
$$
\|\partial_x[(\widehat\LL_{t,M}-\widehat\LL_M)\hat\phi_M]_0(x)\|_{L^1}
\le2\tilde CM^{\max(2+(5\alpha+\beta)/2,1+(\alpha+3\beta)/2)}|(f^M)'(c_1)|^{1/2}|t|\,.
$$
Applying the admissible pair condition~\eqref{bd2}, it follows that 
$\|(\widehat\LL_{t,M}-\widehat\LL_M)\hat\phi_M\|_{\BB}$ is 
bounded from above by a constant times $|t|^{1/2}M^{\max(2+2\alpha,1+\beta)}$. This proves 
inequality~\eqref{eq.strongnorm1}, i.e., the polynomial case of Lemma~\ref{l.strongnorm}.

In the Misiurewicz--Thurston case, instead of applying \eqref{eq.postcrit1}, 
we can use \eqref{eq.lpt}, and we derive
\begin{align*}
\|\partial_x[(\widehat\LL_{t,M}-\widehat\LL_M)\hat\phi_{2M}]_0(x)\|_{L^1}
&\le C|(f^M)'(c_1)|^{1/2}|t|\sum_{k=0}^{M}|(f^{M-k})'(c_{k+1})|^{-1/2}\\
&\le C^2|(f^M)'(c_1)|^{1/2}|t|\sum_{k=0}^{M}\Lambda^{-(M-k)/2}\le C^4|t|^{1/2}\,,
\end{align*}
where in the last inequality we used the admissible pair condition \eqref{bd2}
for $\alpha=\beta=0$.

It only remains to consider 
the terms involving $\varphi_k$'s for
$M+2\le k\le2M+1$ in the Misiurewicz--Thurston case, i.e., terms in the  level $0$
which correspond to a fall from a level between $M+1$ and $2M$. 
The bounds here are similar but easier than those above.  
We do not look at differences as above, but estimate each term individually.
More precisely,  
\begin{multline*}
\|\partial_x[(\widehat \LL_{t,2M}-\widehat \LL_{2M})(\id-\TT_M)(\hat\phi_{2M})]_0(x)\|_{L^1}\\
\le\|\partial_x[\widehat \LL_{t,2M}(\id-\TT_M)(\hat\phi_{2M})]_0(x)\|_{L^1}
+\|\partial_x[\widehat \LL_{2M}(\id-\TT_M)(\hat\phi_{2M})]_0(x)\|_{L^1}\,.
\end{multline*}
Consider the second term on the right hand side. (The first term is estimated similarly.)
Observe that 
$$
\|\partial_x[\widehat \LL_{2M}(\id-\TT_M)(\hat\phi_{2M})]_0(x)\|_{L^1}
\le \sum_{\begin{subarray}{c} k=M+2\\ \varsigma\in\{+,-\}\end{subarray}}^{2M+1} 
\lambda^{k-1}\|\partial_x\frac{\varphi_k(f_\varsigma^{-k}(x))}{|(f^k)'(f_\varsigma^{-k}(x))|}\|_{L^\infty}\,.
$$
By \eqref{eq.eigencom1} and \eqref{eq.eigencom2}, and observing that, in the 
Misiurewicz--Thurston case, we have
$$
|\partial_x((f^k)'(f_\varsigma^{-k}(x)))^{-1})|\le C|(f^k)'(f_\varsigma^{-k}(x)))|^{-1}\le C^2|(f^{k-1}(c_1)|^{-1/2}\,,
$$ 
(see  for example the computation for the last term in \eqref{eq.secondline} when $j=k$ and $\ell=0$), 
we easily deduce that this sum is bounded from above by a constant times $|(f^M)'(c_1)|^{-1/2}$. 
This, in turn, is bounded by a constant times $|t|^{1/2}$,
by the consequence \eqref{eq.admupper} of the admissible pair condition. This concludes the 
proof of \eqref{eq.strongnorm2}, and hence the proof of Lemma~\ref{l.strongnorm}.
\end{proof}

We will deduce  Proposition~\ref{p.strongnorm} from Lemma ~\ref{l.strongnorm}. The proof is
divided in two parts: We first show the uniform bounds on $P_t$ and  $C_t$, which will be used in the
proof of Proposition~ \ref{p.trunc} in  Section \ref{s.mtt}. Then, 
exploiting Proposition~ \ref{p.trunc} (as we may), we  end the proof of Proposition~\ref{p.strongnorm}
by applying Lemma ~\ref{l.strongnorm}.

\begin{proof}[Proof of Proposition~\ref{p.strongnorm}: First part]
We are going to use again the arguments in \cite{kellerliverani}.  
Recall that there exist $\epsilon >0$, $C\ge 1$ so that for all good $|t|\le \epsilon$ 
(with the same goodness parameters) and all
$M$ we have \eqref{LYM}, \eqref{addref0}, and \eqref{addref1},
in particular the essential spectral radius of $\widehat\LL_{t,M}$ acting on $\BB$ is not larger than 
$\Theta_0^{-1}$ for some $\Theta_0>1$.

We are going to prove \eqref{eq.strongnormp1} and \eqref{eq.strongnormp2} in 
Proposition~\ref{p.strongnorm} simultaneously. 
Fix $(M,t)$ as in the statement of the proposition and let $\bar M:=M$ in the setting of \eqref{eq.strongnormp1} (polynomial recurrence), 
and  $\bar M:=2M$ in 
the setting of \eqref{eq.strongnormp2} (Misiurewicz--Thurston). 
Set
\begin{equation}\label{Qnot}
\widehat \QQ_{t,\bar M}=\widehat \QQ_{t,\bar M}(z)= z-\widehat \LL_{t,\bar M}\, .
\end{equation}
(If $t=0$ we remove $t$ from the notation as usual, writing $\QQ_{\bar M}$ instead of 
$\QQ_{0,\bar M}$.)
We claim that there exist a small circle $\gamma$ centered at 
$1$ and a constant $C\ge1$ 
such that for all  sufficiently large $\bar M$  (or equivalently $|t|\le\epsilon$ sufficiently small)
we have the strong norm control
\begin{equation}
\label{eq.resolventb}
\sup_{z\in\gamma}\|\widehat\QQ_{t,\bar M}(z)^{-1}\|_{\BB}\le C\,,
\end{equation}
and the intersection of the disc $D_\gamma$ bordered by $\gamma$
with the spectrum of $\widehat\LL_{t,\bar M}$  is reduced to $\kappa_{t,\bar M}$,
which is a simple eigenvalue.
To see this, we apply first 
\cite[Theorem~1, Corollary~1]{kellerliverani} (just like  in Proposition~\ref{truncspec}) to the operators
$\widehat \LL$ and $\widehat \LL_M$.
Let $V_{r,\Theta_0}=\{z\in\complex \mid |z|\le\Theta_0^{-1}\text{ or }
\dist(z,\sigma(\widehat\LL))<r\}$. 
For any $r>0$, we find an integer $M_0\ge1$ and a constant $\HH\ge1$ 
such that 
\begin{equation}\label{defHH}
\|\widehat\QQ_M (z)^{-1}\|_{\BB}\le \HH
\, ,\qquad
\forall M\ge M_0\, ,
\forall z\in\complex\setminus V_{r,\Theta_0}\, .
\end{equation}
In particular there exists  a small circle $\gamma$ centered at 
$1$ 
and $M_1\ge 1$ so that
the intersection of  the disc $D_\gamma$ bordered by $\gamma$ with the spectrum of $\widehat\LL_{M}$  is reduced to
the simple eigenvalue $\kappa_{M}$ for all $M \ge M_1$. Then,  if $r$ is small enough,
we have $\gamma\subset\complex\setminus V_{r,\Theta_0}$.

To get \eqref{eq.resolventb}, we will apply \cite[Theorem~1, Corollary~1]{kellerliverani} 
to the operators $\widehat \LL_{t,\bar M}$ and $\widehat\LL_{\bar M}$. 
Since we have a ``moving target'' (just like in \cite{BS}), we must be  careful. 
We shall use  that there are constants $C\ge1$ and $0<\eta<1/2$ such that 
\begin{equation}
\label{eq.movingtarget}
\|(\widehat \LL_{t,\bar M}-\widehat\LL_{\bar M})\hat\psi\|_{\BB^{L^1}}
\le C|t|^\eta\|\hat\psi\|_{\BB}\,\quad\forall\ \hat\psi\in\BB\,.
\end{equation}
(We show \eqref{eq.movingtarget}  at the end of the first part of the proof of this proposition.)
The estimate~\eqref{eq.movingtarget} replaces condition~(5) in \cite{kellerliverani}. 
Recalling \eqref{defHH}, if we assume that $z\in\complex\setminus V_{r,\Theta_0}$
when applying the proof of \cite[Theorem~1]{kellerliverani}, 
the constant $H$ in \cite[Equality (13)]{kellerliverani} is bounded from above by $\HH$. 
Since all other constants are uniform in $t$ and $M$, this implies \eqref{eq.resolventb}.

The uniformity claims on the 
\footnote{The renormalisation period can drop a priori, because
eigenvalues $\ne 1$ on the unit circle could move inside
the open unit disc by perturbation.} renormalisation period $P_t$ of $f_t$ and on $\theta_t$ follow from \eqref{eq.resolventb}, 
using appropriate curves $\gamma_j$. Uniformity of $C_t$ then
follows from the proof of Proposition~\ref{truncspec}.

It remains to prove \eqref{eq.movingtarget}.
We can use the estimates in the proof of Lemma~\ref{l.strongnorm}.
For $1\le k\le\bar M +1$, set $\varphi_k=(1-\xi_{k-1})\psi_{k-1}$. 
We have to estimate the term \eqref{eq.level0}, but without taking 
the $x$-derivative,
since on the left hand side of \eqref{eq.movingtarget}, 
we are only considering the weak norm $\|.\|_{\BB^{L^1}}$. 
For $k\le M+1$, recall \eqref{eq.fubini} (without the $x$-derivative). 
Hence, it is enough to estimate the $L^1$ norm of \eqref{eq.stars}.
Similarly as in \eqref{eq.eigencom2}, and using \eqref{sobeb} below, note that 
$$
\|\partial_x^r(\varphi_k(f^{-k}_{s,+}(x)))\|_{L^1}\le 
C k^{(r-1)\beta}\|\psi_{k-1}\|_{W_1^1}\,,\quad\text{for } r=0,1\,,
$$
where $C\lambda^{-k}$ in \eqref{eq.eigencom2} is replaced by 
$\|\psi_{k-1}\|_{W_1^1}$ (since $\hat\psi$ is not necessarily an eigenvector).
By \eqref{eq.tx} and \eqref{eq.txx}, and by the estimate \eqref{eq.mks} when $m=k$, 
we derive that 
$\|\text{\eqref{eq.stars}}\|_{L^1}$ is bounded from above by a constant times 
$k^{\beta/2}|(f^{k-1})'(c_1)|^{1/2}$. 
In the polynomial case, combined with the consequence \eqref{eq.admupper} 
of the admissible pair condition, 
this gives the bound $M\lambda^M|t|^{1/2}\|\hat\psi\|_{\BB}$ 
(up to some constant) 
of the term \eqref{eq.level0} (without the $x$-derivative). 
In the Misiurewicz--Thurston case, for $M+2\le k\le2M+1$, 
we can apply the same comments as in the last paragraph of the 
proof of Lemma~\ref{l.strongnorm}, and we get the 
upper bound $\lambda^M|t|^{1/2}\|\hat\psi\|_{\BB}$ of the term \eqref{eq.level0} 
(without the $x$-derivative). By \eqref{48} and \eqref{bd2}, we find constants 
$C$ and $0<\vartheta<1/2$ such that $M\lambda^M\le C|t|^{-\vartheta}$
which concludes the proof of \eqref{eq.movingtarget}, and, hence, the proof of the first part of Proposition~\ref{p.strongnorm}.
\end{proof}

\begin{proof}[Proof of Proposition~\ref{p.strongnorm}: Second part]
We can now use the assertions of Proposition~\ref{p.trunc} in order to prove 
\eqref{eq.strongnormp1} and \eqref{eq.strongnormp2}.
Denote by
$\PPP_{t,\bar M}(\hat \psi)=\hat \phi_{t,\bar M}  \nu_{t,\bar M}(\hat \psi)$ 
and $\PPP_t(\hat \psi)=\hat \phi_t \nu(\hat \psi)$
the rank-one spectral projectors corresponding to the (simple) eigenvalues 
$\kappa_{t,\bar M}$ and $1$
of $\widehat \LL_{t,\bar M}$ and $\widehat\LL_t$, respectively (recall 
Propositions~\ref{mainprop} and ~\ref{truncspec}). Since
$
\widehat \QQ_{t,\bar M}^{-1}-\widehat \QQ_{\bar M}^{-1}
= \widehat \QQ_{t,\bar M}^{-1} (\widehat \LL_{t,\bar M} -\widehat \LL_{\bar M})
 \widehat \QQ_{\bar M}^{-1}$ and
 $\widehat \QQ_{\bar M}^{-1} (\hat \phi_{\bar M})= 
\frac{\hat \phi_{\bar M}}{z-\kappa_{\bar M}}$, we have
\begin{align}
\label{idea}
(\PPP_{t, \bar M}-\PPP_{\bar M})(\hat \phi_{\bar M})&=
\hat \phi_{t,\bar M}\nu_{t,\bar M}(\hat \phi_{\bar M})-
\nu_{\bar M}(\hat\phi_{\bar M})\hat \phi_{\bar M} \\
\nonumber &=
- \frac{1}{2 i \pi} \int_{\gamma}  \frac{\widehat \QQ_{t,\bar M}^{-1}(z)} {z-\kappa_{\bar M}} 
(\widehat \LL_{t,\bar M} -\widehat \LL_{\bar M}) \hat \phi_{\bar M}\, dz\,.
\end{align}
Recall that $\kappa_{\bar M}$ tends to $1$ as $\bar M\to\infty$.
Hence, if $|t|$ is sufficiently small, by \eqref{eq.resolventb}, 
we find a constant $\tilde C\ge1$ such that
$$
\|\hat \phi_{t,\bar M}\nu_{t,\bar M}(\hat \phi_{\bar M})-\nu_{\bar M}(\hat\phi_{\bar M}) \hat \phi_{\bar M}\|_{\BB}
\le\tilde C\|(\widehat \LL_{t,\bar M} -\widehat \LL_{\bar M}) \hat \phi_{\bar M}\|_{\BB}\,.
$$
By Lemma~\ref{l.strongnorm}, the right hand side of this inequality 
is bounded (up to a constant) by  
$M^{\max(2+2\alpha,1+\beta)}|t|^{1/2}$ in the polynomial case and by $|t|^{1/2}$ in the 
Misiurewicz--Thurston case. Hence, it is only left to estimate the coefficients 
$\nu_{t,\bar M}(\hat \phi_{\bar M})$ and $\nu_{\bar M}(\hat\phi_{\bar M})$, which can 
be done by Proposition~\ref{p.trunc}: 
Recall the normalisation $\nu_{\bar M}(\hat\phi)=1$ in Proposition~\ref{truncspec}.
Since $|1-\nu_{t,\bar M}(\hat \phi_{\bar M})|\le|1-\nu_{\bar M}(\hat\phi_{\bar M})|
+|\nu_{\bar M}(\hat\phi_{\bar M})-\nu_{t,\bar M}(\hat \phi_{\bar M})|$, 
it is sufficient to estimate
$$
\max(|\nu_{\bar M}(\hat\phi)-\nu_{\bar M}(\hat\phi_{\bar M})|,
|\nu_{\bar M}(\hat\phi_{\bar M})-\nu_{t,\bar M}(\hat \phi_{\bar M})|).
$$
By Proposition~\ref{truncspec} and Proposition~\ref{p.trunc}, the first term is bounded 
(up to a constant) by $M^{2+\alpha}|t|^{1/2}$ in the polynomial case and by 
$|t|^{1/2}$ in the Misiurewicz--Thurston case. 
Regarding the second term (using that $\nu_t=\nu_0$), we have
\begin{align*}
 |\nu_{\bar M}(\hat\phi_{\bar M})-\nu_{t,\bar M}(\hat\phi_{\bar M})|  
&\le |\nu_{t}(\hat\phi_{\bar M})-\nu_{t,\bar M}(\hat\phi_{\bar M})| +
|\nu_{\bar M}(\hat\phi_{\bar M})-\nu_{0}(\hat\phi_{\bar M})| \\
&\le (\|\nu_{t}-\nu_{t,\bar M}\|_{\BB^*}+\|\nu_{\bar M}-\nu_0\|_{\BB^*}) \|\hat\phi_{\bar M}\|_{\BB},
\end{align*}
which is, by Proposition~\ref{p.trunc}, bounded (up to a constant) by  
$M^{2+\alpha}|t|^{1/2}$ in the polynomial case and by 
$|t|^{1/2}$ in the Misiurewicz--Thurston case. 
This concludes the proof 
of Proposition~\ref{p.strongnorm}.
\end{proof}

\subsection{The effect of truncation on $\hat \phi_{t}$: Proof of  Proposition~\ref{p.trunc}}
\label{s.mtt}

To obtain \eqref{eq.pmtt} in Proposition \ref{p.trunc}, we cannot apply the Keller--Liverani perturbation \cite{kellerliverani}
result directly,
 because, in the weak norm of $\BB^{L^1}$ the difference between $\widehat  \LL_t$ and $\widehat  \LL_{t,M}$
is not $M^\Gamma |(f^M)' (c_{1,t})|^{-1/2}$ but $e^{ M/\Gamma} |(f^M)' (c_{1,t})|^{-1/2}$
(due to  the $\lambda^k$ factor
in the definition of the weak norm; cf. \eqref{addref1}). 
In \cite{BS} there were other exponential
losses, but in the polynomially recurrent case
of Theorem ~\ref{t.main1}, we can afford to lose (at most) powers (which give the
logarithmic factor there).
\footnote{In Theorem~\ref{t.main3} we cannot afford to lose a logarithmic factor, however
in this setting we have the flexibility of using a cutting time a bit higher
than $M$ in \eqref{eq.pmttl}.} In view of \eqref{idea},
and inspired by Baladi--Young \cite[\S 2]{BY}, which only requires to control 
the difference between the operator and its perturbation {\it applied to the maximal eigenvector,} we develop a variant
\footnote{C. Liverani has explained to us a simpler but less general variant,
which would give a slightly better bound, where
$\tau |\log \tau|^2$ after \eqref{eq.zequal1} would be replaced by $\tau |\log \tau|$. This would
however not improve our final statement  and it only applies to the mixing case.}
of the Keller--Liverani argument to show 
Proposition \ref{p.trunc}.  (In our application, the maximal eigenvectors are weighted by negative powers
of $\lambda$, which ensures that the $\lambda^k$ factor
in the definition of the weak norm does not create problems.)

\begin{proof}[Proof of Proposition~\ref{p.trunc} ]
 We shall consider $t=0$ and arbitrary
$M$ and show
\begin{equation}\label{red1}
\max(\|\hat  \phi_M - \hat \phi\|_{\BB^{L^1}},
\|\nu_{M} - \nu \|_{\BB^*})
\le C M^2 M^{(\alpha-\beta)/2} |(f^M)'(c_1)|^{-1/2}\, . 
\end{equation}
In the polynomial case, the proof  of    \eqref{eq.pmtt} for $(M,t)$ satisfying the admissible pair condition~\eqref{bd2} 
then follows  from \eqref{eq.admupper} and the fact that the distortion estimates in Lemma~\ref{rootsing} hold 
for 
$(M,t)$ so that 
$|(f_t^M)'(c_{1,t})|$ is comparable to $|(f^M)'(c_1)|$ (see inequality~\eqref{eq.tdistortion1}).
We also use the uniformity in $t$ of 
$P_t\le P$, and 
the constant $C_t<C$ in Proposition~\ref{truncspec} given by Proposition~\ref{p.strongnorm}.
In particular, it follows from the uniform bounds on $P_t$ and $C_t$ and
the proof of Proposition~\ref{truncspec} that the distance between
$\sigma(\widehat\LL_{t,M})\setminus \kappa_{t,M}$ and the point $z=1$ is bounded away from
zero, uniformly in $M$ and $t$. 

In the Misiurewicz--Thurston case, we consider $(N,t)$ admissible, and take
$M=2N$  (with $\alpha=\beta=0$) in \eqref{red1}.
Therefore, there is a constant $C$ such that 
$$
\|\hat\phi_{t,2N}-\hat\phi_t\|_{\BB_t^{L^{p}}}
\le C(2N)^2|(f^{2N}_t)'(c_{1,t})|^{-1/2}
\le C^2 (2N)^2 |(f^{2N-N})'(c_N)|^{-1/2}|t|^{1/2}\,,
$$
where in the last inequality we used \eqref{eq.admupper} for $\alpha=0=\beta$.
Since  $|(f^{2N-N})'(c_N)|^{-1/2}$ is bounded from 
above by a constant times $\Lambda^{-N/2}$ by \eqref{eq.lpt},
this gives  \eqref{eq.pmttl}.


We start with some preliminary bounds. 
Set 
$$
\tau_M=\tau_M(\lambda)= \lambda^M M^{(\alpha-\beta)/2}|(f^M)'(c_1)|^{-1/2}
$$
(recalling that $\alpha=\beta=0$
in the Misiurewicz--Thurston case).
Let $\hat \phi_t$ be the fixed point of $\widehat \LL_t$. 
Clearly, there exists $C\ge1$, which by Proposition~\ref{mainprop} 
depends only on the goodness\footnote{Note that
uniformity in the goodness   holds.} of $t$ (once $\delta$, $\beta$, $L$, and $\lambda$ are fixed), such that
for all good $t$ close enough to $0$
\begin{equation}\label{nicedecay}
\|\phi_{t,k}\|_{L^\infty}\le\lambda^{-k}\|\phi_{t,0}\|_{L^\infty}\le C\lambda^{-k} \|\phi_{0}\|_{L^\infty} \, .
\end{equation} 

 In the polynomial case, injecting \eqref{nicedecay} into \eqref{eq.hm}
for $\hat \psi=\hat \phi_t$, we get
\begin{equation}\label{eq.hmw}
\|(\id-\TT_{M})\hat \phi_t\|_{\BB^{L^1}}\le C \lambda^{-M}  \tau_M \, .
\end{equation}
In the Misiurewicz--Thurston case, 
injecting \eqref{nicedecay} into \eqref{eqMT}
for $\hat \psi=\hat \phi_t$, 
from \eqref{decaybound} and \eqref{nicedecay}, we derive, for
any $p\ge 1$
\begin{align}
\label{eq.hm'}\|(\id-\TT_{M})\hat \phi_t\|_{\BB^{L^p}}
\le  C \lambda^{-M}   \tau_M\, .
\end{align}

Next, using again the fact that $\hat\phi_t$ is the fixed point,  
we derive (recalling Remark ~\ref{overlap} and \eqref{overlap'})
\begin{align}\label{check}
\|\phi_{t,M+k}\|_{W_1^1}&=\lambda^{-M}\|\xi_{k',t}\phi_{t,k}\|_{W_1^1}
=\lambda^{-M}\|\xi_{k',t}'\phi_{t,k}+\xi_{k',t}\phi_{t,k'}\|_{L^1}\\
\nonumber &\le\lambda^{-M}(2\sup|\phi_k|+\|\phi_{t,k'}\|_{L^1})\le
3\lambda^{-M}\|\phi_{t,k}\|_{W_1^1}\, ,
\end{align}
where $k'<M+k$ is maximal such that $\xi_{k',t}\not\equiv1$. 
(Regarding the use of Remark ~\ref{overlap} and \eqref{overlap'}, 
observe that in the proof of Lemma~\ref{bottomok}, property \eqref{overlap'} 
is only guaranteed for the functionms $\xi_{k,0}$, i.e., when $t=0$. However, it is straightforward 
to adapt the construction behind \eqref{overlap'} so that this property holds for all good $t$ 
sufficiently close to $0$.)
Hence,
we get 
\begin{equation}\label{strong}
\|(\widehat \LL_{t,M} -\widehat \LL_t)(\hat\phi_t)\|_\BB
\le C\lambda^{-M}\|\hat\phi_t\|_\BB\,.
\end{equation}
(The estimate for the level $0$ gives much smaller contributions.)

Set $\|.\|:=\|.\|_{\BB^{W_1^1}}$, and 
let 
$|.|:=\|.\|_{\BB^{L^1}}$ in the polynomial case, and $|.|:=\|.\|_{\BB^{L^p}}$  
in the Misiurewicz--Thurston case.
We now move to the main part of the proof. 
For $\Theta_0>1$  given by  Proposition ~\ref{mainprop},
set $C_\theta=1/\log\Theta_0$, which implies that $\Theta_0^{-C_\theta \log(1/\tau_M)}\le \tau_M$.
By Proposition~\ref{truncspec},  for all
large enough $M$, the circle
$\gamma_M$ centered at $1$ 
and\footnote{Use that $ \log (1/\delta)<1/\delta$ for
small $\delta>0$.} of radius 
$$
\frac {\log \log (1/\tau_M)}{C_\theta \log (1/\tau_M)}
$$ 
contains exactly one  (simple) eigenvalue
of $\widehat \LL$ (at $z=1$) and  one  (simple) eigenvalue of $\widehat \LL_M$ (at $z=\kappa_M$).
Recall that $\nu_M$ and $\nu$ are normalised so that $\nu_M(\hat\phi)=\nu(\hat\phi)=1$.
Since $\widehat \QQ^{-1} (\hat \phi)=\frac{\hat \phi}{z-1}$,
we have,  like in \eqref{idea},
\begin{equation}
\label{idea'}
(\PPP_M-\PPP)(\hat \phi)=\hat \phi_M-
\hat \phi=
- \frac{1}{2 i \pi} \int_{\gamma_M}  \frac{\widehat \QQ_M^{-1}(z)} {z-1} 
(\widehat \LL_M -\widehat \LL) (\hat \phi)\, dz\, .
\end{equation}
Then,  for $n\ge 1$ to be chosen later,
inspired by \cite{Ke}, see also \cite[p.147]{kellerliverani}, we write
\begin{align}
\nonumber \widehat \QQ^{-1}_M(z)&=
(z-\widehat \LL_M )^{-1}\PPP_M
+(z-\widehat \LL_M)^{-1} (\id -\PPP_M)
\\
&= \frac{\PPP_M}{z-\kappa_M }+
z^{-n} \widehat \QQ^{-1}_M(z)(\id -\PPP_M )\widehat  \LL_M^n 
+\sum_{j=0}^{n-1}z^{-j-1} (\id -\PPP_M)\widehat \LL_M^j\, .
\end{align}
Recalling the spectral observation made after \eqref{red1}, Proposition ~\ref{truncspec}   implies that
the distance between  $z \in \gamma_M$ and  $\kappa_M$ is 
$\ge (\log \log (1/\tau_M)/(2C_\theta\log (1/\tau_M))$, while the distance between $z \in \gamma_M$  and
the rest of the spectrum of $\widehat \LL_M$ is bounded from below uniformly in $M$.
Therefore,  for $z\in \gamma_M$,  recalling  \eqref{strong},  the Lasota--Yorke estimate \eqref{LYM}  and the uniform
weak-norm bounds \eqref{addref0} and \eqref{theclaim} (using also
$|\widehat \LL_M |\le |\widehat \LL|$) give
\begin{align*}
\nonumber &|\widehat \QQ_M^{-1}(z)
(\widehat \LL_M -\widehat \LL) \hat \phi|\\
\nonumber &\quad \le  |\frac{1}{z-\kappa_M}\PPP_M ((\widehat \LL_M -\widehat \LL) (\hat \phi))|+
|z|^{-n} \| \widehat \QQ^{-1}_M(z)(\id -\PPP_M)\| \| \widehat \LL_M^n  (\widehat \LL_M -\widehat \LL) (\hat \phi)\|\\
&\qquad\qquad
+\sum_{j=0}^{n-1}|z|^{-j-1}|\widehat \LL_M^j (\widehat \LL_M -\widehat \LL) (\hat \phi )|\\
\nonumber
&\quad \le  
\frac{|\hat \phi_M|}{|z-\kappa_M|}
|\nu_M((\widehat \LL_M -\widehat \LL) \hat \phi)|+
\\
\nonumber &\qquad\qquad+ C\biggl  |1-\frac{\log \log (1/\tau_M)}{C_\theta\log (1/\tau_M)}\biggr |^{-n} 
 \biggl [\Theta_0^{-n}  \| (\widehat \LL_M -\widehat \LL) (\hat \phi)\|
 +C | (\widehat \LL_M -\widehat \LL) (\hat \phi)|\biggr ]\\\
 \nonumber&\qquad\qquad\qquad\qquad
 + \sum_{j=0}^{n-1} \biggl |1-\frac{\log \log (1/\tau_M)}{C_\theta\log (1/\tau_M)}\biggr |^{-j-1}  \lambda^{-M} \tau_M \\
  \nonumber &\quad \le2 C \lambda^{-M}\tau_M \frac { C_\theta\log (1/\tau_M)}{\log \log (1/\tau_M)} +
 C \biggl  |1-\frac {\log \log (1/\tau_M)}{C_\theta\log (1/\tau_M)}\biggr |^{-n} \lambda^{-M}  ( \Theta_0^{-n} + 2 \tau_M) \,.
\end{align*}
(We used that $|\nu_{t,M}(\hat \psi)|\le C |\hat \psi|$, uniformly in $t$ and $M$, from Proposition~\ref{truncspec}.)
Then,  taking $n=C_\theta \log(1/\tau_M)$, 
and using  $\lim_{n \to \infty}(1-x/n)^{-n}=e^{x}$,
we find  $C\ge1$ such that for any $z\in \gamma_M$
\begin{equation}
\label{eq.zequal1}
|\widehat \QQ_M^{-1}(z)
(\widehat \LL_M -\widehat \LL) \hat \phi|\le  C C_\theta \lambda^{-M}\tau_M\log (1/\tau_M)\, .
\end{equation}
Multiplying  by $|z-1|^{-1}\le\frac{ \log (1/\tau_M)}{\log \log (1/\tau_M)}$, 
and applying  \eqref{idea'}, 
we have shown  that $|(\PPP_M-\PPP)(\hat \phi)|\le   C\lambda^{-M}|\log\tau_M|^2\tau_M/|\log\log \tau_M|$.

We have proved
\begin{equation}\label{red1"}
|\hat \phi_M  - \hat \phi |\le \lambda^{-M}|\log\tau_M|^2  \tau_M\le C M^2 M^{(\alpha-\beta)/2} |(f^M)'(c_1)|^{-1/2} \, .
\end{equation}
We can apply the same argument to the dual operators
$\widehat \LL_M^*$ and $\widehat \LL^*$, up to exchanging
the role of the weak and the strong norm. Then,
 just like after \eqref{dual2}, 
specialising to   $\mu=\nu$, we get, in the polynomial case, 
 \eqref{red1}, and thus \eqref{eq.pmtt} of Proposition~\ref{p.trunc}.
The Misiurewicz--Thurston case is parallel.
\end{proof}


\section{Whitney--H\"older lower bounds in the Misiurewicz--Thurston case --- Proof of Theorem~\ref{t.main3}}
\label{s.ndiff}

To prove Theorem~\ref{t.main3} using the decomposition \eqref{decc},
 we shall combine the $L^p$ version of the Misiurewicz--Thurston upper bounds
in Propositions  ~\ref{p.strongnorm} and  ~\ref{p.trunc}  with the following statement (the proof 
of which is to be found in Section~\ref{ss.lower}):

\begin{proposition}[Upper and lower bounds on spike displacement in the Misiurewicz--Thurston case]
\label{l.lower}
If $f=f_0$ is transversal, mixing, and
Misiurewicz--Thurs\-ton, then, recalling the set $\Delta_{MT}$ from Lemma~\ref{l.existencemt},
there exists a constant $C>1$ such that for each sufficiently small
$D>0$  and  each admissible pair $(M,t)$, 
where $M$ is sufficiently large and $t\in \Delta_{MT}$, the following holds. There is  $A_D\in C^\infty(I)$, such that 
$\|A_D\|_{L^{\tilde q}(I)}\le2D^{1/\tilde q}$, for all $\tilde q\ge1$, and 
$$
C^{-1}|t|^{1/2}\le\Big|\int_IA_D(x)(\Pi_t-\Pi)(\hat \phi_{2M})(x)\,dx\Big|\le C|t|^{1/2}\,,
$$
where $\hat\phi_{2M}$ is the maximal eigenvector of $\widehat\LL_{2M}$ from Proposition~ \ref{truncspec}
(see Section~ \ref{MTcase}).
\end{proposition}

\begin{proof}[Proof of Theorem~\ref{t.main3}]
Let $p>1$. 
Fix $1 <\tilde p < p\frac{2}{p+1}$, and let $1\le \tilde q<\infty$ be such that $\tilde p^{-1}+\tilde q^{-1}=1$.
By Lemma~\ref{projp} and the H\"older inequality, there is a constant $C=C(p,\tilde p)>1$ such that 
\begin{equation}
\label{eq.2holder}
\Big|\int_IA(x)\Pi_t(\hat\psi)(x)\,dx\big|\le C\|A\|_{L^{\tilde q}(I)}\|\hat\psi\|_{\BB_t^{L^{p}}}\,,\quad\forall\ \hat\psi\in
\BB_t^{L^{p}}\,.
\end{equation}

By the decomposition \eqref{decc} (replacing $M$ by $2M$), 
the estimates \eqref{eq.strongnormp2} in Proposition ~\ref{p.strongnorm} and  \eqref{eq.pmttl}  
in Proposition~\ref{p.trunc}
combined with \eqref{eq.2holder}, 
and Proposition~\ref{l.lower}, we conclude that there is a constant $C>1$ and, for each 
sufficiently small $D>0$, a  $C^\infty$ function $A_D$ satisfying
$$
\Big|\int_I A_D(x)\phi_t(x)\, dx-\int_I A_D(x)\phi(x)\, dx\Big|\ge C^{-1}|t|^{1/2}-2CD^{1/\tilde q}|t|^{1/2}\,,
$$
for all $t\in\Delta_{MT}$ sufficiently close to $0$. 
Since the constant $C$ does not depend on $D$, this implies the 
lower bound in Theorem~\ref{t.main3}.

The upper bound  in Proposition~\ref{l.lower} and the same reasoning as for the lower bound
give $|\int_I A_D\phi_t \, dx-\int_I A_D \phi\, dx |\le C|t|^{1/2}+CD^{1/\tilde q}|t|^{1/2}$, and thus the
upper bound in Theorem~\ref{t.main3}.
\end{proof}


\subsection{Proof of Proposition~\ref{l.lower}}
\label{ss.lower}

We shall need the following property of the eigenvector $\hat\phi_{\bar M}$ 
of the truncated operator $\widehat\LL_{\bar M}$.

\begin{lemma}[Lower bound for truncated maximal eigenvectors]
\label{l.lowerC}
Let $f$ be a  Mi\-siu\-re\-wicz--Thurston  map. Then
there exist a neighbourhood $V$ of $c$ and
constants $M_0\ge 1$ and $C_1\ge1$ such that 
$$
\inf_{x\in V}\phi_{M,0}(x)\ge C_1^{-1}\,,\quad\text{for all $ M \ge M_0$}.
$$
\end{lemma}

\begin{proof}[Proof of Lemma~\ref{l.lowerC}]
Since the density $\phi$ of the absolutely continuous probability
measure is $C^1$ 
away from the (finite) postcritical orbit of $f$, we find an interval $J$  such that 
$\phi|_J\in C^1$ and $\inf_{J}\phi >0$. By ergodicity of $f$ on the
support of the absolutely continuous probability
measure, there exists  $\ell\ge 0$ such that 
$c$ lies in the interior of $f^{\ell}(J)$ (we use that $c$ lies
in the interior of the support of the absolutely continuous probability
measure). Thus, using 
$(\LL^{\ell}\phi)(c)=\phi(c)$,   we find a neighbourhood $V$ around $c$ such that 
$\inf_{x\in V}\phi(x)>0$. Since $\phi=\Pi(\hat\phi)$, and the union over
$k\ge 1$ of the supports of $\phi_k\circ f_{\pm}^{-k}$ 
is disjoint from a neighbourhood of $c$,  this implies  $\inf_{x\in V}\phi_0(x)>0$, 
up to shrinking  $V$.
We conclude by using that  $\phi_{\bar M,0}\in W^2_1$ converges to $\phi_0$  in the $L^1$ topology, and 
$\sup_{\bar M}\|\partial_x\phi_{\bar M,0}(x)\|_{L^\infty}<\infty$.
\end{proof}

\begin{proof} [Proof of the lower bound in Proposition~\ref{l.lower} (simplest case)]
As a warmup, we consider the case
where $f$ has the kneading sequence $RLLR^\infty$, i.e., 
the critical point is mapped after $4$ iterations to the fixed point $c_4$ at the right hand side. 
This is the simplest possible combinatorics. (If the critical point were
mapped after $3$ iterations into the fixed point at the right hand side, the map $f$ would be 
renormalisable 
which is excluded by assumption. In fact, in this case the renormalisation of $f$ would be 
conjugated to the Ulam--von Neumann map which is an obstruction to construct the 
set $\Delta_{MT}$ in Lemma~\ref{l.existencemt}; see also remark below Theorem~\ref{t.main3}.)
In order to have a good mental picture, note that for $f$ with the combinatorics 
as above the construction of the set $\Delta_{MT}$ 
in the proof of Lemma~\ref{l.existencemt} could 
be done so that, for all $t\in\Delta_{MT}$,
the Misiurewicz--Thurston map $f_t$ has the kneading sequence 
$RLLR...RLR^\infty$ (where the middle block of $R$'s has odd length). 
In other words the fourth iteration $c_{4,t}$ lies close to the fixed point of $f_t$, 
where we repel until we are mapped to the left of $c$, whereafter we are immediately 
mapped to the fixed point of $f_t$. 

The observable which will give us a lower bound is concentrated around the fixed point 
$c_4$ of $f$.
For $D>0$ small, let $A_D\in C^\infty([0,1])$ satisfy the following properties. 
\begin{itemize}
\item
$\supp(A_D)\subset[c_4-D,c_4+D]$ and $\|A_D'\|_{L^\infty}\le D^{-1}$;
\item
$A_D$ is monotonously increasing in $[c_4-D,c_4]$ 
and monotonously decreasing in $[c_4,c_4+D]$;
\item
$A_D(x)\ge1/3$ if $x\in[c_4-D/2,c_4+D/2]$.
\end{itemize}

It follows immediately from the construction that 
$\|A_D\|_{L^{\tilde q}(I)}\le2D^{1/\tilde q}$, for all $\tilde q\ge1$.
Let $(M,t)$ be an admissible pair with $|t|$  sufficiently small. 
For simplicity assume that \eqref{eq.transl} holds for all $4\le k\le M$. 
This is for example the case when $f_t$ is the logistic family \eqref{logistic} 
(and $f$ is the map in this family with kneading sequence $RLLR^\infty$). 
This assumption  implies that
the assertion below \eqref{eq.lptm} in Lemma~\ref{l.existencemt} is 
satisfied for all $4\le k\le M$ (this follows immediately from the proof 
of Lemma~\ref{l.existencemt}). 

Recall the definition \eqref{defproj} for
$\Pi_t$ and $\Pi=\Pi_0$. By definition, $\hat \phi_{2M}$ is supported in
levels $k \le 2M$.
We consider first the terms which come from levels of the tower not higher than $M$
(they will give the dominant term for the lower bound), and we show at the end of the argument
that the levels between $M+1 $ and $2M$ give a smaller contribution. 
Since the observable $A_D$ is concentrated around $c_4$, we only have 
to consider iterations $4\le k\le M$. 
Recall that, by Lemma~\ref{l.existencemt}, the iterates
$c_{k,t}$ and $c_k$ are either both local maxima or both local minima for $f_t^k$ and $f^k$,
respectively (observe that this is only true if $k\le M$).

Consider first the case when both are local maxima and focus on the branch $f_+^{-k}$. 
By a simple change of variables, we obtain
\begin{align}
\label{eq.diffs}
\int_{0}^{c_4}&\lambda^k A_D(x)\frac{ \phi_{2M,k}(f^{-k}_+(x))}{|(f^k)'(f^{-k}_+(x))|}\,dx
-\int_{0}^{c_{k,t}}\lambda^k A_D(x)\frac{ \phi_{2M,k}(f^{-k}_{t,+}(x))}{|(f^k_t)'(f^{-k}_{t,+}(x))|}\,dx\\
\nonumber
&=\lambda^k\int_c^1(A_D(f^k(x))-A_D(f_t^k(x)))\phi_{2M,k}(x)\,dx\,.
\end{align}
The assertion just after \eqref{eq.lptm} in Lemma~\ref{l.existencemt} implies that 
$ f_t^k(x)<f^k(x)\le c_4$, for $x\in\supp(\phi_{2M,k})(\subset J_k)$, and since 
$A_D$ is monotonously increasing to the left of $c_4$, it follows that 
the integrand in the right hand side of \eqref{eq.diffs} is everywhere nonnegative. 
Recall the constant $C_a$ in the admissible pair condition~\eqref{bd2}.
For $0<D\ll C_a^{-1}$ let $\widetilde M=\widetilde M(D)\ge 4$ be minimal such that 
$$
|c_4-c_{\widetilde M,t}|>D\,.
$$
By the mean value theorem, \eqref{eq.transu}, the transversality estimate
\eqref{eq.transl}, \eqref{eq.tdistortion1}, \eqref{eq.lpt} (for $t=0$), and the admissibility 
condition~\eqref{bd2}, we derive that $\widetilde M$ exists and $\widetilde M <M$. 
(Observe that $M-\widetilde M$ is of the order $|\log D|$;
we shall not need this fact.)
We claim that there is a constant $C\ge1$ so that
$$
C^ {-1}|(f^{\widetilde M-1})'(c_1)||t|\le|c_4-c_{\widetilde M,t}|\le CD\,.
$$
Indeed, we can argue similarly as just above. By the mean value theorem
the inequality on the left hand side follows by 
transversality estimate~\eqref{eq.transl}, 
while the inequality on the right hand side follows 
essentially from the minimality of $\widetilde M$.
We derive that
\begin{equation}
\label{eq.MMs}
|(f^{\widetilde M-1})'(c_1)|^{-1/2}\ge C^{-1}D^{-1/2}|t|^{1/2}\,.
\end{equation}
 
Since the sizes of the levels $E_k$ of the tower for $f$ are uniformly bounded away from $0$, 
we can choose $D$ so small so that 
$\phi_{2M,k}\circ f_{\pm}^{-k}|_{[c_4-D,c_4]}\equiv
\lambda^{-k}\kappa_{2M}^k(\phi_{2M,0}\circ f_{\pm}^{-k})$, 
for all $k\le2M$.
Hence, by Lemma~\ref{l.lowerC}, we derive that 
$$
\phi_{2M,\widetilde M}(f_{\pm}^{-\widetilde M}(x))
\ge \lambda^{-\widetilde M}\kappa_{2M}^{\widetilde M}C_1^{-1}
\ge \lambda^{-\widetilde M}C^{-1}C_1^{-1}
\,,\qquad\forall x\in[c_4-D,c_4]\,,
$$
where in the last inequality we used \eqref{eq.kappa} combined with the
last claim of Proposition~\ref{p.strongnorm}.
Observe that, by \eqref{expiiu} and \eqref{eq.MMs}, we have 
$$
|f_+^{-\widetilde M}([c_4-D/2,c_4])|
\ge C^{-1}\sqrt{D/2}|(f^{\widetilde M-1})'(c_1)|^{-1/2}
\ge C^{-2}2^{-1/2}|t|^{1/2}\,.
$$
By the definition of $\widetilde M$, 
it follows that $A_D(f_t^{\widetilde M}(x))=0$, 
for all $x\in\supp(\phi_{2M,\widetilde M})$, while $A_D(f^{\widetilde M}(x))\ge 1/3$,
for $x \in f_+^{-\widetilde M}([c_4-D/2,c_4])$.
Therefore, there is a constant $\tilde C>1$, 
so that the following lower bound for \eqref{eq.diffs} holds when $k=\widetilde M$:
\begin{align}
\label{starstar}
&\lambda^{\widetilde M}\int_c^1A_D(f^{\widetilde M}(x))\phi_{2M,\widetilde M}(x)\,dx
\ge\lambda^{\widetilde M}\int_{f_+^{-\widetilde M}([c_4-D/2,c_4])}\frac13\phi_{2M,\widetilde M}(x)\,dx
\\
\nonumber &\ge\tilde C^{-1}|t|^{1/2}\,.
\end{align}
Observe that the constant 
$\tilde C$ does not depend on $D$ by our choice of $\widetilde M$. 

If we consider the 
branches $f_-^{-k}$, or the case 
when $c_{k,t}$ and $c_k$ are both local minima for $f_t^k$ and $f^k$,
respectively, then we derive, similarly as above, that the  term corresponding to 
 \eqref{eq.diffs} is still nonnegative, 
for all $4\le k\le M$.

It remains to show that the terms corresponding to 
$M<k\le2M$ can be neglected. 
Recall that $\sup|\phi_{2M,k}|\le\lambda^{-k}\sup|\phi_{2M,0}|\le C\lambda^{-k}$ 
(see Proposition~\ref{truncspec} and Section~\ref{MTcase}).
We can estimate each term separately, and we get 
\begin{multline}
\label{eq.restaboves}
\Big|\int_IA_D(x)(\Pi_t-\Pi)(\id-\TT_M)(\hat \phi_{2M})(x)dx\Big|\\
\le 2C\sum_{k=M+1}^{2M}\|A_D\|_{L^\infty}
\max_{s\in\{0,t\}}|\supp((A_D\circ f_s^k)\cdot \phi_{2M,k})|\,.
\end{multline}
Since, by \eqref{expii} and \eqref{eq.lptm} (note that
we use the Misiurewicz--Thurston assumption here), $|\supp((A_D\circ f_s^k)\cdot \phi_{2M,k})|$ 
is not larger than a constant multiple of $D^{1/2}\Lambda^{-k/2}$, 
we derive that 
\eqref{eq.restaboves} is bounded from above by a constant times 
$D^{1/2}|(f^M)'(c_1)|^{-1/2}$ which is in turn, by \eqref{eq.admupper}, 
bounded from above by a constant $C\ge1$ times $D^{1/2}|t|^{1/2}$. 

Hence, for $D$ sufficiently small, we conclude
from \eqref{eq.restaboves} and \eqref{starstar} that
$$
-\int_IA_D(x)(\Pi_t-\Pi)(\hat \phi_{2M})(x)dx
\ge\tilde C^{-1}|t|^{1/2}-CD^{1/2}|t|^{1/2}\ge\tilde C^{-1}|t|^{-1/2}/2\,,
$$
whenever $|t|$ is sufficiently small.
\end{proof}

\begin{proof} [Proof of Proposition~\ref{l.lower} (general case)]
The first and main part of the proof  is dedicated to the lower bound, 
the upper bound is given at the end of the proof.
Fix a periodic point $c_j$ in the postcritical orbit of $f$ and, 
for $D>0$ small, let $A_D\in C^\infty([0,1])$  satisfy the following properties:
\begin{itemize}
\item
$\supp(A_D)\subset[c_j-D,c_j+D]$ and $\|A_D'\|_{L^\infty}\le D^{-1}$;
\item
$A_D$ is monotonously increasing in $[c_j-D,c_j]$ and monotonously decreasing in $[c_j,c_j+D]$;
\item
$A_D(x)\ge1/3$ if $x\in[c_j-D/2,c_j+D/2]$.
\end{itemize}

The construction  immediately implies
$\|A_D\|_{L^{\tilde q}(I)}\le2D^{1/\tilde q}$.
Let $(M,t)$ be an admissible pair such that $|t|$ is sufficiently small. 
We can assume that $1\le j\le M$.
As in the argument for the simplest case,
we consider only the terms which come from levels of the tower not higher than $M$ 
(the terms coming from levels between $M$ and $2M$ can be handled
just like around \eqref{eq.restaboves} above).
Since the observable $A_D$ is concentrated around $c_j$, and since, by the admissible 
pair condition~\eqref{bd2}, $c_{k,t}$ stays very close to $c_k$, for $1\le k\le M$, 
by possibly increasing the constant $C_a$ in the admissible pair condition, 
we have 
only to consider iterations $1\le k\le M$ when $c_k=c_j$ (i.e., 
if $c_k\neq c_j$ we shall have no contributions). 
Given such an iteration $k$, 
note that, by the admissible pair condition, 
$c_{k,t}$ and $c_k$ are either both local maxima or both local minima for $f_t^k$ and $f^k$,
respectively (observe that this is only true if $k\le M$).
Consider first the case when both are local maxima and focus on the branch $f_+^{-k}$. 
Recall \eqref{eq.diffs}.
If $H(\delta)\le k\le M$ then, by Lemma~\ref{l.existencemt}, we have 
$ f_t^k(x)<f^k(x)\le c_j$, for $x\in\supp(\phi_{2M,k})$, and since 
$A_D$ is monotonously increasing to the left of $c_j$, it follows that 
the integrand in the right hand side of \eqref{eq.diffs} is nonnegative. 
For each $D>0$ sufficiently small let $\widetilde M=\widetilde M(D)$ be minimal such that 
$$
c_{\widetilde M}=c_j\qquad\text{and}\qquad |c_{\widetilde M}-c_{\widetilde M,t}|>D\,.
$$
The admissible pair condition ensures 
that $\widetilde M<M$ (as in the simplest case), and we can assume that $M$ is large enough 
(making $|t|$ smaller)
so that $\widetilde M> H(\delta)$. (In fact, $M-\widetilde M$ is of the order $|\log D|$.)
Let $\ell_0$ be the period of $c_j$. 
By the mean value theorem and by \eqref{eq.transu} and \eqref{eq.transl} combined with the fact that 
$|(f_t^{\ell_0})'(c_{\widetilde M-\ell_0,t})|\approx \Lambda^{\ell_0}$ 
(see \eqref{eq.lpt} and \eqref{eq.lptm}),
we find 
a constant $C$ such that 
\begin{equation*}
|c_{\widetilde M}-c_{\widetilde M,t}|\le C|c_{\widetilde M-\ell_0}-c_{\widetilde M-\ell_0,t}|\le CD\,,
\end{equation*}
where the last inequality follows from the minimality of $\widetilde M$. 
By the transversality estimate~\eqref{eq.transl}, we have that 
$|c_{\widetilde M}-c_{\widetilde M,t}|$ is bounded from below by a constant times 
$|(f^{\widetilde M-1})'(c_1)||t|$. Hence, we find a constant $\tilde C>1$ such that 
\begin{equation}
\label{eq.MM}
|(f^{\widetilde M-1})'(c_1)|^{-1/2}\ge\tilde C^{-1}D^{-1/2}|t|^{1/2}\,.
\end{equation}
By Lemma~\ref{l.existencemt} and the definition of $\widetilde M$, 
it follows immediately that $A_D(f_t^{\widetilde M}(x))=0$, 
for all $x\in\supp(\phi_{2M,\widetilde M})$. 
Since the sizes of the levels $E_k$ of the tower for $f$ are uniformly bounded away from $0$, 
we can choose $D$ so small so that 
\begin{equation}
\label{largelevel}
\phi_{2M,k}\circ f_{\pm}^{-k}|_{[c_j-D,c_j]}\equiv\lambda^{-k}\kappa_{2M}^k\phi_{2M,0}\circ f_{\pm}^{-k}
\, , \quad
\forall \, 1\le k\le2M\, .
\end{equation}
By Lemma~\ref{l.lowerC}, if $M$ in the admissible pair $(M,t)$ is sufficiently large, it follows that 
$$
\phi_{2M,\widetilde M}\circ f_{\pm}^{-\widetilde M}(x)
\ge \lambda^{-\widetilde M}\kappa_{2M}^{\widetilde M}C_1^{-1}\,,\qquad\forall x\in[c_j-D,c_j]\,.
$$
Thus, by possibly slightly increasing the constant $\tilde C>1$, 
by the construction of $A_D$, we get the following lower bound for \eqref{eq.diffs} when $k=\widetilde M$:
\begin{align*}
&\lambda^{\widetilde M}\int_c^1A_D(f^{\widetilde M}(x))\phi_{2M,\widetilde M}(x)\,dx
\ge\lambda^{\widetilde M}\int_{f_+^{-\widetilde M}([c_j-D/2,c_j])}\frac13\phi_{2M,\widetilde M}(x)\,dx
\\
&\ge\frac13\kappa_{2M}^{\widetilde M}C_1^{-1}|f_+^{-\widetilde M}([c_j-D/2,c_j])|
\ge\frac13\kappa_{2M}^{\widetilde M}C_1^{-1}C^{-1}\sqrt{D/2}|(f^{\widetilde M-1})'(c_1)|^{-1/2}\\
&\ge\tilde C^{-2}|t|^{1/2}\,,
\end{align*}
where in the last inequality we used the lower bounds \eqref{eq.MM} and \eqref{eq.kappa}. 
Observe that the constant 
$\tilde C$ does not depend on $D$. 

Next, consider the case when $c_{k,t}$ and $c_k$ are both local minima for $f_t^k$ and $f^k$,
respectively (and $1\le k\le M$ satisfies $c_k=c_j$). 
By a similar reasoning as  above, we see that \eqref{eq.diffs} is nonnegative, 
for all $H(\delta)\le k\le M$, and we find $H(\delta)<\widetilde M<M$ such that \eqref{eq.diffs}, 
when $k=\widetilde M$,
is bounded from below by a constant (independent on $D$) times $|t|^{1/2}$.
If we consider the 
branches $f_-^{-k}$ then,  we see that 
the terms corresponding  to \eqref{eq.diffs}  are 
still nonnegative for all $H(\delta)\le k\le M$. 

For the lower bound,
it only remains to show that the terms corresponding to 
$0\le k<H(\delta)$ can be neglected.
For $0\le k<H(\delta)$, we see immediately that the absolute values 
of \eqref{eq.diffs}  are bounded from above by 
a constant times
\begin{multline}
\label{eq.m0m}
\lambda^k\|A_D'\|_{L^\infty}|c_k-c_{k,t}|\|\phi_{2M,k}\|_{L^{\infty}}
|\cup_{s\in\{0,t\}}\supp(A_D\circ f_s^k\phi_{2M,k})|\\
\le C\Lambda^{H(\delta)/2}\|\phi_{2M,0}\|_{L^\infty}D^{-1/2}|t|\,.
\end{multline}
Hence, if $|t|$ is sufficiently small, the terms corresponding to $0\le k<H(\delta)$ 
can be neglected.

Regarding the upper bound in Proposition~\ref{l.lower}, we need only to consider the 
terms when $H(\delta)\le k\le M$. For $H(\delta)\le k\le\widetilde M$, 
doing a similar estimate as in \eqref{eq.m0m}, 
we derive that the absolute values of \eqref{eq.diffs}  are 
bounded from above by a constant times
$D^{-1/2}\Lambda^{k/2}|t|$. Thus, by \eqref{eq.MM} and \eqref{eq.lpt}, the sum of these terms is
bounded above by a constant times $|t|^{1/2}$. If $\widetilde M\le k\le M$ (and $c_k=c_j$), 
then only the first term on the left hand side of \eqref{eq.diffs} 
is non-zero. As in the estimate~\eqref{eq.restaboves}, the 
sum over these terms can be estimated from above by a constant times 
$\Lambda^{-M/2}D^{1/2}\Lambda^{(M-\widetilde M)/2}$. By 
the definitions of $\widetilde M$ and $M$, we see that $D\Lambda^{M-\widetilde M}$ is 
bounded from above by a constant. Thus, the contribution of these last terms 
is also of the order $|t|^{1/2}$.
\end{proof}

\begin{appendix}
\section{Proof of the key  estimate Proposition~\ref{ubalpha}}
\label{misc}

In order to prove Proposition~\ref{ubalpha}, we recall useful
notations  from \cite{BS} needed to show
variants of tower estimates in \cite{BS}. Let $f_t$ be a good map, and let $\hat f_t$ be the associated
tower map
as defined in \S~\ref{ss.tower}. To simplify the writing we assume $t=0$ and remove $t$ from the 
notation.
For each $x \in I$ we  define inductively an infinite non decreasing sequence
$$
0=S_0(x) \leq T_1(x) < S_1(x) \leq  \dots < S_i(x) \leq T_{i+1}(x) < S_{i+1}(x) \leq \dots \, , 
$$
with $S_i(x)$ and $T_i(x) \in \mathbb{N}\cup\{\infty\}$ as follows: Put $T_0(x)=S_0(x)=0$ for every $x \in I$. 
Let $i \ge 1$ and assume recursively that $S_j(x)$ 
and $T_j(x)$ have been defined for $j\le i-1$.
Then, we set (as usual, we put $\inf \emptyset = \infty$)
$$
T_i(x)= \inf \{ j\geq S_{i-1}(x) \mid \ |f^j(x)|\le  \delta \} \, .
$$
If $T_i(x)=\infty$ for some $i \ge 1$, then we set  $S_i(x)=\infty$. 
Otherwise, either $f^{T_i(x)}(x)=c$, and then we put $S_i(x)=\infty$,
or  $f^{T_i(x)}(x)\in I_j$ for some $j \ge H(\delta)$, and  we put $S_i(x)=T_i(x)+j$.

Note that if $T_i(x) < \infty$ for some $i\geq 1$ then
\begin{align*}
&\hat f^{j}(x,0) \notin E_0\, , \, \, T_{i}(x)+1\le j \le S_{i}(x)-1 \, , \\ 
&
\hat f^{\ell}(x,0) \in E_0 \, , \,  S_{i-1}(x) \le \ell \le T_{i}(x) \, .
\end{align*}
If $T_{i_0}(x)=\infty$ for  $i_0\ge 1$,  minimal
with this property, then   $\hat f^\ell(x,0) \in E_0$ for all $\ell \ge S_{i_0-1}$
(that is, $|f^\ell(x)|> \delta$ for all $\ell \ge S_{i_0-1}$).

In other words, $T_i$ is the beginning of the $i$-th {\it bound period}
and $S_i-1$ is the end of the $i$-th bound period,
\footnote{Bound period refers to the fact that the orbit is
bound to the postcritical orbit.} and if
$S_i < T_{i+1}$ then
$S_i$ is the beginning of the $i+1$-th {\it free period}
(which ends when the $i+1$-th bound  period starts).

For consistency, we also set
$S_i-T_i=0$ if 
$S_i=T_i=\infty$, and
$T_i-S_{i-1}=0$  if $S_{i-1}=T_i=\infty$,
and, for all $x \in I$, we  put
$
(f^\infty)'(x):=\infty$ and  
$f^\infty(x):=c_1$.

The following lemma gives expansion at the end 
of the free period $T_i-1$ (just before climbing 
the tower), at the end $S_i-1$  of the
bound period (after falling from the tower), and  during the free period
(when staying at level zero).

\begin{lemma}[Tower expansion for good maps]\label{expiii}
Let $f$ be a $(\lambda_c,H_0)$-Collet--Eck\-mann map, polynomially
recurrent of exponent $\alpha>0$ for the same $H_0$.
For every small enough $\delta_0>0$,  if $\delta < \delta_0$,   
$\rho >1$, and $C_0$ are so that
\eqref{eq.noreturn} and \eqref{eq.deltareturn} hold, letting
$S_i(x)$ and $T_i(x)$ be the times associated to the tower for
$\delta$ and $L$,  then
\begin{equation}\label{star'}
|(f^{S_i(x)})'(x)| \geq  \rho^{S_i(x)}\, , \qquad  |(f^{T_i(x)})'(x)| \geq  C_0 \rho^{T_i(x)} 
\, , \qquad \forall x \in I \, , \ \forall i \ge 0\, ,
\end{equation}
and 
$$
|(f^{S_i(x)+j})'(x)| \geq  C_0\delta \rho^{S_i(x)}\rho^j\, , 
\qquad \forall x \in I \, ,\
\forall i \ge 0 \, , \ \forall 0 \le  j < T_{i+1}(x)-S_i(x) \, .
$$
\end{lemma}

\begin{proof}  [Proof of Lemma \ref{expiii}]
Choose $\delta_0>0$ small enough so  that $H(\delta)$, for all $\delta<\delta_0$, 
is large enough so that 
\begin{equation}
\label{eq.deltaL}
C_0 \frac1{CL^{3/2}} j^{-\beta/2} \lambda_c^{\frac{j-1}{2}} \geq \rho^{j} \, ,
\quad \forall j \ge H(\delta) \, ,
\end{equation}
where $C$ is the constant in \eqref{expii}.
The rest of the proof is exactly like for \cite[Lemma~3.5]{BS}, we recall it
for the convenience of the reader:
Let $x \in I$. For any $\ell\ge 1$, the definitions imply
$f^{S_{\ell-1}(x)+k}(x) \in I\setminus [-\delta,\delta]$
for all  $0\le k < T_\ell(x) - S_{\ell-1}(x)$ and
$f^{T_\ell(x)}(x)\in I_j$ with  $j=S_\ell(x)-T_\ell(x)\ge H(\delta)$.
Therefore, \eqref{expii}, combined with \eqref{eq.deltaL}, and \eqref{eq.deltareturn}
  give for all $i\ge 0$
$$
|(f^{S_i})'(x)|= \prod_{\ell=1}^{i} |(f^{S_\ell(x)-T_{\ell}(x)})'(f^{T_\ell(x)}x)| |(f^{T_\ell(x)-S_{\ell-1}(x)})'(f^{S_{\ell-1}(x)}x)|\geq  \rho^{S_i(x)} \, ,
$$
and 
$$
|(f^{T_i})'(x)|= |(f^{T_i(x)-S_{i-1}(x)})'(f^{S_{i-1}(x)}x)|
|(f^{S_i})'(x)| \geq C_0  \rho^{T_i(x)-S_i(x)}\rho^{S_i(x)} \, .
$$
Using in addition \eqref{eq.noreturn}, we get, for $0 \le  j \le T_{i+1}(x)-S_i(x)$, 
$$
|(f^{S_i(x)+j})'(x)|=| (f^j)'(f^{S_i(x)}(x))||(f^{S_i})'(x)|
\geq C_0\delta \rho^j  \rho^{S_i(x)} \, .
$$
\end{proof}

\begin{proof}[Proof of Proposition~\ref{ubalpha}]
This proof is very similar to that of  \cite[Proposition~3.7]{BS}, we give it
for the convenience of the reader.
Fix $j\ge 1$. Since the summands are all positive, we may (and shall) group them
in a convenient way, using the
times $T_i:=T_i(c_{j+1})$  and $S_i:=S_i(c_{j+1})$
for a small enough $\delta$.
We have
\begin{align*}
 &\sum_{k =j+1}^\infty \frac{1}{|(f^{k-j})'(f^j(c_1))|}\\
&\quad =\sum_{i=0}^{\infty} \frac{1}{|(f^{S_i})'(c_{j+1})|} u_{T_{i+1}-S_i}(f^{S_i}(c_{j+1})) +\sum_{i=1}^{\infty} \frac{1}{|(f^{T_i})'(c_{j+1})|} u_{S_{i}-T_i}(f^{T_i}(c_{j+1})) \, ,
\end{align*}
where we use the notation
$
u_n(y)= \sum_{\ell=1}^{n} \frac{1}{|(f^{\ell})'(y)|} 
$.
(In particular $u_0\equiv 0$.)
Since $T_{i+1}-S_i=T_1(f^{S_i}(c_{j+1}))$,  Lemma \ref{expiii} implies 
$$
u_{T_{i+1}-S_i}(f^{S_i}(c_{j+1}))\leq \frac{C }{C_0\delta(1-\rho^{-1})} \, ,
$$
(in particular the series converges if  $n=T_{i+1}-S_i=\infty$).
Since $f''(c)\ne 0$, the polynomial recurrence assumption \eqref{eq.alpha}
implies for all $i$
$$|f'(f^{T_i}(c_{j+1}))|= |f'(f^{T_i+j}(c_1))|\geq  C^{-1} (T_i+j)^{-\alpha}\, .
$$
Therefore, the   bounded distortion estimate \eqref{infprod} in the proof of Lemma ~\ref{cd} gives, together with the Collet--Eckmann assumption,
\footnote{The constant $C$ above depends on
$[\sup_{1\le j < H_0} \lambda_c/| (f^j)'(c_1)|^{1/j}]^{H_0}$.
By Proposition~\ref{p.uniformc}, this expression
is uniform for suitable families $f_t$.}
$$
u_{S_{i}-T_i}(f^{T_i}(c_{j+1}))\leq \frac{1}{|f'(f^{T_i}(c_{j+1}))|} \sum_{\ell=0}^{\infty} \frac{C}{|(f^\ell)'(c_1)|}\leq \frac{C^3 }{ (1-\lambda_c^{-1})} (T_i+j)^{\alpha} \, . 
$$

By Lemma \ref{expiii} we have 
$|(f^{S_i})'(c_{j+1})|\geq \rho^{S_i}$
and
$|(f^{T_i})'(c_{j+1})|\geq C_0\rho^{T_i}$.
Therefore, there exists constants $C_1$, $C_2$ so that
$$ 
 \sum_{k =j+1}^\infty \frac{1}{|(f^{k-j})'(f^j(c_1))|}
\leq C_1\delta^{-1} j^\alpha \biggl [  \sum_{i=0}^{\infty} \rho^{-S_i}   +  \sum_{i=1}^{\infty} \rho^{-T_i}T_i^\alpha    \biggl ]\leq C_2\delta^{-1} j^\alpha \, .
$$
\end{proof}

\section{Spectral properties of the transfer operators $\widehat  \LL_t$}
\label{misc2}

Recall that $f_t$ is assumed good. We prove Proposition~\ref{mainprop}:
\begin{proof}
Let $c(\delta)$  be as in \eqref{eq.noreturn}.
For $\hat \psi \in \BB$, our assumptions on the 
$\xi_{j}$ ensure that   $(\widehat \LL (\hat \psi))_k\in W^1_1$ for all $k\ge 1$,
with $(\widehat \LL (\hat \psi))_k$  supported in the desired interval, 
and that  $(\widehat \LL (\hat \psi))_0$  is  supported in the desired interval.

Note that for any interval $U$ (not necessarily containing
the support of $\psi_j$), using the Sobolev embedding,
\begin{equation}\label{sobeb}
\sup_U |\psi_j|\le \min (C \|\psi'_j\|_{L^1}, \int_{U}|\psi'_j|\, dx + |U|^{-1} \int_{U} |\psi_j|\, dx)\, .
\end{equation}
Since $\xi_\ell$ is unimodal if it is not $\equiv 1$, for each $\ell \ge 1$ there exist
$v_\ell >u_\ell$ in $B_\ell$ so that, setting $J_\ell=\{x \in \supp(\psi_j)\mid x\le u_\ell\}\cup \{x \in \supp(\psi_\ell)\mid
x \ge v_\ell\}$,
\begin{align}\label{dirtytrick}
\int_{B_\ell} |\xi_\ell'\psi_\ell|\, dx&= \int_{x \le u_\ell} \xi_\ell' |\psi_\ell|\, dx
- \int_{x \ge v_\ell} \xi_\ell' |\psi_\ell|\, dx
\le 2 \sup_{J_\ell} |\psi_\ell|
\end{align}
Therefore,
for all $k\ge 1$, using also \eqref{sobeb},
\begin{align}
\label{der0}&
\| (\widehat \LL (\hat \psi))_k'\|_{L^1} \le \frac{3C}{\lambda} \| \psi'_{k-1}\|_{L^1} \, . 
\end{align}
More generally, for $1\le n\le k$,
\begin{align}
\label{der0N}&
\| (\widehat \LL^n (\hat \psi))_k'\|_{L^1} \le \frac{3C n}{\lambda^n} \| \psi'_{k-n}\|_{L^1} \, . 
\end{align}

If $|\psi_k(y)|>0$ then $|f^{j}(y)-c_{j}|\le  j^{-\beta} /L$
for all $j\le k$.
If $\xi_k(f^{-(k+1)}_\pm(x))<1$ then $|x-c_{k+1}|\ge L^{-3} (k+1)^{-\beta}$.
Thus,  changing  variables in the integrals, using \eqref{dirtytrick}
for the terms involving $\xi'_k$ for $k \ge 0$, and
recalling   \eqref{expii}, \eqref{rootsing2}, 
\eqref{gluu} and \eqref{finally} from Lemma~\ref{sizeIj} and its proof, 
we see
that $(\widehat \LL (\hat \psi))_0$ belongs to  $W^1_1$ and
\begin{align}
\label{der00}\| (\widehat \LL (\hat \psi))'_0\|_{L^1} 
&\le   Cc(\delta)^{-1} (\|\psi_0'\|_{L^1}+  \|\psi_0\|_{L^1}+ \sup|\psi_0|)\\
\nonumber &\qquad + C
\sum_{k\ge H(\delta)}   \frac{\lambda^k k^{1+2\alpha+\beta/2} }{|(f^{k+1})'(c_1)|^{1/2}}
(\|\psi'_k\|_{L^1} + \sup |\psi_k| + \|\psi_k\|_{L^1} )\, .
\end{align}
In view  of \eqref{48} and \eqref{sobeb}, we have proved that $\widehat \LL$ is bounded
on $\BB$.  

\smallskip
Observe that $\nu(\hat\psi)=\sum_k \int_{B_k}|\psi_k(x) |w(x,k)\, dx$
is finite if $\hat \psi \in \BB$: Indeed,
$\psi_k$ is supported in $J_k$ which decays exponentially
according to  \eqref{decaybound}, 
and we may use the bound $\lambda  < \lambda_c^{1/2}$
from \eqref{48}. So $\nu$ is an element of the dual
of $\BB$.
The fact that  $\widehat \LL^*(\nu)=\nu$  can easily be  proved using the change of variables
formula (see \cite[(85)]{BS}).
Note for further use that $\widehat \LL^* (\nu)=\nu$ implies
\begin{equation}\label{l1bd}
\nu(|\widehat \LL^N(\hat \psi)|)\le \nu(\widehat \LL^N( |\hat \psi|))=\nu(|\hat \psi|)\, .
\end{equation}
Furthermore, note that $\nu(|\hat\psi|)=\|\hat\psi\|_{\BB^{L^1}}$.
\smallskip

We next estimate the spectral and essential spectral radii
of  $\widehat \LL$ on $\BB$.
Using \eqref{dirtytrick} and the overlap control of fuzzy intervals
(see Remark~\ref{overlap}), it is  not very difficult 
(this may be done exactly like in Appendix~B of \cite{BS}, since our
overlap control is in fact much better) 
to adapt the proof of \cite[Sublemma in Section~4]{BV}
to show inductively that for any $\Theta<\Theta_0$,
there exists $C$, and for all $n$ there exists $C(n)$, so that 
\begin{equation}
\|(\widehat \LL^n( \hat \psi))'_0(x)\|_{L^1(I)} \le
C \Theta^{-n}
\|\hat \psi\|_{\BB} + C(n) \|\hat \psi\|_{\BB^{L^1}}\, .
\end{equation}
Recalling \eqref{der0N}, and using \eqref{l1bd}, up to slightly decreasing
$\Theta$, one then finds
$C'$ so that for all $n\ge 1$
(see the proof of \cite[Variation Lemma in Section~4]{BV})
\begin{equation}\label{keyLM}
\|\widehat \LL^n( \hat \psi)\|_{\BB} \le
C' \Theta^{-n}
\|\hat \psi\|_{\BB} + C' \|\hat \psi\|_{\BB^{L^1}}
\, .
\end{equation}
Since \eqref{defban}, \eqref{decaybound}, and \eqref{48} imply that the length
of the support of $\psi_k$ is (much) smaller than 
$\lambda^{-2k}$, we find $\|\hat \psi\|_{\BB^{L^1}}\le   \|\hat \psi\|_{\BB^{L^1}}
+C \lambda^{-M}\|\hat \psi\|_{\BB^{W^1_1}}$
for all $M\ge 1$ (we used again the Sobolev embedding to estimate
the supremum by the $W^1_1$ norm).

Finally,  since
Rellich--Kondrachov implies that $\BB^{W^1_1}$ 
is compactly included in $\BB^{L^1}$ (the total length 
is bounded, even up to $\lambda^k$-expansion at
kevel $k$, by \eqref{decaybound}), the Lasota--Yorke estimate \eqref{keyLM} together with
Hennion's theorem \cite{Hen}  give the claimed bound on 
essential spectral radius of $\widehat \LL$ on $\BB=\BB^{W^1_1}$.
This ends the proof
of the claims on the  essential
spectral radius in Proposition ~ \ref{mainprop}.

We now describe the eigenvalues of modulus $1$:

The bound \eqref{keyLM} implies that the spectral radius of $\widehat \LL$ on $\BB$
is at most one, and thus equal to one
(since the essential spectral radius is strictly smaller than one
while
$1$ is an eigenvalue of the dual operator). 
Since $\widehat \LL_t$ is a nonnegative operator
with spectral radius equal to $1$ and essential spectral radius strictly smaller
than $1$, classical results of Karlin  \cite[pp. 933-935, Thm~27]{karlin} (using 
for $\KK$ the lattice of continuous
functions, with $\hat \psi_1 \ge \hat \psi_2$ if $\psi_{1,k}(x) \ge \psi_{2,k}(x)$
for all $x$ and $k$), imply that its eigenvalues of modulus one are a
finite set of roots of unity $e^{2\imath j \pi/P}$, $j=0,\ldots,P-1$ for some $P\ge 1$.
In addition, the eigenfunctions for the eigenvalue $1$ are nonnegative.
If $\hat \phi$ is a (nonnegative) fixed point
normalised by $\nu(\hat \phi)=1$, we have $\int_I \phi\, dx=\int_I \Pi(\hat \phi)\, dx=1$.
Recalling \eqref{commute}, we get that $\LL(\Pi(\hat \phi))= \phi$,
so that  $\Pi(\hat \phi)\in L^1$ is indeed the invariant density
of $f$. Since this density is  unique and ergodic, the eigenvalue at $1$ is simple,
and therefore also the  eigenvalues of modulus $1$ are also simple. 
If $f_t$ is mixing, then $f_t^j$ is ergodic for all $j\ge 1$, so the only
eigenvalue of modulus one is $1$.
Otherwise, let $P_t\ge 2$
be the renormalisation period and let $\phi_{j,t}$ and $\nu_{j,t}$ be  the eigenvectors
of $\widehat \LL_t$ and $\widehat \LL_t^*$, respectively for $e^{2\imath j\pi/P_t}$.
It is not difficult to
check that $\max_j \|\hat \phi_{j,t}\|_{\BB^{W^1_1}}$
and $\max_j \|\nu_{j,t}\|_{(\BB^{L^1})^*}$ are bounded uniformly
in $t$ (using for example the special structure of maximal eigenvectors
of a positive operator \cite[p. 933--934]{karlin}). 

It only remains to show that $\hat \phi_0 \in W^2_1$
if $f$ is $C^4$.  For this, take $\hat\psi$ so that $\psi_k=0$
for all $k\ge 1$ and $\psi_0$ is $C^\infty$, of Lebesgue
average $1$ (we can even take $\psi_0$ constant
in a neighbourhood of $[c_2,c_1]$), and use
that  \begin{equation}
\label{rrref}
\frac{1}{k} \sum_{n=0}^{k-1}\widehat \LL^{n}(\hat \psi)
\end{equation}
 converges to $\hat \phi$
in the $\BB^{W^1_1}$ norm as $n\to \infty$. 
We claim that
$\|(\widehat \LL^n(\hat \psi))_0\|_{W^3_1}\le C$ for all $n$.
Adapting the proof of \eqref{rootsing3}, one shows
$$\sup_{x \in f^k(I_k)}\biggl |
\partial^3_x \frac{1}{|(f^k)'(f^{-k}_\pm(x))|}
\biggr | \le   C  \frac{k^{8 \alpha}}{|(f^{k-1})'(c_1)|^{1/2}}$$ 
if
$\alpha >1$ and $\beta <2\alpha$.
Then, in view of \eqref{eq.alpha}, one
can exploit (in addition to the properties already used in
the proof of the Lasota--Yorke estimates
for the $W^1_1$ norm in Proposition~\ref{mainprop})
the properties \eqref{kappa'} of $\xi''_k$, $\xi'''_k$
to adapt (152) in \cite[Appendix B]{BS} (noting also
that $\hat\psi|_\omega=0$
if the interval $\omega $ is in some level $E_k$ with $k>0$). 
Note that \eqref{dirtytrick} is not needed, since we only look at
the component of $\widehat \LL^n( \hat \psi)$ at level $0$.
 Details are straightforward and left to the reader.
To conclude, use that if a sequence  converging
to $\hat \phi_0$ in $W^1_1(I)$ has bounded $W^3_1(I)$ norms, 
then $\hat \phi_0\in W^2_1(I)$ by Rellich--Kondrachov
(using \eqref{decaybound} again).
\end{proof}


\end{appendix}

\bibliographystyle{plain}

\bibliography{transv}

\end{document}